\documentclass[12pt,a4paper]{article}

\usepackage[american]{babel}
\usepackage[T1]{fontenc}
\usepackage[utf8]{inputenc}
\usepackage{csquotes}
\usepackage[backend=biber, maxbibnames=10, maxcitenames=4, maxalphanames=4, bibencoding=utf8, style=alphabetic, isbn=false, url=false, doi=true, giveninits=true]{biblatex}

\renewbibmacro{in:}{}
\renewbibmacro*{doi+eprint+url}{%
  \printfield{doi}%
  \newunit\newblock%
  \iftoggle{bbx:eprint}{%
    \usebibmacro{eprint}%
  }{}%
  \newunit\newblock%
  \iffieldundef{doi}{%
  \usebibmacro{url+urldate}}%
  {}%
}
\AtEveryBibitem{\clearfield{month}}
\AtEveryBibitem{\clearfield{note}}
\usepackage{setspace}
\usepackage{lmodern}
\usepackage[indentafter]{titlesec}
\titleformat{name=\section}{}{\thesection.}{0.8em}{\centering\scshape}
\titleformat{name=\subsection}[runin]{}{\thesubsection.}{0.5em}{\bfseries}[.]
\titleformat{name=\subsubsection}[runin]{}{\thesubsubsection.}{0.5em}{\itshape}[.]
\titleformat{name=\paragraph,numberless}[runin]{}{}{0em}{}[.]
\titlespacing{\paragraph}{0em}{0em}{0.5em}
\titleformat{name=\subparagraph,numberless}[runin]{}{}{0em}{}[.]
\titlespacing{\subparagraph}{0em}{0em}{0.5em}
\usepackage{url}
\usepackage{mathtools}
\DeclarePairedDelimiter\abs{\lvert}{\rvert}%
\DeclarePairedDelimiter\norm{\lVert}{\rVert}%
\makeatletter
\let\oldabs\abs
\def\abs{\@ifstar{\oldabs}{\oldabs*}}
\let\oldnorm\norm
\def\norm{\@ifstar{\oldnorm}{\oldnorm*}}
\makeatother
\usepackage{amssymb}
\usepackage{amsthm}
\usepackage{mathrsfs}
\usepackage{dsfont}
\usepackage{enumitem}
\setlist[enumerate]{noitemsep, partopsep=0pt, topsep=0pt, parsep=0pt, itemsep=0pt}
\setlist[itemize]{noitemsep, partopsep=0pt, topsep=0pt, parsep=0pt, itemsep=0pt}
\usepackage{interval}
\intervalconfig{soft open fences}
\usepackage[normalem]{ulem}
\usepackage[stretch=10]{microtype}
\usepackage[affil-it, noblocks]{authblk}

\usepackage{xcolor}
\usepackage[hidelinks,draft=false]{hyperref}
\hypersetup{
  colorlinks   = true, 
  urlcolor     = blue, 
  linkcolor    = blue, 
  citecolor   = blue 
}
\usepackage[nameinlink,noabbrev,capitalize]{cleveref}
\crefname{equation}{}{}
\usepackage[plain]{fullpage}
\newlist{theoenum}{enumerate}{1} 
\setlist[theoenum]{label=\normalfont(\roman*), ref=\theproposition~\normalfont(\roman*), noitemsep, partopsep=0pt, topsep=0pt, parsep=0pt, itemsep=0pt}
\crefalias{theoenumi}{theorem}
\usepackage{tocloft}

\usepackage{titlefoot}
\usepackage{commath}
\usepackage{dirtytalk}
\usepackage{upgreek}
\usepackage{pgfplots}
\usepgfplotslibrary{fillbetween}
\pgfplotsset{compat=newest}
\usepackage{float}        

\theoremstyle{plain}
\newtheorem{theorem}{Theorem}[section]
\newtheorem*{theorem*}{Theorem}
\newtheorem{proposition}[theorem]{Proposition}
\newtheorem{corollary}[theorem]{Corollary}

\newtheorem{lemma}[theorem]{Lemma}
\newtheorem*{lemma*}{Lemma}

\makeatletter
\def\cref@thmoptarg[#1]#2#3#4{%
  \ifhmode\unskip\unskip\par\fi%
  \normalfont%
  \trivlist%
  \let\thmheadnl\relax%
  \let\thm@swap\@gobble%
  \thm@notefont{\fontseries\mddefault\upshape}%
  \thm@headpunct{.}
  \thm@headsep 5\p@ plus\p@ minus\p@\relax%
  \thm@space@setup%
  #2
  \@topsep \thm@preskip               
  \@topsepadd \thm@postskip           
  \def\@tempa{#3}\ifx\@empty\@tempa%
  \def\@tempa{\@oparg{\@begintheorem{#4}{}}[]}%
  \else%
  \refstepcounter[#1]{#3}
  \@namedef{cref@#3@alias}{#1}
  \def\@tempa{\@oparg{\@begintheorem{#4}{\csname the#3\endcsname}}[]}%
  \fi%
\@tempa}%
\makeatother

\crefname{theorem}{Theorem}{Theorems}
\crefname{proposition}{Proposition}{Propositions}

\theoremstyle{definition}
\newtheorem{definition}[theorem]{Definition}

\theoremstyle{remark}

\newtheorem{remark}[theorem]{Remark}
\newtheorem*{remark*}{Remark}      

\newcommand*\diff{\mathop{}\!\mathrm{d}}
\newcommand{\R}{\mathbb{R}}

\newcommand{\heis}{\mathbb{H}}
\newcommand{\Min}{\mathbb{M}}
\newcommand{\N}{\mathbb{N}}

\newcommand{\tmcp}[2]{\mathsf{TMCP^e}(#1,#2)}
\newcommand{\pder}[2]{\frac{\partial #1}{\partial #2}}

\newcommand{\dis}{\mathsf{d}}
\newcommand{\m}{\mathfrak{m}}
\newcommand{\TCD}{\mathsf{TCD}}
\newcommand{\Ent}{\mathsf{Ent}}

\newcommand{\Prob}{\mathscr{P}}

\newcommand{\prehausdorff}[1]{\mathcal{H}_{\uptau,#1}}
\newcommand{\hausdorff}{\mathcal{H}_{\uptau}}
\renewcommand{\L}{\mathrm{L}}
\renewcommand{\exp}{\mathrm{exp}}
\renewcommand{\epsilon}{\varepsilon}
\renewcommand{\subset}{\subseteq}
\newcommand{\mres}{\mathbin{\vrule height 1.6ex depth 0pt width 0.13ex\vrule height 0.13ex depth 0pt width 1.3ex}}

\DeclareMathOperator{\T}{\mathrm{T}}

\DeclareMathOperator{\sgn}{sgn}
\DeclareMathOperator{\Lip}{Lip}
\DeclareMathOperator{\arcsinh}{arcsinh}
\DeclareMathOperator{\supp}{supp}

\let\originalleft\left
\let\originalright\right
\renewcommand{\left}{\mathopen{}\mathclose\bgroup\originalleft}
\renewcommand{\right}{\aftergroup\egroup\originalright}

\usepackage{footmisc}

\title{%
  \begingroup
  \renewcommand\footnotemark{}
  \thanks{Date: \today}%
  \endgroup
  \vspace{-2em}\textbf{\uppercase{\large Hausdorff dimension and failure of synthetic curvature bounds in the sub-Lorentzian Heisenberg group}}
}

\author{%
  Samu\"el Borza%
  \thanks{Faculty of Mathematics, University of Vienna, Oskar-Morgenstern-Platz 1, 1090 Vienna, Austria}
  \protect\footnotemark[2]%
  \and
  Chiara Rigoni%
  \protect\footnotemark[1]%
  \protect\footnotemark[3]%
  \and
  Omar Zoghlami%
  \protect\footnotemark[1]%
  \protect\footnotemark[4]%
}

\date{}          

\begin{document}

\maketitle              

\begingroup
\makeatletter
\renewcommand{\thefootnote}{\fnsymbol{footnote}}
\let\orig@makefnmark\@makefnmark   
\def\@makefnmark{}                 
\footnotetext[2]{\textit{E-mails}: %
  \orig@makefnmark\href{mailto:samuel.borza@univie.ac.at}{samuel.borza@univie.ac.at};\,
  {\let\@makefnmark\orig@makefnmark \footnotemark[3]}%
  \href{mailto:chiara.rigoni@univie.ac.at}{chiara.rigoni@univie.ac.at};\,
  {\let\@makefnmark\orig@makefnmark \footnotemark[4]}%
\href{mailto:omar.zoghlami@univie.ac.at}{omar.zoghlami@univie.ac.at}}
\makeatother
\endgroup

\providecommand{\keywords}[1]
{
  \textbf{\textit{Keywords---}} #1
}

\providecommand{\msc}[1]
{
  \textbf{\textit{MSC (2020)---}} #1
}

\vspace{-1em}
\begin{abstract}
  We study the geodesics, Hausdorff dimension, and curvature bounds of the sub-Lorentzian Heisenberg group. Through an elementary variational approach, we provide a new proof of the structure of its maximizing geodesics, showing that they are lifts of hyperbolae coming from a Lorentzian isoperimetric problem in the Minkowski plane. We prove that the Lorentzian Hausdorff dimension of the space is $4$ and that the corresponding measure coincides with the Haar measure. We further establish a novel result in the spirit of the Ball-Box theorem, giving a uniform estimate of causal diamonds by anisotropic boxes. Finally, we show that the Heisenberg group satisfies neither the timelike curvature-dimension condition $\mathsf{TCD}(K,N)$ nor the timelike measure contraction property $\mathsf{TMCP}(K,N)$ for any values of the parameters $K$ and $N$, in sharp contrast with its sub-Riemannian counterpart.
\end{abstract}


\msc{53C50, 53C17, 49Q22, 28A75, 51K10}        
{\renewcommand{\contentsname}{\large Contents}%
  \small
  \tableofcontents
}
\section{Introduction}

While sub-Riemannian geometry has developed into a large and active field of research, with deep connections to metric geometry \cite{Agrachev2020}, \emph{sub-Lorentzian geometry}--the non-holonomic analogue of Lorentzian geometry--has received little attention. It was already anticipated as a Lorentzian analogue of sub-Riemannian geometry in Strichartz’s early monograph \cite{Strichartz1986}, but the field only began to be systematically studied with the seminal contributions of Grochowski in the early 2000s \cite{Grochowski2002,Grochowski2009}.

In recent years, left-invariant Lorentzian and sub-Lorentzian problems have attracted increasing attention within the framework of geometric control theory. Building on Grochowski’s pioneering work on the sub-Lorentzian geometry of the Heisenberg group, this line of research was further developed by Sachkov and Sachkova in \cite{SaSa1, sachkovsachkova23}. Significant progress has also been made by {Grong} and {Vasil'ev} in \cite{ErlVas}, who studied a left-invariant sub-Lorentzian Lie group isometric to anti-de Sitter space, as well as by Sachkov in \cite{Sa1, Sa2}, who analyzed left-invariant Lorentzian geometry on the Lobachevsky plane. Subsequent studies have extended the theory to other sub-Lorentzian settings. Notable examples include Sachkov’s work on the Martinet distribution \cite{SachMart}, the joint research of Cai, Huang, Sachkov, and Yang on sub-Lorentzian geodesics in the Engel group \cite{CHSY}, and Karmanova’s contribution on the coarea formula for mappings in Carnot groups equipped with a sub-Lorentzian metric \cite{Karmanova}.

We emphasize that all these significant advances in the study of sub-Lorentzian structures have been achieved through the use of optimal \emph{control} theory. However, it is worth noting that in the sub-Riemannian setting, major developments in understanding the geometric properties of sub-Riemannian manifolds have also been obtained by employing tools from optimal \emph{transport} theory. This suggests that similar techniques could offer valuable insights in the sub-Lorentzian context, an area that, to date, remains largely unexplored. The only known work in this direction is the recent preprint \cite{BKW} by the first author, {Klingenberg} and {Wood}, where the authors investigate the synthetic metric spacetime structure of the sub-Lorentzian Heisenberg group and study the associated optimal transport problem. In particular, they establish a sub-Lorentzian version of Brenier's theorem and derive a Monge-Amp\'ere-type equation.

The \emph{Heisenberg group}, denoted by $\mathbb{H}$, is the prototypical example of a geometric structure with constrained dynamics. It arises from restricting motion in $\mathbb{R}^3$ so that velocities are tangent only to a two-dimensional non-integrable \emph{horizontal distribution}. By fixing an inner product on this distribution, the sub-Riemannian, or \emph{Carnot-Carathéodory}, distance $\dis$ between two points is defined as the \emph{infimum} of the lengths of these \emph{horizontal} curves connecting them, see \cite{Capogna2007,Agrachev2020}. The \emph{sub-Lorentzian time-separation} $\uptau$ is defined analogously, but with respect to a Lorentzian metric on the horizontal distribution, as the \emph{supremum} of the proper times along \emph{causal} horizontal curves connecting two events.

The search for maximizing curves in the Heisenberg group, i.e., sub-Lorentzian geodesics, was investigated in \cite{Grochowski2004,Grochowski2006,sachkovsachkova23} using Pontryagin's maximum principle from optimal control theory. As a first contribution, this paper establishes a new proof of the shape of Heisenberg’s sub-Lorentzian geodesics, relying exclusively on elementary variational principles.

It is a well-known fact that horizontal curves in the Heisenberg group are lift of curves in the plane satisfying an area constraint. In the sub-Riemannian case, finding a length minimizer is equivalent to solving the Euclidean isoperimetric problem in the plane, an approach pursued in \cite{Capogna2007,LeDonneBook}, for instance. We show that the length-maximization problem in the sub-Lorentzian Heisenberg group is equivalent to a Lorentzian isoperimetric problem in the Minkowski plane, which we solve by applying direct methods from the calculus of variations. In particular, we prove the following result. Here $\mathbb{M}$ denotes the Minkowski plane with metric \(-\diff x^2 +\diff y^2\); for a curve $\gamma$, $A(\gamma)$ is the signed area enclosed by $\gamma$ and the straight line segment joining its endpoints in $\mathbb{M}$, while $L(\gamma)$ denotes the Lorentzian length of
$\gamma$.

\begin{theorem}\label{maintheorem1:geodesicssubLHeisenberg}
  Let $(a,b) \in \mathbb{R}^2$ and $c \in \mathbb{R}$.
  There exists a future-directed causal curve $\gamma : [0,1] \to \Min$ with
  \[
    \gamma(0) = \mathbf{0}, \quad \gamma(1) = (a,b), \quad A(\gamma)=c
  \]
  if and only if
  \[
    -a^2 + b^2 + 4|c| \leq 0 \quad \text{and} \quad a > 0.
  \]
  In this case, there exists a unique curve (up to reparametrisation) that maximizes the Lorentzian length among all such curves:
  \begin{enumerate}[label=\normalfont(\roman*), topsep=4pt,itemsep=4pt,partopsep=4pt,parsep=4pt]
    \item If $c=0$, the maximizer is the straight line from $\mathbf{0}$ to $(a,b)$ and it is timelike;
    \item If $c \neq 0$ and $-a^2+b^2+4|c|=0$, the maximizer is a broken line with a single breakpoint from $\mathbf{0}$ to $(a,b)$, and this curve is null;
    \item If $c \neq 0$ and $-a^2+b^2+4|c|<0$, the maximizer is an arc of a hyperbola from $\mathbf{0}$ to $(a,b)$, and this curve is timelike.
  \end{enumerate}
\end{theorem}

Therefore, the geodesics in the sub-Lorentzian Heisenberg group are obtained by lifting these isoperimetric curves in the Minkowski plane: their projection is the given curve in $\Min$, while its $z$-coordinate measures the area enclosed by it.

The synthetic approach to Lorentzian geometry, introduced in \cite{kunzingersaemann2018}, is a recent theory that has emerged as a central framework for studying low-regularity spacetimes and singularities in general relativity. The notion of a \emph{Lorentzian length space} plays a role analogous to that of \emph{length space} in metric geometry. Instead of a metric space, one considers a spacetime as a set of events $X$ equipped with a time-separation function $\uptau : X \times X \to [0, +\infty]$, satisfying the reverse triangle inequality, which represents the maximal elapsed ``clock time'' an observer can experience when traveling from one event to another. A notion of \emph{Lorentzian Hausdorff measure and dimension} has been introduced in \cite{McCann2022}, and our second result computes them in the sub-Lorentzian Heisenberg group.

\begin{theorem}
  The Lorentzian Hausdorff dimension of the sub-Lorentzian Heisenberg group is 4, and the corresponding Hausdorff measure coincides, up to a positive constant, with the Haar measure of the Heisenberg group.
\end{theorem}

For comparison, the sub-Riemannian Heisenberg group also has \emph{metric} Hausdorff dimension 4, and its 4-dimensional Hausdorff measure agrees with the Haar measure. For those acquainted with sub-Riemannian geometry, it is well known that in the sub-Riemannian case this result follows from the \emph{Ball-Box} theorem \cite{mitchellSR,Bellaiche1996}, which states that metric balls of radius $r > 0$ centred at the origin $\mathbf{0}$ satisfy
\[
  \mathsf{Box}(c_1 r) \subseteq B(\mathbf{0}, r) \subseteq \mathsf{Box}(c_2 r),
\]
where $\mathsf{Box}(r) := [-r, r] \times [-r, r] \times [-r^2, r^2]$, for universal constants $c_1, c_2 > 0$ and all $r > 0$. In other words, it holds that
\begin{equation}
  \label{eq:distanceboxHeisenberg}
  c_1 \sqrt{x^2 + y^2 + |z|} \leq \mathsf{d}(\mathbf{0}, (x, y, z)) \leq c_2 \sqrt{x^2 + y^2 + |z|}, \qquad \text{ for all } (x, y, z) \in \mathbb{H}.
\end{equation}
It was shown in \cite{Grochowski2006} that naively replacing $\mathsf{d}$ with Heisenberg's time-separation function $\uptau$ in \cref{eq:distanceboxHeisenberg} and the content of the square roots by the Lorentzian analogue $x^2 - y^2 - 4 |z|$ does not hold. More precisely, the lower bound in \cref{eq:distanceboxHeisenberg} holds with $c_1 = 1$, whereas there is no $c_2$ such that the upper bound can be satisfied. Instead, we prove in this work a \emph{Diamond-Box theorem}, the first of its kind, estimating uniformly causal diamond with the anisotropic boxes $\mathsf{Box}(r)$.

\begin{theorem}
  There exists a constant $C_1 > 0$ such that for all $p, q \in \mathbb{H}$, it holds that
  \[
    J(p, q) \subseteq \mathsf{Box}(C_1 \cdot \dis(p, q)).
  \]
  Moreover, there are constants $C_2, D > 0$ such that for every $r > 0$, there exist $p, q \in \mathbb{H}$ with
  \[
    \mathsf{Box}(C_2 r) \subseteq J(p,q),
    \quad \uptau(p,q) = 2Dr,
    \quad \text{and} \quad
    \mathrm{diam}_{\dis}(J(p,q)) \leq 4C_1Dr.
  \]
\end{theorem}

One motivation for developing a Lorentzian metric spacetime theory is to provide a framework for spacetimes of low regularity underlying Einstein's general relativity. In this setting, notions of curvature and curvature bounds play a central role, for instance in the Penrose--Hawking singularity theorems. This has led to the introduction of \emph{timelike} curvature-dimension conditions defined via optimal transport. They are designed to capture the condition $\mathrm{Ric} \geq K$ in all \emph{timelike} directions in a synthetic way, that is, using only the time-separation function, the causal structure, and a fixed background measure.

Curvature bounds via optimal transport were first introduced in the (possibly non-smooth) Riemannian framework
with the curvature-dimension condition $\mathsf{CD}$ \cite{lottvillani,S-ActaI,S-ActaII} and the measure contraction property $\mathsf{MCP}$ \cite{ohta,S-ActaII}. The corresponding \emph{timelike} curvature-dimension conditions $\mathsf{TCD}$ and \emph{timelike} measure contraction properties $\mathsf{TMCP}$ have recently been developed in the synthetic Lorentzian setting in the recent works \cite{cavallettimondino2024,Braun2023a}.

As the third and final aim of this paper, we investigate the validity of these conditions in the sub-Lorentzian Heisenberg group. It is known that sub-Riemannian manifolds fail the $\mathsf{CD}$ condition, as shown for example in \cite{juillet2021,Rizzi2023,universalnoCD}, while the $\mathsf{MCP}$ may either hold or fail, see \cite{borza2025}. Notably, \cite{juillet2009} showed that the sub-Riemannian Heisenberg group satisfies $\mathsf{MCP}(K,N)$ if and only if $K \leq 0$ and $N \geq 5$. Our final result establishes the surprising fact that in the Heisenberg group not only the $\mathsf{TCD}$ but also the $\mathsf{TMCP}$ fails.

\begin{theorem}
  For every $K \in \mathbb{R}$ and $N \geq 1$, the sub-Lorentzian Heisenberg group satisfies neither $\mathsf{TCD}(K,N)$ (with $N \leq +\infty$) nor $\mathsf{TMCP}(K,N)$ (with $N < +\infty$).
\end{theorem}

The paper is organized as follows. In \cref{section:subLgeometry}, we introduce the Heisenberg group together with its sub-Lorentzian structure. In \cref{subsection:isopiffgeod}, we show that length-maximizing curves can be characterized as solutions of an isoperimetric problem in the Minkowski plane, which is then solved in \cref{subsection:variationalisopMink}. In \cref{section:exponentialmap}, we recover the explicit form of the sub-Lorentzian exponential map of the Heisenberg group. Next, in \cref{section:lorentzianmetricspacetime,section:TCDTMCP}, we recall definitions, properties, and known results on Lorentzian length spaces and timelike curvature bounds via optimal transport. In \cref{section:compatibility}, we establish the compatibility between the sub-Lorentzian and synthetic structures. \cref{section:hausdorffdimension} is devoted to the Lorentzian Hausdorff measure and dimension, while \cref{sect:heis_tcd_tmcp} deals with the failure of synthetic curvature notions in the sub-Lorentzian Heisenberg group.

\section*{Acknowledgements}

We thank Clemens Sämann and Luca Rizzi for several helpful discussions. This research was funded by the Austrian Science Fund (FWF) [Grant \href{https://www.fwf.ac.at/en/research-radar/10.55776/EFP6}{DOI 10.55776/EFP6}] and C.R.'s research was funded by FWF [Grant \href{https://www.fwf.ac.at/en/research-radar/10.55776/ESP224}{DOI:10.55776/ESP224}]. For open access purposes, the authors have applied a CC BY public copyright license to any author accepted
manuscript version arising from this submission. Moreover, part of this research was carried out by C.R. at the Hausdorff Institute
of Mathematics in Bonn, during the trimester program “Metric Analysis”. The author wish to express
her appreciation to the institution for the stimulating atmosphere and the excellent working conditions.

\section{The sub-Lorentzian Heisenberg group}

\subsection{Sub-Lorentzian geometry}
\label{section:subLgeometry}

In this section, we introduce the Heisenberg group, together with its sub-Lorentzian structure and some of its standard properties.

For the sake of this work, it is enough to view the Heisenberg group $\heis$ as the Lie group consisting of the Euclidean space $\R^3$ equipped with the non-abelian group law:
\[
  (x,y,z) \ast (x',y',z') \coloneqq \left( x+x',y + y', z + z' +\frac{1}{2}(x y'- x' y) \right),
\]
for all $(x, y, z), (x', y', z') \in \R^3$. Clearly, the identity element is the origin $e = \mathbf{0}= (0, 0, 0)$ of $\R^3$ while $(x, y, z)^{-1} = (-x, -y, -z)$.

The left translation by $p \in \heis$ is the automorphism $L_p : \heis \to \heis$ given by
\[
  L_p(x, y, z) \coloneqq p \ast (x, y, z), \qquad \forall\, (x, y, z) \in \heis,
\]
and a vector field $V$ of $\heis$ is said to be \emph{left-invariant} if $\diff_q L_p [V(q)] = V(p \ast q)$ for all $p, q \in \heis$. Recall that a left-invariant vector field is uniquely determined by its value at the group identity element. The left-invariant vector fields $X, Y, Z$ that coincide at the origin $e$ with $\pder{}{x}, \pder{}{y}, \pder{}{z}$, respectively, are
\[
  X = \pder{}{x} - \frac{y}{2} \pder{}{z},\ Y = \pder{}{y} + \frac{x}{2} \pder{}{z}, \text{ and } Z = \pder{}{z}.
\]
The Heisenberg group also admits the {\em dilations} $\delta_\lambda(x, y, z) \coloneqq (\lambda x, \lambda y, \lambda^2 z)$ for all $\lambda \in \R$, which are also automorphism.

The sub-bundle $\Delta$ of the tangent bundle $\T(\heis)$ of $\heis$, defined as
\[
  \Delta \coloneqq \operatorname{span} \left\{ X,Y \right\} \subseteq \T(\heis),
\]
is referred to as the {\em horizontal distribution} of $\heis$, and is left-invariant since $X$ and $Y$ are as well. At any point $p \in \heis$, the vectors $v \in \Delta_p$ are called {\em horizontal vectors}.
Importantly, the distribution $\Delta$ is {\em not} involutive; in fact, we have
\[
  [X,Y] = \pder{}{z} = Z \not \in \Delta.
\]

A curve $\gamma : I \to \heis$ defined on some interval $I \subseteq \R$ is said to be {\emph horizontal} if $\gamma$ is absolutely continuous and if there exists $u, v \in L^{\infty}(I, \R)$ such that
\begin{equation}
  \label{eq:horizontalcurveinH}
  \dot \gamma (t) = u(t) X(\gamma(t)) + v(t) Y(\gamma(t)), \qquad \text{ for almost every } t \in I.
\end{equation}
For any $p \in \heis$ and any horizontal curve $\gamma$, the translated curve $p \ast \gamma \coloneqq L_p \circ \gamma$ is also horizontal since $\Delta$ is left-invariant and left translations are diffeomorphisms.

The sub-Lorentzian structure on $\heis$ is induced by considering the Lorentzian metric $g$ on $\Delta$ uniquely determined by the conditions
\[
  g(X,X)=-1, \quad g(X,Y) = 0, \quad g(Y,Y) = 1.
\]
The structure $(\heis, \Delta, g)$ is known as the {\em sub-Lorentzian Heisenberg group}.
\begin{remark}
  The sub-Lorentzian structure is preserved under left translations. Indeed, for any $p \in \heis$, since $X$ and $Y$ are left-invariant and form a Lorentzian basis for the metric $g$, consider any pair of tangent vectors $v, w \in \Delta_q$ at a point $q \in \heis$, along with their Lorentzian product with respect to $g_q$. The differential $\diff_q L_p$ maps the pair $(v, w)$ to a pair of vectors in $\Delta_{L_p(q)}$ that preserves this product. This implies, in particular, that $\diff_q L_p$ defines an isometry between the spaces $(\Delta_q, g_q)$ and $(\Delta_{L_p(q)}, g_{L_p(q)})$.
\end{remark}

For any $p \in \heis$, we say that an horizontal vector $v \in \Delta_p$ is
\[
  \begin{cases}
    &\text{causal} \\
    &\text{timelike} \\
    &\text{null} \\
    &\text{spacelike}
  \end{cases}
  \quad \text{ if } \quad g_p(v, v) \quad
  \begin{cases}
    &\leq 0 \text{ and } v \neq 0 \\
    &< 0 \\
    &= 0 \text{ and } v \neq 0 \\
    &> 0 \text{ or } v = 0
  \end{cases},
\]
and is \emph{future-directed} if $g(v, X) < 0$.

Accordingly, we say that a horizontal curve $\gamma$ is {\em causal} (resp. {\em timelike}, {\em null}, {\em future-directed}) if the horizontal vector $\dot{\gamma}(t)$ is {\em causal} (resp. {\em timelike}, {\em null}, {\em future-directed}) for almost every $t$. For any two points $p,q \in \heis$, we say that $p$ {\em causally precedes} $q$, denoted by $p \leq q$, if \(p=q\) or there exists a future-directed causal curve $\gamma$ joining $p$ to $q$\footnote{Note that $q \leq q$ for any $q \in \mathbb{H}$, so that $q \in J^+(q)$ by definition.}. Analogously, we say that $p$ {\em chronologically precedes} $q$, denoted by $p \ll q$, if there exists a future-directed timelike curve joining $p$ to $q$. The {\em causal} and the {\em chronological future} of $A \subset \mathbb{H}$ are defined, respectively, as
\begin{equation}
  \label{eq:causalchronologicalfuturedefinition}
  J^+(A) \coloneqq \{ y \in \mathbb{H} \mid \exists x \in A,\ x \leq y \}, \ \text{ and } \ I^+(A) \coloneqq \{ y \in \mathbb{H} \mid \exists x \in A,\ x \ll y \}.
\end{equation}
Similarly, one can define the causal and chronological past $J^-(A)$ and $I^-(A)$. The {\em causal} and {\em chronological diamonds} of two sets $A, B \subseteq \mathbb{H}$ are given by
\begin{equation}
  \label{eq:causalchronologicaldiamonddefinition}
  J(A, B) \coloneqq J^+(A) \cap J^-(B), \ \text{ and } \ I(A, B) \coloneqq I^+(A) \cap I^-(B).
\end{equation}
We will also use the notation
\begin{equation}
  \label{eq:notationA2}
  A^2_{\leq} \coloneqq \{(x, y) \in A \times A \mid x \leq y \} , \ \text{ and } \ A^2_{\ll} \coloneqq \{(x, y) \in A \times A \mid x \ll y \}.
\end{equation}

The {\em Lorentzian length} of a future-directed causal curve $\gamma : I \to \heis$ is given by
\[
  L(\gamma) \coloneqq \int_I \sqrt{-g(\dot{\gamma}(t),\dot{\gamma}(t))} \, \diff t.
\]

This notion of length allows for the definition of the {\em time-separation function} $\uptau$ between two points $p,q \in \heis$:
\[
  \uptau(p,q) \coloneqq \sup \{ L(\gamma) \ | \ \gamma \ \text{is future-oriented causal, joining $p$ to $q$}\} \quad \text{ if } p \leq q,
\]
and $\uptau(p,q) \coloneqq 0$ if $p$ does not causally precede $q$. It is easy to show that the time-separation $\uptau$ satisfies the {\em reverse triangular inequality}, i.e.
\[
  \uptau(q_1, q_3) \geq \uptau(q_1, q_2) + \uptau(q_2, q_3) \qquad \text{ if } q_1 \leq q_2 \leq q_3 \text{ in } \heis.
\]
\begin{remark}\label{rmk:translations_preserve_geodesics}
  Left-translations are isometries, and therefore, they preserve the causal character of curves. Additionally, they also preserve the time orientation of vectors, as $X$ is left-invariant. In particular, this implies that for every points $p, q_1,q_2 \in \heis$, we have that
  \begin{equation}
    \label{eq:leqLeftTranslations}
    q_1 \leq q_2 \iff p \ast q_1 \leq p \ast q_2, \ \text{ and } \ q_1 \ll q_2 \iff p\ast q_1 \ll p\ast q_2.
  \end{equation}
  It also follows that $L(\gamma) = L(p\ast \gamma)$ for any future-directed causal curve $\gamma$, which immediately implies that $\uptau(q_1,q_2) = \uptau(p\ast q_1, p\ast q_2)$.

\end{remark}
We say that a future-directed causal curve $\gamma$ from $p$ to $q$ is a {\em maximizing geodesic} if $L(\gamma) = \uptau(p,q)$, namely that it maximizes the Lorentzian length among all future-directed causal curves joining $p$ to $q$.

\subsection{An isoperimetric characterization of geodesics}
\label{subsection:isopiffgeod}
We now aim to provide a characterization of geodesics in the sub-Lorentzian Heisenberg group as solutions to an isoperimetric problem in the Minkowski plane. This is analogous to what happens in the sub-Riemannian Heisenberg group, see for example \cite{Capogna2007}.

In what follows, we will denote by $\Min$ the two-dimensional Minkowski plane, i.e. the plane $\R^2$ equipped with the Lorentzian metric
\[
  \tilde g = - \diff x^2 + \diff y^2
\]
and time-direction given by the vector field $\partial/\partial x$. More precisely, a locally Lipschitz curve $\gamma = (x, y) : I \to \Min$ is causal and future-directed if, for almost every $t \in I$,
\[
  \tilde g(\dot \gamma(t), \dot \gamma (t)) = - \dot x(t)^2 + \dot y(t)^2 \leq 0 \ \text{ and } \ \tilde g(\dot \gamma(t), \partial/\partial x) = - \dot x(t) < 0.
\]
The Lorentzian length of such a curve is then given by
\[
  \tilde L(\gamma) \coloneqq \int_I \sqrt{-\tilde g(\dot{\gamma}(t),\dot{\gamma}(t))} \, \diff t = \int_I \sqrt{\dot x(t)^2 - \dot y(t)^2} \, \diff t,
\]
and it is well known that the corresponding time-separation function in $\Min$ is explicitely obtained as
\[
  \tilde \uptau((x_1, y_1), (x_2, y_2)) = \sqrt{(x_1-x_2)^2 - (y_1- y_2)^2}, \qquad \text{ for all } (x_1, y_1), (x_2, y_2) \in \Min.
\]
Finally, we denote by $\mathsf{p} : \heis \to \Min$ the projection defined as $\mathsf{p}(x, y, z) \coloneqq (x, y)$. Note that, for every future-directed causal curve $\gamma$ in $\heis$, the curve $\mathsf{p} \circ \gamma$ is a future-directed causal curve $\gamma$ in $\Min$ and
\begin{equation}
  \label{eq:pisanisometry}
  L(\gamma) = \tilde L(\mathsf{p} \circ \gamma).
\end{equation}

The following proposition describes the relationship between future-directed causal curves in the Heisenberg group and those in the Minkowski plane. Because of left invariance, it is enough to consider curves starting from the origin.

\begin{proposition}
  A curve $\gamma = (x, y, z) : [t_0, t_1] \to \heis$ is future-directed and causal with $\gamma(t_0) = e$ if and only if the curve $\tilde \gamma \colon [t_0, t_1] \to \Min$, defined by setting $\tilde \gamma := \mathsf{p} \circ \gamma$, is future-directed and causal with $\tilde \gamma(t_0) = 0$ and
  \begin{equation}
    \label{eq:z(t)issignedarea}
    z(t) = \frac{1}{2} \int_{t_0}^t \left( x(s) \dot y(s) - \dot x(s) y(s) \right) \diff s, \qquad \text{ for almost every } t \in I.
  \end{equation}
  Geometrically, the quantity $z(t)$ is the signed area enclosed by $\tilde \gamma$ from $t_0$ to $t$ and the segment joining $\tilde \gamma(t_0)$ to $\tilde \gamma(t)$, see \cref{fig:horizontality_condition}.
\end{proposition}

\begin{remark}
  We define the lift of any curve $\alpha \colon [t_0, t_1] \to \Min$ starting from the origin as $\gamma(t) \coloneqq (\alpha(t), z(t))$, where $z(t)$ is given by \cref{eq:z(t)issignedarea}. Clearly, the lift of a curve in $\Min$ is a horizontal curve in $\heis$.
\end{remark}

\begin{proof}
  We notice, using \cref{eq:horizontalcurveinH}, that a curve $\gamma =(x, y, z) : I \to \heis$ is horizontal if and only if it is locally Lipschitz (in charts) and
  \[
    \dot z(t) =  \frac{1}{2}\left( x(t) \dot y(t) - \dot x(t) y(t) \right), \qquad \text{ for almost every } t \in I,
  \]
  which is equivalent to \cref{eq:z(t)issignedarea} by integration.
  To interpret \cref{eq:z(t)issignedarea} geometrically, we use Green's formula to obtain the signed area $A(t)$ enclosed by the concatenation of the curves $\tilde \gamma |_{[t_0, t]}$ and $\sigma(s) \coloneqq (1-s)\tilde{\gamma}(t)$ for $s \in [0, 1]$:
  \[
    A(t) =\frac{1}{2} \int_{t_0}^{t} (x(s)\dot y(s) - \dot x(s)y(s)) \diff s + \frac{1}{2} \int_{0}^1 \sigma_1(s)\dot{\sigma}_2(s) - \dot{\sigma}_1(s)\sigma_2(s) \diff s = z(t).
  \]
  Conversely, any curve $\tilde{\gamma} = (\tilde x, \tilde y): [t_0, t_1] \to M$ that starts from the origin and encloses a signed area $\tilde z(t)$ at time $t$ can be lifted to the curve $\gamma(t) \coloneqq (\tilde{\gamma}(t), \tilde z(t))$ in $\heis$. Clearly, this curve $\gamma$ starts from the origin and satisfies $\tilde \gamma = \mathsf{p} \circ \gamma$ as well as \cref{eq:z(t)issignedarea} by construction.

  The only remaining thing to observe is that the projection map $\mathsf{p}$ is an isometry on horizontal curves preserving time-orientation, i.e.,
  \[
    g(\dot \gamma(t), \dot \gamma(t)) = \tilde{g}(\dot{\tilde \gamma}(t), \dot{\tilde \gamma}(t)),\ \text{ and } \ g(\dot \gamma(t),  X) = \tilde{g}(\dot{\tilde \gamma}(t), \partial/\partial x),
  \]
  for almost every $t \in [t_0, t_1]$. Therefore, future-directed causal curves in $\heis$ are mapped to future-directed causal curves in $\Min$ and vice versa.
\end{proof}
\begin{figure}
  \centering
  \begin{tikzpicture}
    \begin{axis}[
        scale =1.5,
        xmin=-12, ymin=-12, xmax=15, ymax=12,
        view={120}{30},
        axis lines= middle,
        axis on top,
        ytick=\empty,
        xtick=\empty,
        ztick=\empty,
        enlargelimits=false,
        clip=false,
        compat=1.8
      ]
      \addplot3[red, domain=0:10, samples=100, samples y=0,] ({x},{0.25*x*(9-x)},{9/8*x^2 -x^3/12});
      \addplot3[blue, domain=0:10, samples=100, samples y=0,] ({x},{0.25*x*(9-x)},{0});
      \addplot3[name path = projection,blue, domain=0:8, samples=100, samples y=0,] ({x},{0.25*x*(9-x)},{0});
      \addplot3[only marks] coordinates {(0,0,0) (8,2,0) (8,2,88/3)};
      \draw[color=purple, dashed, name path= segment] (axis cs:0,0,0) -- (axis cs:8,2,0);
      \addplot3[cyan] fill between [of =segment and projection];
      \draw[color=purple, dashed, name path= segment] (axis cs:0,0,0) -- (axis cs:8,2,0);
      \draw[color=cyan, dashed] (axis cs:8,2,0) -- (axis cs:8,2,88/3);
      \node[color= cyan] at (axis cs:8,-2,15) {\(z(t) = A(t)\)};
      \node[label ={[xshift=-4em]\(\gamma(t) = (x(t),y(t),z(t))\)}] at (axis cs:8,2,88/3) {};
      \node[label ={[xshift=0.5em,yshift=-2em]\((x(t),y(t),0)\)}] at (axis cs:8,2,0) {};
      \node at (axis cs:4,2.5,0) {\(A(t)\)};
    \end{axis}
  \end{tikzpicture}
  \caption{\(z(t)\) equals the area enclosed by the projection of the curve \(\gamma\) on $\Min$ at time \(t\) and the segment joining the origin to the point $(x(t), y(t), 0)$.}
  \label{fig:horizontality_condition}
\end{figure}
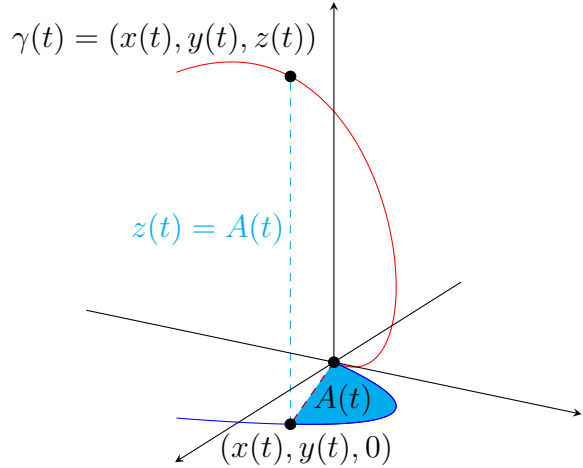
For brevity, we will say that a curve $\tilde \gamma$ in $\Min$ starting at the origin {\em encloses signed area} $z (t)$ at a fixed time $t$ if the concatenation of $\tilde \gamma$ with the segment joining $\tilde \gamma(t)$ to the origin encloses signed area $z(t)$.
As an immediate corollary of the previous proposition and \cref{eq:pisanisometry}, we obtain a characterization of geodesics in $\heis$ in terms of suitable curves in $\Min$.
\begin{corollary}
  \label{corollary:geodesicHiffisopinM}
  There is a bijective correspondence between maximizing geodesics in $\heis$ connecting the origin to a point $p = (x, y, z) \in \heis$, and future-directed
  causal curves in $\Min$ connecting the origin to $(x,y)$ that have maximal length among all the future-directed causal curves from the origin to $(x,y)$ enclosing signed area $z$.
\end{corollary}

The maximization problem in $\Min$ described in \cref{corollary:geodesicHiffisopinM}, which consists of finding the curves with maximal Lorentzian length among all future-directed causal curves connecting the origin to a given point $(x,y) \in \Min$ and enclosing a fixed signed area $z$, is a \emph{Lorentzian planar isoperimetric problem}, analogous to Dido's celebrated isoperimetric problem in the Euclidean plane, which we are going to address in the next section.

\subsection{Variational approach to the planar Lorentzian isoperimetric problem}
\label{subsection:variationalisopMink}

As established above, the geodesic problem in the sub-Lorentzian Heisenberg group admits a reduction to an isoperimetric problem in the Minkowski plane. For the reader’s convenience, we restate this problem here, suppressing the tilde notation on curves, since, throughout the present section, our analysis is restricted to $\Min$.

Let us fix $(a, b, c) \in \mathbb{R}^3$.  Among all the Lipschitz curves $\gamma \colon [t_0, t_1] \to \Min$ such that $ \gamma(t_0) = \mathbf{0}$, $\gamma(t_1) = (a, b)$, satisfying the causal condition
\begin{equation}\label{eq:causalityconditionM}
  \dot \gamma_1(t) \geq |\dot \gamma_2(t)| \quad \text{for a.e. } t \in [t_0, t_1],
\end{equation}
and subject to the following constraint on the signed area
\[
  A(\gamma) = \frac{1}{2} \int_{t_0}^{t_1} \left(\gamma_1(t)\dot{\gamma}_2(t) - \dot{\gamma}_1(t)\gamma_2(t)\right) \, \mathrm{d}t = c,
\]
we consider those curves that maximize the Minkowski length
\begin{equation}\label{eq:DidoMinkowskian}
  L(\gamma) = \int_{t_0}^{t_1} \sqrt{\dot{\gamma}_1(t)^2 - \dot{\gamma}_2(t)^2} \, \mathrm{d}t.
\end{equation}

Such curves are future-directed and causal in $\Min$, and therefore remain at all times within the causal future of the origin, namely within the region $\{x \geq |y|\}$. It follows that a necessary condition for the existence of at least one curve of this type is $a \geq |b|$. Moreover, it is not restrictive to assume that all curves are parametrized on the interval $[0,1]$, since both the Minkowski length and the signed area are invariant under reparametrization.

The following theorem provides the solution to the \emph{Minkowskian isoperimetric problem} stated above, which can be regarded as a Lorentzian analogue of the classical Dido's isoperimetric problem.
\begin{theorem}\label{thm:isoperimetric_solution}
  Consider any point $(a,b,c) \in \mathbb{R}^3$, and define the set
  \[
    \mathcal{A} \coloneqq \left\{ \alpha : [0,1] \to \Min \mid \alpha \text{ future-directed, causal,} \ \alpha(0)=\mathbf{0}, \ \alpha(1) = (a,b), \ A(\alpha)=c \right\}.
  \]
  Then, the set $\mathcal{A}$ is non-empty if and only if
  \[
    -a^2 + b^2 + 4\abs{c} \leq 0 \quad \text{and} \quad a > 0.
  \]
  In this case, there exists a unique (up to reparametrization) curve $\gamma \in \mathcal{A}$ such that
  \begin{equation}
    \label{eq:LorentzianIsopProb}
    L(\gamma) = \max_{\alpha \in \mathcal{A}} L(\alpha).
  \end{equation}
  In particular, depending on the point $(a, b , c) \in \R^3$, $\gamma$ takes the following form:
  \begin{enumerate}[label=\normalfont(\roman*), topsep=4pt,itemsep=4pt,partopsep=4pt, parsep=4pt]
    \item If $c = 0$, then $\gamma$ is a straight line from $\mathbf{0}$ to $(a,b)$. Moreover, such a curve is timelike.
    \item If $c \neq 0$ and $-a^2 + b^2 + 4\abs{c} = 0$, then $\gamma$ is a broken line with a single break point, joining $\mathbf{0}$ to $(a,b)$. Moreover, such a curve is null;
    \item If $c \neq 0$ and $-a^2 + b^2 + 4\abs{c} < 0$, then $\gamma$ is an arc of a hyperbola from $\mathbf{0}$ to $(a,b)$. Moreover, such a curve is timelike.
  \end{enumerate}
\end{theorem}
The rest of this section is devoted to the proof of \cref{thm:isoperimetric_solution}, through a series of intermediate results. The main argument relies on a reduction to a variational problem. We begin by considering the trivial case where $a = \abs{b}$, or equivalently $a^2 = b^2$. The following proposition asserts that the only future-directed causal curve from the origin to $(|b|,b)$ is a straight line segment.
\begin{proposition}\label{prop:lightlike_case}
  Suppose that $a = |b|$. Then there exists a unique future-directed causal curve joining $\mathbf{0}$ to $(a,b)$, namely the null straight line connecting these points. In this case, the Lorentzian isoperimetric problem admits a solution if and only if $c = 0$, and this solution is precisely the straight line.

\end{proposition}
\begin{proof}
  We start by noting that the null straight segment $t \in [0, 1] \mapsto (ta,tb)$ is a future-directed causal curve joining the origin to $(a,b)$. Suppose now that $\gamma$ is any future-directed causal curve from $\mathbf{0}$ to $(a,b)$. Integrating the causality condition \eqref{eq:causalityconditionM} gives
  \[
    a=\gamma_1(1) = \int_0^1 \dot{\gamma}_1(t) \diff t \geq \int_0^1 \abs{\dot{\gamma}_2(t)} \diff t \geq \abs{\int_0^1 \dot{\gamma}_2(t) \diff t} = \abs{\gamma_2(1)} = \abs{b}.
  \]
  Since we are assuming $a = \abs{b}$, these inequalities are actually equalities. In particular, $\dot{\gamma}_2$ has a constant sign almost everywhere, which we denote by $\epsilon \in\{-1,1\}$. Now, the equality
  \[
    \int_0^1 \left(\dot{\gamma}_1(t) - \abs{\dot{\gamma}_2(t)}\right) \diff t = 0,
  \]
  combined with \eqref{eq:causalityconditionM}, yields that
  \[
    \dot{\gamma}_1(t) = \abs{\dot{\gamma}_2(t)} = \epsilon \dot{\gamma}_2(t), \quad \text{ for a.e. } t \in [0, 1].
  \]
  Integrating this condition on $[0,t]$ gives that
  \[
    \gamma_1(t) = \epsilon \gamma_2(t), \quad \text{ for every } t \in [0, 1].
  \]
  It follows that the curve $\gamma$ parametrizes the line $y = \epsilon x$;
  moreover, since $L(\gamma) = A(\gamma) = 0$, the proof is complete.
\end{proof}

From now on, we assume that $a > \abs{b}$. The following proposition shows that it suffices to consider curves that are graphs of Lipschitz functions vanishing at their endpoints. To this end, we introduce the linear map $R \colon \R^2 \to \R^2$ defined by the matrix
\begin{equation}
  \label{eq:matrixLorentzRotation}
  R :=
  \begin{pmatrix}
    \dfrac{a}{\sqrt{a^2-b^2}} & -\dfrac{b}{\sqrt{a^2-b^2}}\\
    -\dfrac{b}{\sqrt{a^2-b^2}} & \dfrac{a}{\sqrt{a^2-b^2}}
  \end{pmatrix}.
\end{equation}
Observe that $R$ is a future-preserving Lorentzian isometry of the Minkowski plane, which sends $(a,b)$ to $(\sqrt{a^2 - b^2},0)$. For a given $(a, b)$, $a > \abs{b}$ , we shall usually adopt the notation $T \coloneqq \sqrt{a^2 - b^2}$.

\begin{proposition}
  \label{prop:MinkowskiToGraphical}
  Let $(a,b) \in \Min$ with $a > \abs{b}$, and let $\gamma \colon [0,1] \to \Min$ be a future-directed causal curve connecting $\mathbf{0}$ to $(a,b)$. Then the image of the curve $R \circ \gamma$, which is future-directed, has the same Minkowski length and encloses the same area as $\gamma$, and connects $\mathbf{0}$ to $(T,0)$, is realized as the graph of a function $f \colon [0,T] \to \R$ satisfying $\Lip(f) \leq 1$.
\end{proposition}
\begin{proof}
  Observe that since $a > \abs{b}$, the point $(a,b)$ lies in the chronological future of the origin, i.e. in the set $\{x > \abs{y}\}$.

  The fact that $R$ is a future-preserving Lorentzian isometry ensures that the curve $\bar{\gamma} \coloneqq R \circ \gamma$ is a future-directed causal curve with the same length as $\gamma$, connecting the origin to the point $(T,0)$. Since $\det(R)=1$, it is also immediate to see that $\gamma$ and $\bar{\gamma}$ enclose the same area, i.e. $A(\gamma) = A(\bar \gamma)$.

  We now show that the image of $\bar \gamma$ is the graph of a function.
  Since $\bar \gamma$ is future-directed, it holds that
  \[
    \dot{\bar{\gamma}}_1(t) > 0 \qquad \text{for a.e.} \ t \in [0,1],
  \]
  and therefore the function
  \[
    s(t) \coloneqq \int_0^t \dot{\bar{\gamma}}_1 = \bar{\gamma}_1(t)
  \]
  is a homeomorphism from $[0,1]$ to $[0,T]$.
  Notice that both $s(t)$ and its inverse $t(s)$ are increasing functions, and thus the parametrization $\alpha(s) \coloneqq \bar{\gamma}(t(s))$ is well defined. Observe that this parametrization is a graph of a function, since
  \begin{equation}\label{eq:graphical_parametrization}
    \alpha_1(s) := \bar{\gamma}_1(t(s)) = s(t(s))= s
  \end{equation}
  We now denote the second component $\alpha_2(s) := \bar{\gamma}_2(t(s))$ by $f(s)$. Because $\bar{\gamma}$ is causal, we have that
  \[
    \bar{\gamma}_1(t_2) - \bar{\gamma}_1(t_1) = \int_{t_1}^{t_2} \dot{\bar{\gamma}}_1(t) \, dt \geq \int_{t_1}^{t_2} \abs{\dot{\bar{\gamma}}_2(t)} \, dt \geq \abs{\int_{t_1}^{t_2} \dot{\bar{\gamma}}_2(t) \, dt} = \abs{\gamma_2(t_2) - \gamma_2(t_1)},
  \]
  for any $t_1 < t_2$, while for $s_1 < s_2$, we have that $t(s_1) < t(s_2)$, since $t(s)$ is increasing, and we finally obtain
  \[
    \abs{f(s_2) -f(s_1)} = \abs{\bar{\gamma}_2(t(s_2)) - \bar{\gamma}_2(t(s_1))} \leq \bar{\gamma}_1(t(s_2)) - \bar{\gamma}_1(t(s_1)) = s_2-s_1,
  \]
  where the last equality follows from \eqref{eq:graphical_parametrization}.

  This shows that $f$ is a Lipschitz function with Lipschitz constant bounded by 1. Moreover, we have $f(0) =f(T) = 0$, since $\bar{\gamma}_2(0) = \bar{\gamma}_2(1)=0$.
\end{proof}

In the following, with a slight abuse of notation, we denote by $A(f)$ the signed area enclosed by the curve $\gamma(t) = (t, f(t))$, which is graph of the function $f$, i.e.
\[
  A(f) = A(\gamma) \coloneqq \int_0^T f(t) \, \diff t.
\]
The previous proposition, together with the fact that $R$ is a measure-preserving Lorentzian isometry, shows that the original Lorentzian isoperimetric problem in \eqref{eq:DidoMinkowskian} is equivalent to the following: considering the set
\[
  \mathcal{F} \coloneqq \left\{ f \in \Lip([0,T], \R) \ \colon \ f(0) = f(T) = 0, \ \int_0^T f(t) \diff t = c, \ \Lip(f) \leq 1 \right\},
\]
maximize the quantity
\begin{equation}\label{eq:graphical_length}
  L(f) = L(\gamma) \coloneqq \int_0^T \sqrt{1 - \dot{f}(t)^2} \diff t
\end{equation}
over all functions $f \in \mathcal{F}$.

We now address the first claim of \cref{thm:isoperimetric_solution}, namely the necessary and sufficient conditions on the set $\mathcal{A}$ (or  equivalently $\mathcal{F}$) to be non-empty.
\begin{proposition}\label{prop:future_set_origin}
  The set $\mathcal{F}$ is non-empty if and only if $-a^2+b^2+4\abs{c} \leq 0$.
\end{proposition}
\begin{proof}
  Suppose that there exists a function $f \in \mathcal{F}$. Since $f$ satisfies $\Lip(f) \leq 1$, its graph lies within the cone $\{ \abs{y - f(x_0)} \leq \abs{x - x_0} \}$ for any $x_0 \in [0, T]$. In particular, taking $x_0 = 0$ and $x_0 = T$ this implies that for the graph of $f$, $\operatorname{graph}(f)$, the following inclusion holds
  \begin{equation}
    \operatorname{graph}(f) \subseteq \label{eq:intersectiontwocones}
    \{ \abs{y} \leq \abs{x} \} \cap \{ \abs{y} \leq \abs{x - T} \},
  \end{equation}
  which is equivalent to say that
  \begin{equation}
    \label{eq:graphinclusion}
    \operatorname{graph}(f) \subseteq \{ \abs{y} \leq \min(\abs{x}, \abs{x - T}) \}.
  \end{equation}
  Recalling the condition $A(f) = c$, we observe that
  \begin{equation}\label{eq:area_constraint_estimates}
    \abs{c} = \abs{\int_0^T f(t) \diff t} \leq \int_0^T \abs{f(t)} \diff t \leq \int_0^T \min(\abs{t}, \abs{t - T}) \diff t = \frac{T^2}{4}.
  \end{equation}
  Since $T = \sqrt{a^2-b^2}$, the previous condition can be rewritten as $-a^2+b^2+4\abs{c} \leq 0$.

  Conversely, assume that $-a^2 + b^2 + 4\abs{c} \leq 0$, that is, $4\abs{c}/T^2 \leq 1$. Consider the function $f \colon [0,T] \to M$ defined as
  \[
    f(x) =\frac{2c}{T} -\frac{4c}{T^2}\abs{x-\frac{T}{2}}.
  \]
  We illustrate this function in \cref{fig:functionbrokenline}. It is easy to verify that $f$ is a $4\abs{c}/T^2$-Lipschitz function, so that $\Lip(f) \leq 1$. Moreover, it is also immediate to see that $f(0)=f(T)=0$ and that $A(f) = c$, hence $f \in \mathcal{F}$.
\end{proof}
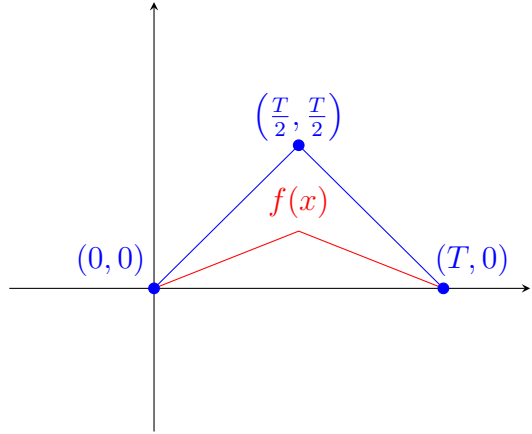
\begin{figure}
  \centering
  \begin{tikzpicture}
    \begin{axis}[
        axis lines = middle,
        xtick = \empty,
        ytick = \empty,
        xmax = 1.3,
        ymax = 1,
        xmin = -0.5,
        ymin = -0.5,
        enlargelimits = false,
        clip = false
      ]
      \addplot+ [sharp plot, color=blue, mark options={fill=blue}] coordinates {(0,0) (0.5,0.5) (1,0) }
      node at (axis cs:-0.15,0.1) {\((0,0)\)}
      node at (axis cs:0.5,0.6) {\(\left( \frac{T}{2}, \frac T 2\right)\)}
      node at (axis cs:1.1,0.1) {\((T,0)\)};
      \addplot [color=red] coordinates {(0,0) (0.5,0.2) (1,0)}
      node at (axis cs: 0.5, 0.3) {\(f(x)\)}
      ;
    \end{axis}
  \end{tikzpicture}
  \caption{The function \(f(x)=\frac{2c}{T} -\frac{4c}{T^2}\abs{x-\frac{T}{2}}\)}
  \label{fig:functionbrokenline}
\end{figure}
As a corollary of the previous proposition, we can already solve the Lorentzian isoperimetric problem stated in \cref{thm:isoperimetric_solution} whenever we are in the case $c = 0$ or in the case $-a^2 + b^2 + 4\abs{c} = 0$ with $c \neq 0$.
\begin{corollary}\label{cor:maximum_area_and_timelike_line_case}
  If $c=0$ and $a > \abs{b}$, then the solution to the Lorentzian isoperimetric problem in \cref{thm:isoperimetric_solution} is given by a straight timelike line joining $\mathbf{0}$ to $(a,b)$.

  If $c\neq0$ and $-a^2+b^2+4\abs{c} = 0$, the set $\mathcal{F}$ has only one element, namely, the function $f$ defined by
  \begin{equation}\label{eq:brokenline}
    f(x) \coloneqq  \sgn(c)\min(x,\abs{x-T}).
  \end{equation}
  In particular, the maximizer of \eqref{eq:graphical_length} over $\mathcal{F}$ is precisely given by $f$, whose graph defines a broken lightlike line.
\end{corollary}
\begin{proof}
  Assume that $c = 0$ and $a > \abs{b}$. Then the condition $-a^2 + b^2 + 4\abs{c} \leq 0$ is satisfied, so by \cref{prop:future_set_origin}, the set $\mathcal{F}$ is nonempty. Notice that for any $f \in \mathcal{F}$, we can estimate the length $L(f)$ given in \eqref{eq:graphical_length} as follows:
  \[
    L(f) = \int_0^T \sqrt{1 - \dot{f}(t)^2} \, dt \leq \int_0^T 1 \, dt = T,
  \]
  with equality if and only if $\dot{f}(t) = 0$ for almost all $t \in [0,T]$, which is equivalent to $f \equiv 0$ since $f$ vanishes on the boundary. Since we are assuming that $c = 0$, the function $f \equiv 0$ indeed belongs to the set $\mathcal{F}$ and is therefore the unique solution to the maximization problem. We also observe that this curve is timelike.

  Assume now that $c \neq 0$ and that $-a^2 + b^2 + 4\abs{c} = 0$. The chain of inequalities shown in \eqref{eq:area_constraint_estimates} then becomes a chain of equalities. In particular, $f$ has a definite sign, which we denote by $\epsilon \in \{-1, 1\}$, and we also have
  \[
    \abs{f(x)} = \epsilon f(x) = \min(x, \abs{x - T}).
  \]
  Rewriting the function as $f(x) = \epsilon \min(x, \abs{x - T})$, the condition $A(f) = c$ yields
  \[
    c = \int_0^T \epsilon \min(x, \abs{x - T}) \diff x = \epsilon \, \frac{T^2}{4} = \epsilon \abs{c},
  \]
  from which it follows that $\sgn(c) = \epsilon$, thus completing the proof.
\end{proof}
In the Riemannian case, where we {\em minimize} the Euclidean length, the solution to the isoperimetric problem is known to be an arc of a circle that encloses an area precisely equal to $c$. In the Lorentzian setting, the time-separation replaces the distance, and the locus of points which are ``equidistant'' to a center/vertex is given by a \emph{hyperbola}. This will exactly be the solution mentioned in \cref{thm:isoperimetric_solution}, so we now seek the specific hyperbola that satisfies the area constraint.
\begin{lemma}\label{lemma:existence_hyperbola}
  Consider any point $(a, b) \in \Min$ with $a > \abs{b}$ and any $c \neq 0$ satisfying the constraint $-a^2 + b^2 + 4\abs{c} < 0$. Then there exists a unique arc of hyperbola from $\mathbf{0}$ to $(a, b)$ that encloses area $c$. Moreover, such an arc is timelike.
\end{lemma}
\begin{proof}
  We start by computing the equation of the (Minkowskian) axis between the origin and the point $(T,0)$:
  \[
    (x-T)^2-y^2 = x^2-y^2 \implies x = \frac{T}{2}.
  \]
  If $(T/2, y_C)$ is the vertex, then the equation of the hyperbola is
  \[
    \left(x - \frac{T}{2}\right)^2 - (y_C - y)^2 = \pm \left(\frac{T^2}{4} - y_C^2\right).
  \]
  If the minus sign is chosen, then the constraint that the hyperbola passes through the origin implies that it has area $T^2/4 > |c|$ because it degenerates into two broken lines, as in \eqref{eq:brokenline} (up to a sign). Thus we must choose the plus sign and setting $x = T/2$ in this equation implies that $y_C^2 - T^2/4 \geq 0$. Therefore we obtain
  \[
    y = y_C \pm \sqrt{\left(x - \frac{T}{2}\right)^2 + y_C^2 - \frac{T^2}{4}}.
  \]
  Imposing that the hyperbola passes through the origin tells us that the correct branch of the hyperbola is
  \[
    f_{y_C}(x) \coloneqq y_C - \sgn(y_C) \sqrt{\left(x - \frac{T}{2}\right)^2 + y_C^2 - \frac{T^2}{4}}.
  \]

  It is straightforward to verify that this defines a causal hyperbola on $[0, T]$ which also passes through the point $(T, 0)$. More precisely, the hyperbola degenerates into a broken line when $y_C = \pm T/2$, in which case the curve is null. If $\abs{y_C} > T/2$, the curve is timelike. It is also immediate to see that $f_{y_C}$ has the same sign as $y_C$, so that $f_{y_C} \geq 0$ when $y_C \geq T/2$, and $f_{y_C} \leq 0$ when $y_C \leq -T/2$.

  We now need to choose $y_C$ so that the hyperbola encloses the desired area. To this end, we differentiate the expression for the hyperbola with respect to $y_C$; this is justified since, away from $y_C = 0$, the function is smooth in $y_C$ and satisfies $y_C^2 \geq T^2/4 > 0$. We get
  \[
    \pder{f_{y_C}(x)}{y_C} = 1 - \frac{\abs{y_C}}{\sqrt{\left(x - \frac{T}{2}\right)^2 + y_C^2 - \frac{T^2}{4}}} \leq 0,
  \]
  which shows that the function is non-increasing in $y_C$ on each connected component of $(-\infty, -T/2] \cup [T/2, +\infty)$. More precisely, if $x \neq 0$ or $x \neq T$, the derivative above is strictly negative; and since the hyperbola evaluates to zero at $x = 0$ and $x = T$, this tells us that the family of functions parametrized by $y_C$ has fixed boundary values and is decreasing (again, on each connected component of $(-\infty, -T/2] \cup [T/2, +\infty)$) as $y_C$ increases. Next, we consider the area enclosed by the hyperbola, given by
  \[
    A(f_{y_C}) = \int_0^T y_C - \sgn(y_C)\sqrt{\left(x - \frac{T}{2}\right)^2 + y_C^2 - \frac{T^2}{4}} \diff x,
  \]
  and observe that, by the monotonicity of the integrand and the continuity of the functions $f_{y_C}$, the area is decreasing in $y_C$ on each connected component as well. We notice that for $y_C = \pm T/2$, the area becomes
  \[
    A(f_{y_C}) = \int_0^T \pm\frac{T}{2} \mp \abs{x - \frac{T}{2}} \diff x = \pm\frac{T^2}{4},
  \]
  whereas, by letting $y_C \to \pm\infty$ and applying the Monotone Convergence Theorem, we find
  \[
    \lim_{y_C \to \pm\infty} A(f_{y_C}) = \int_0^T \lim_{y_C \to +\infty} \left( y_C - \sgn(y_C) \sqrt{\left(x - \frac{T}{2}\right)^2 + y_C^2 - \frac{T^2}{4}} \right) \diff x = 0.
  \]
  The Monotone Convergence Theorem actually ensures that the area is a continuous function of $y_C$, hence $A(f_{y_C})$ attains all values in $(0, T^2/4]$ for $y_C \geq T/2$, and all values in $[-T^2/4, 0)$ for $y_C \leq -T/2$. Since we are assuming $0 < \abs{c} < T^2/4$, it follows by the Intermediate Value Theorem and the monotonicity of the area that there exists a unique value $\bar y_C \in (-\infty, -T/2] \cup [T/2, +\infty)$ such that $A(f_{\bar y_C}) = c$.

  Therefore, there exists a unique timelike hyperbola joining the origin to $(T, 0)$ that satisfies the area constraint.
\end{proof}

Notice that the Lipschitz constant of the function $f_{\bar y_C}(\cdot)$ is strictly less than 1.

To prove \cref{thm:isoperimetric_solution}, we aim to apply methods from the Calculus of Variations. We would like to compute the Euler-Lagrange equation associated with the functional $L$ and show that its solution must be a maximizer. However, a subtle issue arises when trying to work with variations of a function $f \in \mathcal{F}$: the condition $\operatorname{Lip}(f) \leq 1$ is not preserved under variations. As a workaround, we establish the following lemma.

\begin{lemma}\label{lemma:reduction_timelike_case}
  Let $(a,b,c) \in \R^3$ with $c \neq 0$ and $-a^2 + b^2 + 4\abs{c} < 0$. Consider the set defined by
  \[
    \mathcal{F}' = \left\{ f \in \operatorname{Lip}([0,T]) \mid f(0) = f(T) = 0, \ \operatorname{Lip}(f) < 1, \ \int_0^T f \, dt = c \right\} \subseteq \mathcal F.
  \]
  Then it holds
  \begin{equation}\label{eq:sup_timelike}
    \sup_{f \in \mathcal{F}'} L(f) = \sup_{f \in \mathcal{F}} L(f).
  \end{equation}
\end{lemma}
\begin{proof}
  As $\mathcal{F}' \subset \mathcal{F}$, the inequality $\leq$ in \eqref{eq:sup_timelike} is immediate. To prove the opposite inequality, consider any function $g \in \mathcal{F}$. By \cref{lemma:existence_hyperbola}, we know that $\mathcal{F}' \neq \emptyset$, so consider $f \in \mathcal{F}'$. Let us define the sequence of functions $\{g_n\}_n$ by
  \[
    g_n = \frac{1}{n}f + \left( 1 - \frac{1}{n} \right)g,
  \]
  that is, $g_n$ is a convex combination of $f$ and $g$. It is straightforward to see that $g(0) = g(T) = 0$, and that by linearity and convexity, $A(g_n) = c$. Moreover, we have that
  \[
    \Lip(g_n) \leq \frac{1}{n}\Lip(f) + \left(1 - \frac{1}{n}\right)\Lip(g) < \frac{1}{n} + 1 - \frac{1}{n} = 1,
  \]
  since we know that $\Lip(f) < 1$.
  This implies that $(g_n)_n$ is a sequence in $\mathcal{F}'$. It is easy to see that the sequence $(g_n)_n$ converges pointwise to $g$, and the sequence of derivatives $(\dot{g}_n)_n$ also converges pointwise almost everywhere to $\dot{g}$. The set where all the derivatives are defined has full measure, given that we are dealing with a countable collection of functions defined almost everywhere.

  By the Dominated Convergence Theorem, and noting that the integrand of $L_\uptau$ is uniformly bounded above by 1, we obtain
  \[
    \lim_{n \to +\infty} L(g_n) = \int_0^T \lim_{n \to +\infty} \sqrt{1 - \dot{g}_n^2} \diff t = \int_0^T \sqrt{1 - \dot{g}^2} \diff t = L(g).
  \]
  Therefore, we have proven that for any $g \in \mathcal{F}$, there exists a sequence of functions $(g_n)_n$ in $\mathcal{F}'$ whose lengths approach the length of $g$, thus establishing the inequality $\geq$ in \eqref{eq:sup_timelike}.
\end{proof}

We are now ready to prove \cref{thm:isoperimetric_solution}:
\begin{proof}[Proof of \cref{thm:isoperimetric_solution}]
  The fact that $\mathcal{A} \neq \emptyset$ if and only if $a>0$ and $-a^2+b^2+4\abs{c} \leq 0$ was already established by \cref{prop:future_set_origin}. Assuming that the set $\mathcal{A}$ is non-empty, the simpler cases described in items {\em (i)} and {\em (ii)} follow from \cref{prop:lightlike_case} and \cref{cor:maximum_area_and_timelike_line_case}. We therefore focus on proving item {\em (iii)}, which is equivalent, by \cref{lemma:reduction_timelike_case}, to finding the maximizers of the Lorentzian length $L$ over the set $\mathcal{F}'$.

  If $f \in \mathcal{F}'$ is a maximum for our functional, let us consider any Lipschitz function $\varphi$ such that $\varphi(0)=\varphi(T)=0$ and such that it is average-free. The function $f + s\varphi$ then belongs to $\mathcal{F}'$ for small values of $s$. Indeed, the area constraint and boundary conditions are trivially satisfied, while the Lipschitz constant satisfies
  \[
    \Lip(f+s\varphi) \leq \Lip(f) + s \Lip(\varphi),
  \]
  which is strictly less than $1$ at $s=0$ by hypothesis, and thus remains so near $s=0$ by continuity.
  We are then allowed to compute the first variation of the functional with respect to these variations, and we find that
  \begin{equation}
    \label{eq:FirstVariation}
    0 = \frac{\diff}{\diff s} \big( L(f+s\varphi) \big)\big|_{s = 0} = -\int_0^T\frac{\dot{f}\dot{\varphi}}{\sqrt{1-\dot{f}^2}} \diff t,
  \end{equation}
  by the Dominated Convergence Theorem. Given any Lipschitz function $\psi$ vanishing on the boundary of $[0,T]$ and any Lipschitz function $\rho$ that also vanishes on the boundary and  integrates to
  1, we define
  \[
    \varphi \coloneqq \psi - \rho \int_0^T \psi \diff t.
  \]
  The function $\varphi$ is Lipschitz, vanishes at boundary points, and is average-free. Therefore, plugging it into the first variation formula \eqref{eq:FirstVariation} gives
  \begin{equation}\label{eq:euler_lagrange}
    \int_0^T\dot{\psi}(t)\left(\frac{\dot{f}(t)}{\sqrt{1-\dot{f}^2(t)}} \right) \diff t=  \int_0^T \psi(t)\int_0^T \frac{\dot{f}(u)\dot{\rho}(u)}{\sqrt{1-\dot{f}^2(u)}} \diff u \diff t.
  \end{equation}
  Since this equality holds for any Lipschitz function $\psi$ that vanishes at the boundary, we have that
  \[
    D\left( \frac{\dot{f}(t)}{\sqrt{1-\dot{f}^2(t)}} \right) = -\int_0^T \frac{\dot{f}(u)\dot{\rho}(u)}{\sqrt{1-\dot{f}^2(u)}} \diff u,
  \]
  where $D$ denotes the weak derivative. Notice that the right-hand side is constant in $t$, so we denote it by $\lambda$. It is well known that a function with $L^1$ weak derivative is equal to an absolutely continuous function almost everywhere, and so integrating the previous identity yields
  \[
    \frac{\dot{f}(t)}{\sqrt{1-\dot{f}^2(t)}} = \lambda t + C \qquad \text{ for a.e. } t \in [0,T],
  \]
  for some constant $C \in \R$. Noting that $\dot{f}$ and $\lambda t + C$ share the same sign for almost every $t$, we can square the equation and rearrange it to obtain
  \[
    \dot{f}(t) = \frac{\lambda t + C}{\sqrt{1 + (\lambda t + C)^2}}.
  \]
  If $\lambda = 0$, then $f$ is a constant function. Since it vanishes at the boundary, we have $f \equiv 0$. This would imply that the area enclosed by $f$ is zero, but we are considering the case $c \neq 0$. Therefore $\lambda \neq 0$ and integrating the previous identity yields
  \[
    f(t) = \frac{1}{\lambda} \sqrt{1 + (\lambda t + C)^2} + K,
  \]
  for some constant $K$. Observe that this function describes a hyperbola. Since $f \in \mathcal{F}'$, we have $f(0) = f(T) = 0$ and $A(f) = c$. By \cref{lemma:existence_hyperbola}, there exists a unique hyperbola satisfying these conditions, which is also Lipschitz with $\Lip(f) < 1$. Therefore, the constants $C$, $K$, and $\lambda$ are uniquely determined and hence if a maximum of $L$ over $\mathcal{F}'$ does exist, it must be this unique hyperbola in $\mathcal{F}'$, which satisfies the Euler--Lagrange equation \eqref{eq:euler_lagrange}.

  To prove that the hyperbola is indeed the maximizer of the functional $L$, we argue by convexity. Let $f$ be said hyperbola and consider any other element $g \in \mathcal{F}'$ with $g \neq f$; such choice is possible as \cref{prop:future_set_origin} shows the existence of a piecewise linear function in $\mathcal{F}'$, which must then be different from the hyperbola. We consider the function
  \[
    F(s) \coloneqq L(f+s(g-f)) = \int_0^T \sqrt{1 - (\dot{f} +s(\dot{g} - \dot{f}))^2} \diff t.
  \]
  Since $\mathcal{F}'$ is convex, the function $f + s(g - f)$ belongs to $\mathcal{F}'$ for all $s \in [0,1]$. By the Dominated Convergence Theorem, it is also clear that $F$ is a smooth function. We compute the second derivative of $F$ and obtain
  \[
    \frac{\diff^2F}{\diff s^2}(s) = - \int_0^T \frac{(\dot{g} - \dot{f})^2}{\left(1 - (\dot{f} + s(\dot{g} - \dot{f}))^2\right)^{3/2}} \diff t.
  \]
  This integral is strictly negative, since it is nonpositive, and if it were zero, then $\dot{g} = \dot{f}$ for almost every $t$; integrating and using $f(0) = g(0) = 0$, we would obtain $f = g$, contradicting our assumption that $f \neq g$.

  This shows that the function $F$ is strictly concave. Applying Lagrange's Remainder Theorem, we find that for some $t' \in (0,1)$,
  \begin{equation}
    \label{eq:greatinequality}
    F(1) = F(0) + F'(0) + \frac{F''(t')}{2} < F(0) + F'(0),
  \end{equation}
  where the inequality follows from the fact that the second derivative of $F$ is strictly negative. We also notice that the function $g - f$ is Lipschitz, vanishes at the boundary, and has zero average, and that $f + s(g - f)$ is one of the variations considered in the derivation of the Euler-Lagrange equation \eqref{eq:euler_lagrange}. Since $f$ is a solution of the Euler-Lagrange equation and $F'(0)$ is precisely the first variation of the functional with respect to the variation $g - f$, it follows that $F'(0) = 0$. Therefore, last inequality \eqref{eq:greatinequality} becomes
  \[
    F(1) < F(0),
  \]
  which, written explicitly, reads
  \[
    L(g) < L(f).
  \]
  This confirms that $f$ is the unique maximizer in $\mathcal{F}'$, and hence the unique solution to our problem.
\end{proof}

As shown at the beginning of this subsection, a maximizer in $\heis$ joining $e$ to $(a, b, c) \in \heis$ projects to a maximizer of the corresponding Minkowskian isoperimetric problem in $\Min$, and conversely, any such isoperimetric maximizer lifts to a maximizing curve in $\heis$ via \eqref{eq:z(t)issignedarea}.
When \(a > \abs{b}\), we have seen that any maximizer of the Minkowskian isoperimetric problem can be transformed, via rotation by the matrix $R$ from \eqref{eq:matrixLorentzRotation} and an appropriate reparametrization, into a maximizer of the functional $L(f)$ over all $f \in \Lip([0,T], \R)$ belonging to $\mathcal{F}$, and vice versa. Using these correspondences, \cref{maintheorem1:geodesicssubLHeisenberg} follows directly from \cref{thm:isoperimetric_solution}.

\begin{remark}
  We have studied geodesics in the sub-Lorentzian Heisenberg group via the Lorentzian planar isoperimetric problem. Alternatively, one can use optimal control theory and apply Pontryagin's maximum principle to derive first order necessary conditions for length maximizers. This approach was used in \cite{sachkovsachkova23}, and our results agree with theirs. A key feature of Pontryagin's maximum principle is that optimizers can be \emph{normal} or \emph{abnormal}. Normal curves satisfy a Hamiltonian differential equation, while abnormal ones are harder to characterize. Many open problems and conjectures concern them in sub-Riemannian geometry, see \cite{Agrachev2020}. In sub-Lorentzian geometry, this subject remains almost entirely unexplored. All maximizers in the sub-Lorentzian Heisenberg group are either normal or abnormal. The only abnormal ones correspond precisely to maximal area solutions, i.e. to the broken lightlike segments we computed. Furthermore, these are the only solutions not smooth everywhere, having a single point of non-differentiability.
\end{remark}

\subsection{The sub-Lorentzian exponential map}
\label{section:exponentialmap}

We conclude this section by deriving the so-called exponential map of the sub-Lorentzian Heisenberg group, that is, the map describing the timelike geodesic flow. In doing so, we recover the exponential map from \cite[Eq. (5.4)]{sachkovsachkova23} (see also \cite[Def. 7]{borza2025}) as a direct consequence of solving the Minkowskian isoperimetric problem.
In the previous section, we showed that for any $(a, b, c) \in I^{+}(\mathbf{0}) \subseteq \heis$ with $c \neq 0$, there is an hyperbola with vertex $(x_C, y_C)$ and Lorentzian length $L > 0$ joining the origin of $\Min$ to $(a, b)$, which is the unique solution to the Minkowskian isoperimetric problem. Recall that by construction, the vertex $(x_C, y_C)$ of the constructed hyperbola satisfies $\abs{y_C} > \abs{x_C}$, since it is non-degenerate and is a graph over the \(x\)-axis. Conversely, starting from a triple $(x_C, y_C, L)$ such that $\abs{y_C} > \abs{x_C}$ and $L > 0$, we can move along the hyperbola that has vertex $(x_C, y_C)$ passing through the origin and $(a, b)$ until the Lorentzian length of the traversed arc equals $L$ (note that since the hyperbola is timelike, the Lorentzian length along it always increases). We can then lift the point we reach to a point $(a, b, c) \in \heis$ such that the geodesic connecting it to the origin is exactly the lift of the hyperbola. From this observation, one immediately recovers the formulas
\begin{equation}
  \label{eq:causalfutureof0}
  J^+(\mathbf{0})=\{ (x, y, z) \in \heis \mid -x^2+y^2+4\abs{z} \leq 0 \text{ and } x \geq 0\},
\end{equation}
and
\begin{equation}
  \label{eq:chronologicalfutureof0}
  I^+(\mathbf{0})=\{ (x, y, z) \in \heis \mid -x^2+y^2+4\abs{z} < 0 \text{ and } x > 0\},
\end{equation}
which already appear in \cite{Grochowski2006,sachkovsachkova23}.

We stress that we have not provided an explicit formula for determining the vertex of the hyperbola from a given point in $I^+(\mathbf{0})$; we only know that it is uniquely determined.

\begin{remark}
  The hyperbola degenerating into a straight line, i.e. when its vertex $(x_C, y_C)$ \say{tends to infinity}, corresponds to $c$ tending to $0$ and to the area enclosed by the curve joining the origin to $(a, b)$ vanishing.
\end{remark}

We now wish to write down a natural parametrization for the maximizing geodesics of $\heis$, by lifting a parametrization of the isoperimetric maximizer. We start from the expression of the hyperbola as in the proof of \Cref{lemma:existence_hyperbola}:
\begin{equation}
  \label{eq:hyperbola}
  y = y_C - \sgn(y_C) \sqrt{(x - x_C)^2 + y_C^2 - x_C^2}.
\end{equation}
We compute the Lorentzian arclength $l(x)$ of this hyperbola and obtain
\[
  l(x) = \sqrt{y_C^2-x_C^2} \left( \arcsinh \left( \frac{x-x_C}{\sqrt{y_C^2-x_C^2}} \right) - \arcsinh \left( \frac{-x_C}{\sqrt{y_C^2-x_C^2}} \right) \right).
\]
By isolating the first term on the right-hand side and applying $\sinh$, we find that
\[
  x = x_C + \abs{y_C} \sinh\left(\frac{l(x)}{\sqrt{y_C^2 - x_C^2}}\right) - x_C \cosh\left(\frac{l(x)}{\sqrt{y_C^2 - x_C^2}}\right).
\]
We can then substitute in the expression \cref{eq:hyperbola} to also get
\[
  y = y_C + \sgn(y_C)x_C \sinh \left(\frac{l(x)}{\sqrt{y_C^2-x_C^2}}\right) -y_C \cosh \left(\frac{l(x)}{\sqrt{y_C^2-x_C^2}}\right).
\]
Hence, we can use the variable $l \in [0, L]$ as a parameter for the hyperbola. However, we reparametrize it using $t \coloneqq l / L$, so that $t \in [0, 1]$. By computing the initial velocity of the geodesic with respect to this parameter $t$, we define two parameters $u, v$ as
\[
  u \coloneqq \dot{x}(0) = \frac{\sgn(y_C)L}{\sqrt{y_C^2-x_C^2}} \, y_C, \qquad v \coloneqq \dot{y}(0) = \frac{\sgn(y_C)L}{\sqrt{y_C^2-x_C^2}} \, x_C.
\]
We can then compute the component $z(t)$, which represents the area enclosed by the curve at time $t$. A straightforward computation shows that
\[
  z(t) = \frac{y_C^2 - x_C^2}{2} \left( \sinh\left(\frac{-\sgn(y_C) L}{\sqrt{y_C^2 - x_C^2}} \, t \right) - \left( \frac{-\sgn(y_C) L}{\sqrt{y_C^2 - x_C^2}} \, t \right) \right).
\]
We finally define a parameter $w$ as
\[
  w \coloneqq \frac{-\sgn(y_C)L}{\sqrt{y_C^2 - x_C^2}},
\]
so that when $w$ is positive (resp. negative), the area increases (resp. decreases) with time, and $w = 0$ corresponds to the case of vanishing area.

The triple $(u,v,w)$ is clearly completely determined by $(x_C,y_C,L)$. Conversely, it is not hard to see that
\begin{equation}\label{eq:center_from_u_v_w}
  x_C = -\frac{v}{w}, \qquad y_C = - \frac{u}{w}, \quad \text{ and } \quad L = u^2-v^2.
\end{equation}

Notice that the condition \(\abs{y_C} > \abs{x_C}\) is equivalent to \(u > \abs{v}\). In other words, there is a bijection between the triples \((u, v, w) \in \{u \geq \abs{v}\} \times \R\) and the triples $(x_C, y_C, L)$, which induces a bijection between $\{u \geq \abs{v}\} \times \R$ and $I^+(\mathbf{0})$. The case $w = 0$ is again interpreted as corresponding to vanishing area and the vertex of the hyperbola going to infinity.

The expression \cref{eq:center_from_u_v_w} also extends to the case where \(u=\abs{v}\), which corresponds to the case where the hyperbola degenerates into a 45° broken line. When this happens, it holds that \(L=0\), the line will be lightlike and thus the Lorentzian length does not increase with proper time; however, the parametrization does not see the break, but instead keeps running along the lightlike line. Nevertheless, the lightlike line still corresponds to points in the causal future of the origin, so it is still meaningful to consider this case when dealing with \(J^+(\mathbf{0})\).

\begin{definition}\label{def:exponential_map}
  For given \(t \in \R\) and \((u, v, w) \in U \coloneqq \{u > \abs{v}\} \times \R\), the sub-Lorentzian exponential map of the Heisenberg group is given by
  \[
    \exp_{\mathbf{0}}^{(t)}(u, v, w) \coloneqq
    \begin{pmatrix}
      \dfrac{v(\cosh(wt)-1) +u\sinh(wt)}{w}\\[1em]
      \dfrac{v\sinh(wt) +u(\cosh(wt)-1)}{w}\\[1em]
      \dfrac{u^2-v^2}{2}\cdot\dfrac{\sinh(wt)-wt}{w^2}
    \end{pmatrix}, \qquad \text{ if } w \neq 0,
  \]
  and smoothly extended to $w = 0$ by
  \[
    \exp_{\mathbf{0}}^{(t)}(u, v, 0) \coloneqq
    \begin{pmatrix}
      x(t)\\
      y(t)\\
      z(t)
    \end{pmatrix}
    =
    \begin{pmatrix}
      ut\\
      vt\\
      0
    \end{pmatrix}.
  \]
\end{definition}

\begin{proposition}
  \label{prop:expdiffeom}
  For all $(u,v,w) \in U$ and all $T > 0$, the curve defined on \([0,T]\) given by
  \[
    t \mapsto \exp_{\mathbf{0}}^{(t)}(u,v,w)
  \]
  is a length maximizing future-directed timelike geodesic in $\heis$, starting from $\mathbf{0}$. Furthermore, the exponential map $\exp_{\mathbf{0}}^{(1)}$ is a diffeomorphism from $U$ onto $I^+(\mathbf{0})$.
\end{proposition}
\begin{proof}
  We know that the curves are length maximizing, future-directed timelike and geodesics from \Cref{subsection:isopiffgeod} and \Cref{thm:isoperimetric_solution}. The fact that \(\exp_\mathbf{0}^{(1)}\) is bijective follows from the way we defined the exponential map, as we showed that there is a bijection between \((u,v,w) \in \{u> \abs{v} \times \R\}\) and the triples \((x_C,y_C,L)\) with \(y_C > \abs{x_C}\) and \(L>0\); in turn, these triples are in bijection with points of the chronological future of \(\mathbf{0}\) as we discussed at the beginning of this section.

  The fact that the map is a diffeomorphism comes from noticing that it is smooth and that by \Cref{lemma:jacobian_exponential} its Jacobian determinant never vanishes on the set \(\{u > \abs{v}\} \times \R\), hence it is locally invertible with smooth inverse; being bijective, the inverse map is well defined globally and thus smooth.
\end{proof}

A priori, we do not know anything about the exponential map evaluated at negative times. The following statement shows that even in that case we still get maximizing geodesics, so that the exponential map actually yields past-directed geodesics as well:
\begin{proposition}
  The map $\exp^{(-1)}_{\mathbf{0}}$ is a diffeomorphism from $U$ onto $I^-(\mathbf{0})$.
  Moreover, for any $(u,v,w) \in U$ the curve $\exp_{\mathbf{0}}^{(t-1)}(u,v,w)$ is a timelike future-directed maximizing geodesic joining $\exp^{(-1)}_{\mathbf{0}}(u,v,w)$ to the origin.
\end{proposition}
\begin{proof}
  If $p \ll \mathbf{0}$, then left-translation by $-p$ gives $\mathbf{0} \ll -p$, see \cref{eq:leqLeftTranslations}. So $-p = \exp_{\mathbf{0}}^{(1)}(u,v,w)$ for some $(u,v,w) \in U$. We define the function
  \[
    L(u,v,w)
    \coloneqq
    \begin{pmatrix}
      \cosh(w) & \sinh(w) & 0\\
      \sinh(w) & \cosh(w) & 0\\
      0 & 0 & 1
    \end{pmatrix}
    \begin{pmatrix}
      u\\
      v\\
      w
    \end{pmatrix},
  \]
  and denote $(u', v', w') \coloneqq L(u,v,w)$. The $2 \times 2$ upper-left block of the matrix is a future-preserving Lorentzian isometry of the Minkowski plane. Therefore, the vector $(u',v')$ is future-directed and timelike. In particular, $L(U) \subset U$. It is easy to see that $L$ is a smooth diffeomorphism, with inverse
  \[
    R(u',v',w') \coloneqq
    \begin{pmatrix}
      \cosh(w') & -\sinh(w') & 0\\
      -\sinh(w') & \cosh(w') & 0\\
      0 & 0 & 1
    \end{pmatrix}
    \begin{pmatrix}
      u'\\
      v'\\
      w'
    \end{pmatrix}.
  \]
  A simple computation shows that
  \[
    (\exp_{\mathbf{0}}^{(-1)} \circ L)(u,v,w) = - \exp_{\mathbf{0}}^{(1)}(u,v,w) = -(-p) = p.
  \]
  Precomposing both sides with $R$ gives, for any $(u',v',w') \in U$,
  \[
    \exp_{\mathbf{0}}^{(-1)}(u',v',w') = - \exp_{\mathbf{0}}^{(1)}(R(u',v',w')),
  \]
  which directly implies that $\exp_{\mathbf{0}}^{(-1)}$ is a smooth diffeomorphism and that its image coincides with $I^-(\mathbf{0})$.
  A slightly more involved calculation shows that for any $(u',v',w') \in U$ and for $p = \exp^{(-1)}(u',v',w')$, it holds that
  \[
    \exp_{\mathbf{0}}^{(t-1)}(u',v',w') = p \ast \exp^{(t)}_{\mathbf{0}}(R(u',v',w')).
  \]
  This identity implies that the curve $t \mapsto \exp^{(t - 1)}_{\mathbf{0}}(u',v',w')$ is a length-maximizing, future-directed geodesic from $p$ to $\mathbf{0}$, since left translations are future-preserving isometries.
\end{proof}

We conclude this section with a short discussion about the {\em cut locus} of points in the Heisenberg group. For any curve \(\gamma\) in \(\heis\) starting at \(t=0\), we say that it is a {\em locally maximizing geodesic} if for any time \(t\) in the domain of \(\gamma\) there exists a neighbourhood on which the restriction of the curve is a maximizing geodesic; we also define its \emph{cut time} as
\[
  t_{\mathrm{cut}}[\gamma] = \sup \{t  \mid \gamma|_{[0,t]} \text{ is a maximizing geodesic}\}.
\]
We then define a {\em cut-reaching} locally maximizing geodesic to be a locally maximizing geodesic \(\gamma\) whose domain is \([0, t_{\mathrm{cut}}[\gamma]]\) and such that there exists no extension to a locally maximizing geodesic \(\sigma\) starting at \(t=0\) with \(t_{\mathrm{cut}}[\gamma] < t_{\mathrm{cut}}[\sigma]\).
The cut locus at $q \in \heis$ is the set
\[
  \mathrm{Cut}(q) \coloneqq \left\{ \gamma(t_{\mathrm{cut}}[\gamma]) \mid \gamma \text{ is a cut-reaching locally maximizing geodesic starting at } q \right\},
\]
From the characterization of geodesics as lift of the isoperimetric Minkowskian problem, we can immediately prove the following proposition:
\begin{proposition}\label{prop:cut_is_empty}
  For any \(q \in \heis\), it holds that \(\mathrm{Cut}(q) = \emptyset\). Equivalently, every locally maximizing geodesic \(\gamma\) starting from \(q\) is actually a maximizing geodesic.
\end{proposition}
\begin{proof}
  Suppose on the contrary that there exists a cut-reaching locally maximizing curve \(\gamma\) starting at \(q\) and a point \(q'\) such that
  \[
    \gamma(t_{\mathrm{cut}}[\gamma]) = q'.
  \]
  We can assume, up to translations, that \(q = \mathbf{0}\). We know that \(\gamma\) is defined on \([0,t_{cut}[\gamma]]\) and that any extension cannot have greater cut time. Since on said interval it is length maximizing, by uniqueness it must be equal to the unique geodesic from \(q\) and \(q'\). If it is lightlike (either a straight line or a broken one), we can just extend it linearly for an arbitrary amount of time; it follows from \Cref{subsection:variationalisopMink} that this extension is the unique causal curve between its endpoints, so it is length maximizing, which contradicts the maximality assumption. If it is timelike, then it coincides with the curve \(t \mapsto \exp^{(t)}(u,v,w)\) for some \((u,v,w) \in \{u > \abs{v}\} \times \R \). However, from \Cref{prop:expdiffeom} we know that such curve is a length maximizing geodesic for all times and thus in particular for \(t > t_{cut}[\gamma]\), hence we get a contradiction. We conclude that \(q'\) does not exist and so \(\mathrm{Cut}(q) = \emptyset\).
\end{proof}
\begin{remark}
  The fact that \(\mathrm{Cut}(q) = \emptyset\) highlights a substantially different behavior of geodesics in the sub-Lorentzian Heisenberg group from those of the sub-Riemannian one: in such case, the cut locus associated to the origin coincides with the $z$-axis. This will also play a crucial role in \Cref{sect:heis_tcd_tmcp} where we will show that the \(\tmcp{K}{N}\) condition cannot hold for any choice of \(K \in \R\) and \(N\geq 1\), since we can use the exponential map to control the transport of volumes at any distance from the origin.
\end{remark}

The authors of \cite{sachkovsachkova23} obtained further results based on the structure of geodesics in the sub-Lorentzian Heisenberg group, which they derived via Pontryagin's maximum principle. Their arguments apply here without modification; the only difference is that in our setting the geodesics are obtained from their characterization through the Lorentzian isoperimetric problem in the Minkowski plane.

\section{Synthetic timelike Curvature-Dimension bounds}

In this section, we present the general framework of measured Lorentzian pre-length spaces and Lorentzian optimal transport, and explain how the sub-Lorentzian Heisenberg group fits into this context.

\subsection{Lorentzian metric spacetime and optimal transport}
\label{section:lorentzianmetricspacetime}
Here, we rely mainly on \cite{kunzingersaemann2018, cavallettimondino2024} for our exposition.
\begin{definition}
  \label{def:lorentzianprelengthspace}
  A {\em Lorentzian pre-length space} $(X, \dis, \ll,\leq, \uptau)$ consists of a metric space $(X, \dis)$ equipped with a preorder $\leq$, a transitive relation $\ll$ contained in $\leq$, and with a lower semicontinuous function $\uptau \colon X \times X \to [0, +\infty]$ such that for all $x, y, z \in X$
  \begin{enumerate}[label=\normalfont(\roman*), topsep=4pt,itemsep=4pt,partopsep=4pt, parsep=4pt]
    \item $\uptau(x, y) = 0$ if $x \not \leq y$;
    \item $\uptau(x, y) > 0$ if and only if $x \ll y$;
    \item $\uptau(x, y) \geq \uptau(x, y) + \uptau(y, z)$ if $x \leq y \leq z$.
  \end{enumerate}
  The map $\uptau$ is referred to as the {\em time-separation function}.
\end{definition}
We endowed $X$ with the metric topology. The relations defining a Lorentzian pre-length space $(X, \dis, \ll,\leq, \uptau)$ encode the concept of causality. We say that $x$ and $y$ in $X$ are {\em causally} (resp. {\em timelike}) related if $x \leq y$ (resp. $x \ll y$). We adopt the same notation as in \cref{eq:causalchronologicalfuturedefinition,eq:causalchronologicaldiamonddefinition,eq:notationA2}, with the understanding that it is now interpreted in the setting of $X$ in place of $\mathbb{H}$.

A curve $\gamma \colon I \to X$, defined on an interval $I$, is said to be ({\em future-directed}) {\em causal} (resp. {\em timelike}) if it is locally Lipschitz with respect to the metric $\dis$ and if for any $s,t \in I$ with $s < t$, it holds that $\gamma(s) \leq \gamma(t)$ (resp. $\gamma(s) \ll \gamma(t)$). A {\em null} (or {\em lightlike}) curve is a causal curve for which no pair of times $s < t$ satisfy $\gamma(s) \ll \gamma(t)$. The Lorentzian length of a causal curve $\gamma : I \to X$ is defined by
\begin{equation}
  \label{eq:L=Ltau}
  L(\gamma) \coloneqq \inf \sum_{i = 0}^{N - 1}\uptau(\gamma(t_i), \gamma(t_{i + 1})),
\end{equation}
where the infinimum is taken over all $N \in \mathbb{N}$ and all finite partitions $\{t_i\}_{i = 0, \dots, N}$ of $I$ satisfying $t_0 < t_1 < \cdots < t_N$. The identity \cref{eq:L=Ltau} is reminiscent of the notion of length induced by a metric $\dis$:
\begin{equation}
  \label{eq:L=Ldis}
  L_{\dis}(\gamma) \coloneqq \sup \sum_{i = 0}^{N - 1}\dis(\gamma(t_i), \gamma(t_{i + 1})).
\end{equation}

\begin{definition}[The causal ladder]
  \label{def:causalladder}
  A Lorentzian pre-length space $(X, \dis, \ll,\leq, \uptau)$ is
  \begin{enumerate}[label=\normalfont(\roman*), topsep=4pt,itemsep=4pt,partopsep=4pt, parsep=4pt]
    \item \emph{strongly causal} if the so-called Alexandrov topology obtained by taking $\{ I^+(x) \cap I^-(y) \mid x, y \in X \}$ as a subbase coincides with the metric topology;
    \item \emph{locally causally closed} if every point \(x \in X\) admits a neighbourhood \(U\) such that the set given by
      \[
        \leq_{\vert\bar{U} \times \bar{U}} \ \subset \ \bar{U} \times \bar{U}
      \]
      is closed in \(\bar{U} \times \bar{U}\);
    \item \emph{causally closed} if $X^2_\le$ is a closed set ($\leq$ is a closed relation in the product topology);
    \item \emph{$\dis$-compatible} if for every $x \in X$, there exists a neighbourhood $\mathcal{U}$ of $x$ and a constant $C > 0$ such that $\L_{\dis}(\gamma) \leq C$ for all causal curves $\gamma$ with values in $\mathcal{U}$;
    \item \emph{non-totally imprisoning} if for all compact subsets $K$ of $X$, there is a uniform bound $C > 0$ for the $\dis$-lengths of causal curves $\gamma$ with values in $K$;
    \item \emph{globally hyperbolic} if it is non-totally imprisoning and also the sets $J^+(x) \cap J^-(y)$ are compact for all $x,y\in X$;
    \item $\mathcal{K}$-\emph{globally hyperbolic} if it is non-totally imprisoning and also the sets $J^+(K_1) \cap J^-(K_2)$ are compact for all compact subsets $K_1, K_2 \subseteq X$.
  \end{enumerate}
\end{definition}

One the one hand, a Lorentzian pre-length space $(X, \dis, \ll,\leq, \uptau)$ is said to be {\em Lorentzian length space} if for all $x, y \in X$ such that $x \leq y$, it holds
\[
  \uptau(x, y) = \sup\{L(\gamma) \mid \gamma : [t_0, t_1] \to X \text{ causal, } \gamma(t_0) = x, \ \gamma(t_1) = y\},
\]
if it is locally causally closed, $I^{\pm}(x) \neq \emptyset$ for all $x \in X$ and if it is timelike path connected.

One the other hand, a Lorentzian pre-length space $(X, \dis, \ll,\leq, \uptau)$ is said to be {\em Lorentzian geodesic space} if it is \emph{localizable} in the sense of \cite[Definition 3.16]{kunzingersaemann2018},  is $\dis$-compatible, and if for all $x, y \in X$ such that $x \ll y$, there exists a {\em maximizer} joining $x$ to $y$, i.e. a future-directed causal curve $\gamma : [t_0, t_1] \to X$ such that $\uptau(x, y) = L(\gamma)$.

The set $\mathrm{Geo}(X)$ of {\em geodesics} contains the continuous maximizers defined on $[0, 1]$ and parametrized proportionally with respect to the Lorentzian length, that is to say
\[
  \mathrm{Geo}(X) \coloneqq \{ \gamma \in C([0, 1], X) \mid \uptau(\gamma(s), \gamma(t)) = (t - s) \uptau(\gamma(0), \gamma(1)), \text{ for all } s < t \}.
\]
The set of timelike geodesics is given by
\[
  \mathrm{TGeo}(X) \coloneqq \{ \gamma \in \mathrm{Geo}(X) \mid \gamma(s) \ll \gamma(t), \text{ for all } s < t \}.
\]
A Lorentzian pre-length space is said to be {\em timelike non-branching} if, for any $\gamma_1, \gamma_2 \in \mathrm{TGeo}(X)$, we have that $\gamma_1 = \gamma_2$ whenever there is $t \in (0, 1)$ such that  $\gamma_1|_{[0, t]} = \gamma_2|_{[0, t]}$ for all $s \in [0, t]$. It will be useful to consider the {\em evaluation map}
\[
  e_t : \mathrm{Geo}(X) \to X : \gamma \mapsto \gamma(t), \quad \forall t \in [0, 1].
\]

The time-separation function of globally hyperbolic Lorentzian length space $(X, \dis, \ll,\leq, \uptau)$ is finite and continuous. It follows that maximizers have a constant $\uptau$-speed parametrizations, and thus any two distinct causally related points are joined by a causal geodesi: a globally hyperbolic Lorentzian length space is a Lorentzian geodesic space.

As we turn our attention to optimal transport, we define a {\em measured Lorentzian pre-length space} $(X, \dis, \m, \ll, \leq, \uptau)$ to be a Lorentzian pre-length space equipped with a Radon measure $\m$. We will use the notation $\Prob(X)$, $\Prob_{\mathrm{c}}(X)$, $\Prob_{\mathrm{c}}^{\mathrm{ac}}(X, \m)$ for the set of Borel probability measures, the set of compactly supported probability measures and the set of probability measures absolutely continuous with respect to $\m$, respectively.

The set of \emph{transport plans} between two probability measures $\mu, \nu \in \Prob(X)$ is given by
\[
  \Pi(\mu, \nu) \coloneqq \left\{ \pi \in \Prob(X \times X) \mid (P_1)_\sharp \pi = \mu, (P_2)_\sharp \pi = \nu \right\},
\]
where $P_1,P_2:X \times X \to X$ are the projections onto the respective coordinate, i.e. $P_1(x, y) \coloneqq x$ and $P_2(x, y) \coloneqq y$, for all $x, y \in X$. We also introduce the set of {\em causal}  (respectively {\em chronological}) transport plans from $\mu$ to $\nu$ as
\[
  \Pi_\leq(\mu, \nu) \coloneqq \left\{ \pi \in \Pi(\mu,\nu) \mid \pi(X^2_{\leq}) = 1 \right\}, \ \ \Pi_\ll(\mu, \nu) = \left\{ \pi \in \Pi(\mu,\nu) \mid \pi(X^2_{\ll}) = 1 \right\}.
\]
For $p \in (0, 1]$, the $p$-Lorentz-Wasserstein distance between $\mu$ and $\nu$ is defined as
\begin{equation}
  \label{eq:LorentzWasserstein}
  \ell_p(\mu, \nu) \coloneqq \left[ \sup_{\pi \in \Pi_\leq(\mu, \nu)} \int_{X \times X} \uptau(x, y)^p \diff \pi(x, y) \right]^{1/p}, \quad \text{ if } \Pi_\le(\mu,\nu) \neq \emptyset.
\end{equation}
By convention, we set $\ell_p(\mu,\nu) = -\infty$ if $\Pi_\le(\mu,\nu) = \emptyset$. A plan $\pi$ realizing the supremum in \cref{eq:LorentzWasserstein} is called an \emph{optimal transport plan} from $\mu$ to $\nu$. The set of optimal plans between $\mu$ and $\mu$ will be denoted by $\Pi_{\leq}^{p, \mathrm{opt}}(\mu, \nu)$. We say that a pair $(\mu, \nu) \in \Prob(X)$ is {\em timelike $p$-dualizable} if the set
\[
  \Pi_{\ll}^{p, \mathrm{opt}}(\mu, \nu) \coloneqq \Pi_{\leq}^{p, \mathrm{opt}}(\mu, \nu) \cap \Pi_{\ll}(\mu, \nu)
\]
is not-empty and if there exist $a, b : X \to \R$ with $a \oplus b \in L^1(\mu \oplus \nu)$ and $\uptau(x, y)^p \leq a(x) + b(y)$ for all $(x, y) \in \mathrm{supp}(\mu) \times \mathrm{supp}(\nu)$. An element of this set is said to {\em $p$-dualize} $(\mu, \nu)$.

\begin{remark}
  \label{eq:compactsupportpdualizable}
  If the Lorentzian pre-length space is globally hyperbolic and $(\mu, \nu) \in \Prob_{\mathrm{c}}(X)$ and $\mathrm{supp}(\mu) \times \mathrm{supp}(\nu) \subset X_{\ll}^2$, then the pair is timelike $p$-dualizable for any $p \in (0, 1]$
  , see \cite[Remark 2.20]{cavallettimondino2024}.
\end{remark}

The Lorentz-Wasserstein distance $\ell_p$ acts as a time-separation function on the space of probability measures $\Prob(X)$. In particular, it satisfies the reverse triangle inequality:
\[
  \ell_p(\mu_0, \mu_1) + \ell_p(\mu_1, \mu_2) \leq \ell_p(\mu_0, \mu_2), \qquad \forall \mu_0, \mu_1, \mu_2 \in \mathcal{P}(X),
\]
where by convention we set $\infty - \infty = - \infty$. We say that $(\mu_t)_{t \in \interval{0}{1}} \subseteq \Prob(X)$ is an {\em $\ell_p$-geodesic} if
\begin{equation}
  \label{eq:Wassersteingeodesic}
  \ell_p(\mu_s, \mu_t) = (t - s) \ell_p(\mu_0, \mu_1), \qquad \text{ for all } t, s \in \interval{0}{1} \text{ with } s \leq t.
\end{equation}
It is a {\em timelike $\ell_p$-geodesic} if $\ell_p(\mu_s, \mu_t) > 0$ for all $s < t$.

For $\mu_0, \mu_1 \in \Prob(X)$, we denote by $\mathrm{OptTGeo}_p(\mu_0, \mu_1)$ the set {\em timelike $p$-optimal dynamical plan}, i.e. the set consisting of all $\nu \in \Prob(\mathrm{TGeo}(X))$ such that $(e_0, e_1)_{\sharp} \nu \in \Pi_{\ll}^{p, \mathrm{opt}}(\mu, \nu)$. It is a standard fact that if $\nu \in \Prob(\mathrm{OptTGeo}(X))$, then $\mu_t \coloneqq (e_t)\sharp \nu$ is a timelike $\ell_p$-geodesic; and vice versa, any timelike $\ell_p$-geodesic $(\mu_t)_{t \in [0, 1]}$ can be lifted to a $p$-optimal dynamical plan $\nu \in \mathrm{OptTGeo}_p(\mu_0, \mu_1)$.

\subsection{\texorpdfstring{The $\mathsf{TCD}$ and $\mathsf{TMCP}$ conditions}{The TCD and TMCP conditions}}
\label{section:TCDTMCP}
We introduce the synthetic curvature-dimension conditions with the $N$-Rényi entropy, following \cite{Braun2023a}.

Given a probability measure $\mu \in \mathscr{P}(X)$, the  entropy $\Ent(\mu|\m)$ is defined by
\[
  \Ent(\mu|\m) \coloneqq \int_{X} \rho \log(\rho) \diff \m,
\]
if $\mu = \rho \m$ is absolutely continuous with respect to $\m$ and $(\rho\log(\rho))_{+}$ is  $\m$-integrable;  otherwise we set $\Ent(\mu|\m) = +\infty$.
For $N \in [1, +\infty)$, the $N$-Rényi entropy $S_{N}(\mu_t | \m)$ is given by
\[
  S_{N}(\mu | \m) \coloneqq \int_X\rho^{1-\frac{1}{N}} \diff \m,
\]
if $\mu = \rho \m$ is absolutely continuous with respect to $\m$;  otherwise we set $S_{N}(\mu | \m) = +\infty$.

We also introduce the distortion coefficients from the Lott--Sturm--Villani theory of synthetic curvature-dimension bounds \cite{LV-Ricci, S-ActaI, S-ActaII}.
Given $K\in \R$, $N \in [1, +\infty)$, we set $\tau_{K, N}^{(t)}(\theta)$ for $(t, \theta) \in [0, 1] \times [0, \infty]$ by
\begin{equation*}
  \tau_{K, N}^{(t)}(\theta) \coloneqq
  \begin{cases}
    +\infty                                                                                                                                & K\theta^2 \geq (N-1)\pi^2 \text{ and } K > 0, \\
    t^{\frac 1 N} \left(\frac{\sin\left(t\theta \sqrt{K/(N-1)}\right)}{\sin\left(\theta \sqrt{K/(N-1)}\right)}\right)^{1 - \frac 1 N}       & 0<K\theta^2<(N-1)\pi^2,    \\
    t                                                                                                                                       & K\theta^2 =0, \text{ or } K<0 \text{ and } N=1,              \\
    t^{\frac 1 N} \left(\frac{\sinh\left(t\theta \sqrt{|K|/(N-1)}\right)}{\sinh\left(\theta \sqrt{|K|/(N-1)}\right)}\right)^{1 - \frac 1 N} & K\theta^2<0.
  \end{cases}
\end{equation*}

The following definition, extending the well-known curvature-dimension condition from metric geometry to the Lorentzian setting, is found in \cite[Definition 3.2]{Braun2023a}, together with its infinite-dimensional analogue in \cite[Definition 4.1]{Braun2023}.

\begin{definition}[Timelike curvature-dimension]\label[definition]{def:TCD}
  Let $p \in (0, 1)$, $K \in \R$, and $N \in [1, +\infty]$. A measured Lorentzian pre-length space $(X, \dis, \m, \leq, \ll, \uptau)$ satisfies the timelike curvature-dimension condition $\TCD_p(K, N)$ if for every timelike $p$-dualizable pair $(\mu_0, \mu_1) = (\rho_0 \m, \rho_1 \m) \in \mathscr{P}_{\mathrm{c}}^{\mathrm{ac}}(X, \m)$, there exists $\nu \in \mathrm{OptTGeo}_p(\mu_0, \mu_1)$ such that for all $t \in [0, 1]$ and $N' \geq N$, we have
  \[
    S_{N'}(\mu_t | \m) \leq - \int_{X \times X}  \tau_{K, N'}^{(1 - t)}(\uptau(x, y)) \rho_0(x)^{-1/N'} + \tau_{K, N'}^{(t)}(\uptau(x, y)) \rho_1(y)^{-1/N'} \diff \pi(x, y),
  \]
  if $N < +\infty$, and, if $N = +\infty$,
  \[
    \Ent(\mu_t | \m) \leq (1 - t) \Ent(\mu_0|\m) + t \Ent(\mu_1 | \m) - \frac{K}{2} t(1 - t) ||\uptau||_{L^2(\pi)},
  \]
  where $\mu_t \coloneqq (e_t)_\sharp \nu$ and $\pi \coloneqq (e_0, e_1)_\sharp \nu$.
\end{definition}

Of course, one can show that the condition $\mathsf{TCD}_p(K, N)$ for a fixed $N \in [1, +\infty)$ implies $\mathsf{TCD}_p(K, +\infty)$, see \cite[Lemma 3.8]{Braun2023a}. The main motivation behind the definition of the $\mathsf{TCD}$ conditions is that it provides a complete characterization of lower bounds on Ricci curvature and upper bounds on dimension in the smooth Lorentzian setting. As shown in \cite[Theorem 3.1]{cavallettimondino2024}, \cite[Theorem 3.35, Theorem A.1]{Braun2023a} and \cite[Theorem 2.15]{clemenscones},
the following are equivalent on a Lorentzian manifold $(M, g)$ equipped with its natural causal structure, time-separation function $\uptau$, and volume measure $\mathsf{vol}_g$:
\begin{enumerate}[label=\normalfont(\roman*), topsep=4pt,itemsep=4pt,partopsep=4pt, parsep=4pt]
  \item $\mathrm{Ric}_g(v, v) \geq - K g(v, v)$ for every timelike $v \in \T(M)$, and $\mathrm{dim}(M) \leq N$;
  \item $(M, \dis_g, \mathsf{vol}_g, \ll, \leq, \uptau)$ satisfies $\mathsf{TCD}_{p}(K,N)$.
\end{enumerate}
In other words, \cref{def:TCD} generalizes the concept of Ricci lower bounds to abstract measured Lorentzian pre-length spaces.

Given  $A_0,A_1\subset X$ and $t\in [0,1]$, we define the set of timelike $t$-intermediate points between $A_0$ and $A_1$ as
\begin{equation}
  \label{eq:tintermediate}
  Z_t(A_0, A_1) \coloneqq A_t \coloneqq \{\gamma_t \mid \gamma\in \mathrm{TGeo}(X),\, \gamma_0 \in A_0,\, \gamma_1\in A_1\}.
\end{equation}
Furthermore, we also set
\begin{align}\label{Eq:THETA}
  \Theta \coloneqq
  \begin{cases} \sup\uptau(A_0\times A_1) & \textnormal{if }K<0,\\
    \inf\uptau(A_0\times A_1) & \textnormal{otherwise}.
  \end{cases}
\end{align}

\begin{definition}[Brunn-Minkowki inequality]
  \label[definition]{def:Brunn-Minkowski}
  Let $p \in (0, 1)$, $K \in \R$, and $N \in [1, +\infty]$. A measured Lorentzian pre-length space $(X, \dis, \m, \leq, \ll, \uptau)$ satisfies the timelike Brunn-Minkowki inequality $\mathsf{TBM}_p(K, N)$ if for every relatively compact Borel sets $A_0,A_1\subset X$ with  $\m(A_0)\,\m(A_1) >0$ such that
  \begin{align*}
    (\mu_0,\mu_1) \coloneqq (\m(A_0)^{-1}\,\m\mres A_0,\m(A_1)^{-1}\, \m\mres A_1).
  \end{align*}
  is timelike $p$-dualizable, it holds that for every $t\in [0,1]$ and every $N'\geq N$,
  \begin{equation}
    \label{Eq:BM}
    \m(A_t)^{1/N'} \geq \tau_{K,N'}^{(1-t)}(\Theta)\,\m(A_0)^{1/N'} + \tau_{K,N'}^{(t)}(\Theta)\,\m(A_1)^{1/N'},
  \end{equation}
  if $N < +\infty$, and, if $N = +\infty$,
  \begin{equation}
    \label{Eq:BMinfty}
    \begin{split}
      \log\left[ \frac{1}{\m(A_t)} \right] \leq (1 - t) \log\left[ \frac{1}{\m(A_0)} \right] &+ t \log\left[ \frac{1}{\m(A_1)} \right] - \frac{K t (1 - t)}{2} \inf_{(x, y) \in A_0 \times A_1} \uptau(x, y)^2.
    \end{split}
  \end{equation}
\end{definition}

It is well-known that the curvature-dimension condition implies the Brunn-Minkowski inequality, see \cite[Proposition 3.11]{Braun2023a} (and also \cite[Proposition 3.4]{cavallettimondino2024} for the entropic version).

\begin{proposition}[$\mathsf{TCD}_{p}(K,N) \Rightarrow \mathsf{TBM}_p(K,N)$]\label{prop:TCDtoTBM}
  Let $p \in (0, 1)$, $K \in \R$, and $N \in [1, +\infty]$. The  $\mathsf{TCD}_{p}(K,N)$ condition implies $\mathsf{TBM}_p(K,N)$ for measured Lorentzian pre-length spaces.
\end{proposition}

Next, we define the timelike measure contraction property, a variant of the timelike-curvature dimension. This condition was introduced in various forms in see \cite[Definition 3.7]{cavallettimondino2024}, \cite[Definition 4.1]{Braun2023a}, or \cite[Definition 2.12]{clemenscones}, for example.

\begin{definition}[Timelike measure-contraction properties]\label{Def: TMCP}
  Let $p\in(0,1)$, $K\in\R$, and $N\in[1, +\infty)$. We say that $(X, \dis, \m, \ll, \leq, \uptau)$ satisfies the \emph{future timelike measure contraction property} $\mathsf{TMCP}^+_p(K,N)$ if for every $\mu_0=\rho_0\,\m\in\mathscr{P}^{\mathrm{ac}}_{\mathrm{c}}(X)$ and every $x_1\in X$ with $\mu_0(I^-(x_1))=1$, there exists $\nu \in \mathrm{OptTGeo}_p(\mu_0, \mu_1)$ such that for all $t \in [0, 1]$ and $N' \geq N$, we have
  \begin{equation}
    \begin{aligned}\label{Eq:TMCP+}
      \int_X\rho_t^{1-\frac{1}{N'}} \diff \m\geq\int_{X^2}\tau_{K,N'}^{(1-t)}(\uptau(x,x_1))\rho_0(x)^{-\frac{1}{N'}} \diff \mu_0(x), \qquad \forall t \in [0, 1),
    \end{aligned}
  \end{equation}
  where we have written $(e_t)_\sharp \nu = \rho_t \m + \mu_t^{\mathrm{sing}}$ with $\mu_t^{\mathrm{sing}} \perp \m$. We say that $X$ satisfies the \emph{past} $\mathsf{TMCP}_p^-(K,N)$ if \cref{Eq:TMCP+} holds with $\mu_1=\rho_1\,\m$ and $\mu_0=\delta_{x_0}$ for some $x_0\in X$ with $\mu_1(I^+(x_0))=1$, i.e.,
  \begin{equation}
    \begin{aligned}\label{Eq:TMCP-}
      \int_X\rho_t^{1-\frac{1}{N'}}\,\diff \m\geq\int_{X^2}\tau_{K,N'}^{(t)}(\uptau(x_0,x))\rho_1(x)^{-\frac{1}{N'}} \diff \mu_1(x).
    \end{aligned}
  \end{equation}
  Finally, we say that $X$ satisfies the \emph{timelike measure contraction property} $\mathsf{TMCP}_p(K,N)$ if it satisfies both the conditions $\mathsf{TMCP}_p^+(K,N)$ and $\mathsf{TMCP}_p^-(K,N)$.
\end{definition}

The timelike measure contraction property is a \emph{weaker} variant of the timelike curvature-dimension condition, as shown in \cite[Proposition 4.9]{Braun2023a}.

\begin{proposition}[$\mathsf{TCD}_{p}(K,N) \Rightarrow \mathsf{TMCP}_p(K,N)$]\label{prop:TCDtoTMCP}
  Fix $p\in (0,1)$, $K\in \R$, $N\in [1,\infty)$. The $\mathsf{TCD}_{p}(K,N)$ condition implies $\mathsf{TMCP}_p(K,N)$ for globally hyperbolic Lorentzian geodesic spaces.
\end{proposition}

The $\mathsf{TMCP}$ is also compatible with the smooth Lorentzian setting. Indeed, given a globally hyperbolic Lorentzian manifold $(M, \dis_g, \mathsf{vol}_g, \ll, \leq, \uptau)$ with $\dim M = n$, then by \cite[Theorem A.4]{Braun2023a} and \cite[Appendix A]{cavallettimondino2024}:

\begin{enumerate}[label=\normalfont(\roman*), topsep=4pt,itemsep=4pt,partopsep=4pt, parsep=4pt]
  \item $(M, \dis_g, \mathsf{vol}_g, \ll, \leq, \uptau)$ satisfies $\mathsf{TMCP}_{p}(K,n)$ if and only if $\mathrm{Ric}_g(v, v) \geq - K g(v, v)$ for every timelike $v \in \T(M)$;
  \item If $(M, \dis_g, \mathsf{vol}_g, \ll, \leq, \uptau)$ satisfies $\mathsf{TCD}_{p}(K,N)$, then $n \leq N$.
\end{enumerate}

The timelike measure contraction property is related to a kind of ``half Brunn-Minkoswki inequality'', when one of the two sets in \cref{def:Brunn-Minkowski} is a point.

\begin{definition}[Half Brunn-Minkowki inequality]
  \label{def:halfBrunn-Minkowski}
  Let $p \in (0, 1)$, $K \in \R$, and $N \in [1, +\infty)$. A measured Lorentzian pre-length space $(X, \dis, \m, \leq, \ll, \uptau)$ satisfies the \emph{future} (resp. \emph{past}) half timelike Brunn-Minkowki inequality $\mathsf{TBM}^{1/2+}_p(K, N)$ (resp. $\mathsf{TBM}^{1/2-}_p(K, N)$) if for every relatively compact Borel sets $A\subset X$ with  $\m(A) >0$ and every $x \in I^+(A)$ (resp. $x \in I^-(A)$) such that
  \begin{align*}
    (\mu_0,\mu_1) \coloneqq (\m(A)^{-1}\,\m\mres A,\delta_{x}) \qquad \text{  (resp. $(\delta_{x}, \m(A)^{-1}\,\m\mres A)$)}
  \end{align*}
  is timelike $p$-dualizable, it holds that, for every $t\in [0,1)$,
  \begin{equation}
    \label{Eq:halfBM}
    \m(A_t)^{1/N} \geq \tau_{K,N'}^{(1-t)}(\Theta)\,\m(A)^{1/N} \qquad \text{  (resp. $\m(A_t)^{1/N} \geq \tau_{K,N}^{(t)}(\Theta)\,\m(A)^{1/N}$)}.
  \end{equation}
  We say that $X$ satisfies the \emph{half timelike Brunn-Minkowski inequality} $\mathsf{TBM}^{1/2}_p(K,N)$ if it satisfies both the conditions $\mathsf{TBM}^{1/2+}_p(K, N)$ and $\mathsf{TBM}^{1/2-}_p(K, N)$.
\end{definition}

The following implications can be easily proven, with an argument nearly identical to \cite[Proposition 3.11]{Braun2023a}.

\begin{proposition}[$\mathsf{TMCP}_{p}(K,N) \Rightarrow \mathsf{TBM}^{1/2}_p(K,N)$]\label{prop:TMCP->TBM1/2}
  Let $p \in (0, 1)$, $K \in \R$, and $N \in [1, +\infty)$. The  $\mathsf{TMCP}_{p}(K,N)$ condition (resp. $\mathsf{TMCP}^+_{p}(K,N)$, $\mathsf{TMCP}^-_{p}(K,N)$) implies $\mathsf{TBM}^{1/2}_p(K,N)$ (resp. $\mathsf{TBM}^{1/2+}_{p}(K,N)$, $\mathsf{TBM}^{1/2-}_{p}(K,N)$) for measured Lorentzian pre-length spaces.
\end{proposition}

We have the following properties when a measured Lorentzian pre-length space $(X, \dis, \m, \leq, \ll, \uptau)$ satisfies the $\mathsf{TCD}_p(K, N)$ condition (resp. the $\mathsf{TMCP}_p(K, N)$):

\begin{enumerate}[label=\normalfont(\roman*), topsep=4pt,itemsep=4pt,partopsep=4pt, parsep=4pt]
  \item \textbf{Consistency:} $(X, \dis, \m, \leq, \ll, \uptau)$ satisfies $\mathsf{TCD}_p(K', N')$ (resp. $\mathsf{TMCP}_p(K', N')$) for every $K' \leq K$ and $N' \geq N$;
  \item \textbf{Scaling:} The rescaled measured Lorentzian pre-length space $(X, a \cdot \dis, b \cdot \m, \leq, \ll, c \cdot \uptau)$ satisfies $\mathsf{TCD}_p(K/c^2, N)$ (resp. $\mathsf{TMCP}_p(K/c^2, N)$) for every $a, b, c > 0$.
\end{enumerate}

\subsection{The metric spacetime structure of the Heisenberg group}
\label{section:compatibility}

Recall that the \emph{sub-Riemannian structure} on $\heis$ is induced by considering the Riemannian metric $g$ on $\Delta$ uniquely determined by the conditions
\[
  g(X,X)=1, \quad g(X,Y) = 0, \quad g(Y,Y) = 1.
\]
The \emph{sub-Riemannian distance} is then obtained as
\[
  \dis(p, q) \coloneqq \inf\left\{ L_{\mathrm{sR}}(\gamma) \mid \gamma : [0, T] \to \heis \text{ is horizontal, joining $p$ to $q$} \right\},
\]
where
\[
  L_{\mathrm{sR}}(\gamma) \coloneqq \int_0^T \sqrt{g_{\mathrm{sR}}(\dot \gamma(s), \dot \gamma(s))} \diff s
\]
is the \emph{sub-Lorentzian length} of a horizontal curve $\gamma$. We refer the reader to \cite{Agrachev2020} for a comprehensive introduction to sub-Riemannian geometry and the Heisenberg group.

The sub-Lorentzian Heisenberg group, equipped with the relations $\leq$, $\ll$ and the time-separation function $\uptau$ introduced in \cref{section:subLgeometry}, and with the sub-Riemannian distance $\dis$ is a Lorentzian pre-length space. More precisely, we have the following.

\begin{theorem}
  The sub-Lorentzian Heisenberg group $(\heis, \dis, \ll, \leq, \uptau)$ is a $\mathcal{K}$-globally hyperbolic and timelike non-branching Lorentzian geodesic space.
\end{theorem}

\begin{proof}
  This result first appeared in \cite[Section 2.3]{borza2025}, where the argument relies on the characterization of geodesics via Pontryagin's Maximum Principle, as initiated in \cite{sachkovsachkova23}. We now briefly explain how the theorem also follows from our approach through the Minkowskian isoperimetric problem. That the structure $(\heis, \dis, \ll, \leq, \uptau)$ is a Lorentzian pre-length space follows directly from the construction in \cref{section:subLgeometry}, except for the ``if'' clause in \cref{def:lorentzianprelengthspace}. Specifically, one must rule out the existence of a pair of points that are positively $\uptau$-separated but cannot be connected by a timelike curve. This is excluded by \Cref{prop:future_set_origin}, which asserts that any two points at positive distance can indeed be joined by a timelike curve. Moreover, \cref{maintheorem1:geodesicssubLHeisenberg} shows that $(\heis, \dis, \ll, \leq, \uptau)$ is a timelike non-branching geodesic space. The proof of $\mathcal{K}$-global hyperbolicity proceeds word for word as in \cite[Proposition 18]{borza2025}, relying on the explicit expressions for the causal and chronological diamonds in \cref{eq:causalfutureof0,eq:chronologicalfutureof0}. The only remaining property to verify is localizability, which holds trivially since the space is $\dis$-compatible, $\uptau$ is continuous by $\mathcal{K}$-global hyperbolicity together with the proof of \cite[Theorem 3.28]{kunzingersaemann2018}, and length maximizers between causally related points always exist.
\end{proof}

The bridge between the synthetic and smooth settings is provided by \cite[Proposition 19]{borza2025}: a curve $\gamma : [0, T] \to \heis$ is future-directed causal (resp. timelike) if and only if it is locally Lipschitz with respect to $\dis$ and causal (resp. timelike) in the synthetic sense, i.e. $\gamma(s) \leq \gamma(t)$ (resp. $\gamma(s) \ll \gamma(t)$) for all $s, t \in [0, T]$ with $s < t$. The following theorem improves this result.

\begin{theorem}
  Let $\gamma : [0, T] \to \heis$ be a causal
  curve in the synthetic sense, i.e. $\gamma(s) \leq \gamma(t)$
  for all $s, t \in [0, T]$ with $s < t$. Then, $\gamma$ is differentiable almost everywhere and at every differentiability time $t \in [0, T]$, either $\dot \gamma(t)$ is a (horizontal) future-directed causal
  vector or it vanishes. If $\gamma$ is also continuous, then there exists a reparametrization that is Lipschitz in charts, so this reparametrization of $\gamma$ is causal
  in the sub-Lorentzian sense and locally Lipschitz with respect to $\dis$.
\end{theorem}

\begin{proof}

  Denote by $g$ the Lorentzian metric on $\heis$ satisfying $g(X, X) = -1$, $g(Y, Y) = 1$ and $g(Z, Z) = 1$. The Lorentzian metric $g$ naturally extends the sub-Lorentzian metric. It is not difficult to see that the Lorentzian manifold $(\heis, g)$ is strongly causal. If $\gamma$ is a causal curve, in the synthetic sense, with respect to the sub-Lorentzian structure $(\heis, \leq, \ll)$, it is also causal in the synthetic sense with respect to the Lorentzian one, i.e. with respect to the causal structure $(\heis, \leq_g, \ll_g)$ induced from $g$. By \cite[Theorem A.1]{octet}, we deduce that $\gamma$ is differentiable at almost every time. The fact that there is a Lipschitz-in-charts reparametrization also follows from this argument using $g$, by \cite[Lemma 2.7]{cavmondmanini}.

  Let \(s\) be a point where \(\gamma\) is differentiable (and thus continuous), and consider any $t > s$. Since $\gamma(s)\le \gamma(t)$ by assumption, left invariance and \cref{eq:causalfutureof0} implies that
  \[
    -\gamma(s)\ast \gamma(t)\in\{x\ge 0\}\cap\{-x^2 + y^2 +4\abs{z}\le 0\},
  \]
  which, written out explicitly, means that
  \begin{equation*}
    -(\gamma_1(t)-\gamma_1(s) )^2 + (\gamma_2(t)-\gamma_2(s))^2 +4 \abs{\gamma_3(t) -\gamma_3(s) + \frac{1}{2} \left(\gamma_1(t)\gamma_2(s) -\gamma_1(s)\gamma_2(t) \right)} \leq 0.
  \end{equation*}
  Dividing this expression by \(t-s\), which is a positive number, does not alter the inequality. We can algebraically manipulate the expression to get
  \begin{equation}\label{eq:ineqcausalsmoothsynthetic}
    \begin{split}
      -(t-s) &\left(  \frac{\gamma_1(t) -\gamma_1(s)}{t-s} \right)^2 + (t-s) \left( \frac{\gamma_2(t) -\gamma_2(s)}{t-s} \right)^2 \\
      &+4 \abs{ \frac{\gamma_3(t) - \gamma_3(s)}{t-s} +  \frac{1}{2} \left( \gamma_2(t)\frac{\gamma_1(t) - \gamma_1(s)}{t-s} - \gamma_1(t) \frac{\gamma_2(t) - \gamma_2(s)}{t-s} \right) } \leq 0.
    \end{split}
  \end{equation}
  The first two terms have a similar structure: they have a first factor \((t-s)\) which goes to zero as \(t \to s^+\), whereas the second factors go to \(\dot{\gamma_1}(s)^2\) and \(\dot{\gamma_2}(s)^2\) respectively, so overall the first two terms vanish. Hence sending \(t \to s^+\) in \cref{eq:ineqcausalsmoothsynthetic} yields that
  \[
    4\abs{\dot{\gamma_3}(s) + \frac{1}{2} \left( \gamma_2(s)\dot{\gamma_1}(s) - \gamma_1(s)\dot{\gamma_2}(s) \right)} \leq 0,
  \]
  which implies, as the absolute value is non negative, that
  \[
    \dot{\gamma_3}(s) = \frac{1}{2} \left( \gamma_1(s)\dot{\gamma_2}(s) - \gamma_2(s)\dot{\gamma_1}(s) \right).
  \]
  But we know this condition is equivalent to horizontality for curves starting at the origin, so the curve is horizontal.

  It remains to prove that $\gamma$ is future-directed and causal. The last term in \cref{eq:ineqcausalsmoothsynthetic} is non-negative, thus
  \[
    -(-\gamma_1(s) + \gamma_1(t))^2 + (-\gamma_2(s) + \gamma_2(t))^2 \leq 0.
  \]
  Dividing by \((t-s)^2\) and sending \(t \to s^+\), we immediately see that
  \[
    -\dot{\gamma_1}(s)^2 + \dot{\gamma_2}(s)^2 \leq 0,
  \]
  which is the needed condition for causality. The point \(-\gamma(s) \ast \gamma(t)\) must be in the set \(\{x \geq 0\}\), so in particular
  \[
    -\gamma_1(s) + \gamma_1(t) \geq 0,
  \]
  with equality if and only if \(\gamma(s) = \gamma(t)\), by \eqref{eq:ineqcausalsmoothsynthetic}.  Dividing by \(t - s\) and sending \(t \to s^+\), we get that
  \[
    \dot{\gamma_1}(s) \geq 0,
  \]
  which means that the derivative of \(\gamma_1\) either vanishes, in which case by causality and horizontality we have that \(\dot{\gamma}(s)=0\), or it is future-directed and thus concludes the proof.
\end{proof}

If we supplement $(\heis, \dis, \ll, \leq, \uptau)$ with a measure $\m$, we obtain a measured Lorentzian structure. The question of which natural measures to consider on the sub-Lorentzian Heisenberg group is addressed in the next section.

\section{Lorentzian Hausdorff measure of the Heisenberg group}
\label{section:hausdorffdimension}

This section is devoted to the study of the Lorentzian Hausdorff measure for Lorentzian pre-length spaces, introduced in \cite{McCann2022}, in the specific case of the sub-Lorentzian Heisenberg group.

\begin{definition}[{\cite[Definitions 2.3 \& 3.1]{McCann2022}}]
  Let $(X, \dis, \m, \leq, \ll, \uptau)$ be a Lorentzian pre-length space, $d \geq 0$, and $\delta > 0$. The {\em $d$-dimensional $\delta$-Lorentzian Hausdorff pre-measure $\mathcal{H}^d_{\uptau,\delta}$} is defined, for any Borel subset $E \subset X$, by
  \[
    \mathcal{H}^d_{\uptau,\delta}(E)  \coloneqq \,  \sup \left\{ \sum_i \omega_d \, \uptau(p_i,q_i)^d \ \big| \ \{J(p_i,q_i)\}_{i \in \mathbb{N}} \ \text{is a cover of} \ E,\, \operatorname{diam}(J(p_i,q_i))<\delta \right\},
  \]
  where \(\omega_d\) is the volume of the unit-diamond in the \(d\)-dimensional Minkowski spacetime. The {\em \(d\)-dimensional Lorentzian Hausdorff $\mathcal{H}^d_\uptau(E)$ measure} is given, for any Borel subset $E \subset X$, by
  \[
    \mathcal{H}^d_\uptau(E) \coloneqq \lim_{\delta \to 0^+} \mathcal{H}^d_{\uptau,\delta}(E).
  \]
  The \emph{Lorentzian Hausdorff dimension} is then defined as
  \[
    \dim_{\uptau}(X) \coloneqq \inf\left\{ d \geq 0 \ \mid \ \mathcal{H}^d_\uptau(X) < +\infty \right\}.
  \]
\end{definition}

\begin{remark}
  By \cite[Proposition 2.4]{McCann2022}, we know that the map $\delta \mapsto \mathcal{H}^d_{\uptau,\delta}(E)$ is monotonically nonincreasing, that $\mathcal{H}^d_{\uptau,\delta}$ is an outer measure, and that $\mathcal{H}^d_\uptau$ is a Borel measure. It is also proven in \cite[Theorem 4.8]{McCann2022} that, on a smooth Lorentzian manifold $(M, g)$, the measure $\mathcal{H}^d_{\uptau,\delta}$ coincides with the natural volume measure $\mathrm{vol}_g$. Further properties of the Lorentzian Hausdorff dimension can be found in \cite[Section 3]{McCann2022}.
\end{remark}

A measure $\m$ on $\heis$ is \emph{left-invariant} if $(L_q)_\sharp \m = \m$ for all $q \in \heis$. On the sub-Lorentzian Heisenberg group, we thus have the following three natural measures:
\begin{enumerate}[label=\normalfont(\roman*), topsep=4pt,itemsep=4pt,partopsep=4pt, parsep=4pt]
  \item the 3-dimensional Lebesgue measure \(\mathcal{L}^3\), which is a Haar measure;
  \item the \(d\)-dimensional Hausdorff measure \(\mathcal{H}^d\) induced by the sub-Riemannian distance \(\dis\), for which it is known that the (metric) Hausdorff dimension is \(\dim_{\dis}(\heis) = 4\) and that for such choice of dimension it is a Haar measure (i.e \(\mathcal{H}^4 = c\mathcal{L}^3\) for some positive constant \(c>0\), see \cite[Section 2.2.3]{Capogna2007} or \cite[Section 3.3]{LeDonneBook});
  \item the \(d\)-dimensional Lorentzian Hausdorff measure \(\mathcal{H}^d_\uptau\).
\end{enumerate}

The aim of this section is to prove the following result.
\begin{theorem}\label{thm:lorentzian_hausdorff}
  The Lorentzian Hausdorff dimension of the sub-Lorentzian Heisenberg group $(\heis, \dis, \leq, \ll, \uptau)$ is $4$. Moreover, the measure \(\hausdorff^4\) is a Haar measure, i.e. there exists a constant \(c>0\) such that \(\hausdorff^4 = c\mathcal{L}^3\).
\end{theorem}
The proof roughly follows the one given for the standard Hausdorff measure in the sub-Riemannian Heisenberg group, see \cite[Theorem 2.1.34]{LeDonneBook}. The main idea is that diamonds can be approximated by sub-Riemannian balls, and vice versa, uniformly with respect to the diameter. In particular, we will prove a Lorentzian \say{ball-box theorem}, which ensures that diamonds grow at the same rate as sub-Riemannian boxes \say{along the time axis}.

We will also explicitly compute the Lebsegue volume of diamonds, which will then lead to an estimate on the volume growth in terms of the fourth power of the time separation function between the vertices of the diamond; this is also an indication of the fact that the correct dimension choice is 4.
\begin{proposition}\label{prop:hausdorff_dimension_must_be_4}
  On the space $(\heis, \dis, \leq, \ll, \uptau)$, the Lorentzian Hausdorff measure $\hausdorff^d$ is a left-invariant measure for any \(d \geq 0\). It is also homogeneous with respect to dilations, i.e. \( \hausdorff^d(\delta_\lambda(E)) = \lambda^d \hausdorff^d(E)\) for any Borel set \(E\) and \( \lambda >0\). If \(\hausdorff^d(B) \in (0, +\infty)\) for some \(d \geq 0\) and for some metric ball \(B\), then \(d=4\) and \(\hausdorff^4\) is a Haar measure.
\end{proposition}
\begin{proof}
  The fact that measure is invariant under left translations easily follows from the definition.
  Indeed, left translations preserve the sub-Riemannian distance and the time separation function, so they map diamonds to diamonds with the same diameter and time separation. The \(\delta\)-pre-measures are thus invariant under left translations, and so is \(\hausdorff^d\).

  To prove homogeneity with respect to dilations, consider a Borel set \(E\) and a positive number \(\lambda > 0\). If \(\delta > 0\) and \(\{J(p_i,q_i)\}_i\) is a \(\delta\)-cover of diamonds of \(E\), then for any \(i \) it holds \(\delta_\lambda(J(p_i,q_i)) = J(\delta_\lambda(p_i), \delta_\lambda(q_i))\) and \(\{\delta_\lambda(J(p_i,q_i))\}_i\) is a \(\delta\lambda\)-cover of \(\delta_\lambda(E)\) with
  \[
    \prehausdorff{\delta\lambda}^d(\delta_\lambda(E)) \leq \sum_i \omega_d \uptau(\delta_\lambda(p_i),\delta_\lambda(q_i))^d = \lambda^d \sum_i \omega_d \uptau(p_i,q_i)^d.
  \]
  Taking the infimum over all \(\delta\)-covers of diamonds of \(E\), we get
  \[
    \prehausdorff{\lambda\delta}^d(\delta_\lambda(E)) \leq \lambda^d \prehausdorff{\delta}^d(E).
  \]
  Moreover, taking the limit \(\delta \to 0^+\) yields \(\hausdorff^d(\delta_\lambda(E)) \leq \lambda^d \hausdorff^d(E)\). Since this holds for any \(E\), we can set \(E := \delta_{1/\lambda}(F)\) for some Borel set \(F\) to get the opposite inequality.

  We now assume the existence of a metric ball \(B\) with positive and finite \(\hausdorff^d\)-measure, for some \(d \geq 0\). By the previous reasoning, since any ball can be translated and dilated into any other, the measure \(\hausdorff^d\) is finite and positive on every ball. This also shows that \(\hausdorff^d\) is finite on the closure of balls, hence a Radon measure. Moreover, note that \(\hausdorff^d\) is non-trivial, since it takes positive values on balls, and its left-invariance guarantees that it is a Haar measure. Thus, \(\hausdorff^d = c \mathcal{L}^3\) for some constant \(c > 0\). We now consider any ball \(B(\mathbf{0},r)\) centered at the origin. Since \(\delta_r(B(\mathbf{0},1)) = B(\mathbf{0},r)\) we have that
  \[
    \hausdorff^d(B(\mathbf{0},r)) = r^d \hausdorff^d(B(\mathbf{0},1)) = cr^d \mathcal{L}^3(B(\mathbf{0},1)).
  \]
  Since we know that the Lebesgue measure is 4-homogeneous with respect to dilations of \(\heis\), we get
  \[
    cr^d \mathcal{L}^3(B(\mathbf{0},1)) = cr^{d-4} \mathcal{L}^3(B(\mathbf{0},r)) = r^{d-4}\hausdorff^d(B(\mathbf{0},r)),
  \]
  therefore,
  \[
    \hausdorff^d(B(\mathbf{0},r)) = r^{d-4}\hausdorff^d(B(\mathbf{0},r)).
  \]
  Since \(\hausdorff^d\) is finite and positive on any ball, this implies \(r^{d-4} = 1\) for every \(r > 0\), thus \(d=4\).
\end{proof}

\cref{prop:hausdorff_dimension_must_be_4} tells us that it is sufficient to prove that there exists a dimensional exponent \(d\) such that balls have positive and finite measure to conclude that the 4-dimensional Lorentzian Hausdorff measure equals, up to a constant, the Lebesgue measure.

We now prove the ``ball-box'' type theorem mentioned earlier, starting with an upper bound on diamonds in terms of sub-Riemannian boxes. For $r \geq 0$, we use the notation
\[
  \operatorname{Box}(r) \coloneqq [-r,r] \times [-r,r] \times [-r^2,r^2] \subseteq \heis.
\]
\begin{theorem}\label{thm:ball_box_lorentzian}
  Given \(p, q \in \heis\), we have
  \[
    J(p,q) \subset p \ast \operatorname{Box}(\dis(p,q)).
  \]
  In particular, there exists a constant \(C > 0\), independent of \(p\) and \(q\), such that
  \[
    J(p,q) \subset B(p, C \cdot \dis(p,q)),
  \]
  so that \(\operatorname{diam}(J(p,q)) \leq 2C \cdot \dis(p,q)\).
\end{theorem}
\begin{proof}
  Of course, \(J(p, q) = \emptyset\) if \(p \not\leq q\), so we assume \(p \leq q\). By left translating by \(-p\), it is enough to consider \(p = \mathbf{0}\).
  We let \( q = (a,b,c) \), and for any \( (x,y,z) \in J(\mathbf{0},q) \), we have
  \[
    \begin{split}
      (x,y,z) \in J^+(\mathbf{0}) \cap J^-(q) &\subset \{x \geq 0\} \cap \{-x^2+y^2+4\abs{z} \leq 0\} \cap \{x \leq a\} \\
      &=\{0\leq x \leq a\} \cap \{y^2 + 4\abs{z} \leq x^2\}\\
      & \subset \{-a \leq x \leq a\} \cap \{y^2 +4 \abs{z} \leq a^2\} \\
      &\subset \{-a\leq x \leq a\} \cap \{y^2 \leq a^2\} \cap \{4\abs{z} \leq a^2\}\\
      &\subset \{-a\leq x \leq a\} \cap \{-a\leq y \leq a\} \cap \{-a^2\leq z \leq a^2\}\\
      &= \operatorname{Box}(a).
    \end{split}
  \]
  Note that \(a \leq \dis(\mathbf{0}, q)\) since the length of any horizontal curve from $\mathbf{0}$ to \(q\) equals the Euclidean length of its projection on the \(xy\)-plane, which joins $\mathbf{0}$ to \((a,b)\) and is therefore at least \(\sqrt{a^2 + b^2}\), the length of the straight segment between them. We have therefore proven that
  \[
    J(\mathbf{0},q) \subset \operatorname{Box}(\dis(p,q)).
  \]
  Applying the standard sub-Riemannian ball-box theorem yields the existence of a constant \(C > 0\), independent of \(p\) and \(q\), such that
  \[
    \operatorname{Box}(\dis(p,q)) \subset B(\mathbf{0},C \cdot \dis(p,q)),
  \]
  which concludes the proof.
\end{proof}
\begin{remark} \label{rmk:tau_diameter_equals_tau_distance_vertices}
  In any Lorentzian pre-length space, every diamond is completely determined by its two vertices. Moreover, for any diamond \(J(p,q)\), it is always true that
  \[
    \operatorname{diam}_\uptau(J(p,q)) = \uptau(p,q),
  \]
  where \(\operatorname{diam}_\uptau(E) \coloneqq \sup_{p,q \in E} \uptau(p,q)\). Hence, we always have control of the \say{time-diameter} of diamonds in terms of the time separation between their vertices. The previous theorem allows us, in the special case of the Heisenberg group \(\heis\) endowed with its sub-Riemannian distance, to control the \say{metric diameter} of diamonds as well, in terms of the sub-Riemannian distance between their vertices.
\end{remark}

We continue our discussion by proving the \say{reverse} bound, namely that for any ball, we can find a suitable diamond that contains it and whose diameter is controlled by the radius of the ball.
\begin{proposition}\label{prop:controlling_balls_with_diamonds}
  Consider any metric ball \(B(p,r)\) in \(\heis\). There exists a constant \(D>0\), independent of \(p\) and \(r\), and a diamond \(J\) with the following properties:
  \begin{enumerate}[label=\normalfont(\roman*), topsep=4pt,itemsep=4pt,partopsep=4pt, parsep=4pt]
    \item \(B(p,r) \subset J\);
    \item \(\operatorname{diam}_\uptau(J) = 2Dr\);
    \item \(\operatorname{diam}_{\dis}(J) \leq 4CDr\),
  \end{enumerate}
  where \(C\) is the constant given by \cref{thm:ball_box_lorentzian}.
\end{proposition}
\begin{proof}
  By left-translating everything by \(-p\), it is again not restrictive to assume that \(p = \mathbf{0}\). Given \(J \coloneqq J((-1,0,0),(1,0,0))\), we claim that
  \[
    \dis((-1,0,0),(1,0,0)) = \uptau((-1,0,0),(1,0,0)) = 2.
  \]
  It is enough to note that the segment joining the two points is length-minimizing for the sub-Riemannian distance and length-maximizing for \(\uptau\); in both cases, its length is 2. Applying \cref{thm:ball_box_lorentzian}, we get \(\operatorname{diam}_{\dis}(J) \leq 4C\).

  Since $\mathbf{0}$ lies in the interior of \(J\), there there exists \(\rho >0\) such that \(B(\mathbf{0},\rho) \subset J\). Notice that up to this point there is no dependency on \(p\) nor on \(r\). Dilating by a factor $r/\rho$, the ball \(B(\mathbf{0},\rho)\) is sent to \(B(\mathbf{0},r)\), while the diamond \(J\) is sent to \(J(x,y)\), with \(x = - y= (-r/\rho, 0 , 0)\). We set \(D \coloneqq 1/\rho\) and obtain that \(\dis(x,y) = \uptau(x,y) = 2Dr\) since both \(\dis\) and \(\uptau\) are homogeneous with respect to dilations. In particular, we know by \cref{rmk:tau_diameter_equals_tau_distance_vertices} that \(\operatorname{diam}_\uptau(J(x,y)) = 2Dr\). Since \(\operatorname{diam}_{\dis}(J) \leq 4C\), we obtain \(\operatorname{diam}_{\dis}(J(x,y)) \leq 4CDr\).
\end{proof}
\begin{remark}
  The previous proof is non-quantitative, as we only proved the existence of the constant \(D\) without knowing its value. One could try to find more explicitly the maximal radius \(\rho\) such that the sub-Riemannian ball \(B(\mathbf{0},\rho)\) is contained in \(J\).
\end{remark}

The last ingredient needed to prove that the 4-dimensional Lorentzian Hausdorff measure is a Haar measure is an estimate on the growth of the Lebesgue volume of diamonds in terms of the time separation function between their vertices. We compute explicitly the volume of diamonds of the form \(J(\mathbf{0},q)\) (this is not restrictive by left-translation invariance) and estimate it in the following proposition.
\begin{proposition}\label{prop:volume_growth_estimate}
  We have $\mathcal{L}^3(J(\mathbf{0},q)) = 0$ if \(q \notin I^+(\mathbf{0})\), and
  \begin{equation}
    \label{eq:volume_diamonds_general_case}
    \mathcal{L}^3(J(\mathbf{0},q)) = -\frac{(a^2 - b^2)^2}{8} \left(mM + m^2 \ln m + M^2 \ln M\right),
  \end{equation}
  if \(q= (a,b,c) \in I^+(\mathbf{0})\), where $m \coloneqq \tfrac{1}{2}\left(1 + \tfrac{4c}{a^2-b^2} \right)$ and $M \coloneqq \tfrac{1}{2}\left(1 - \tfrac{4c}{a^2-b^2} \right)$. Furthermore, there exists a constant \(K>0\), independent of \(q\), such that
  \begin{equation}
    \label{eq:volume_diamonds_growth}
    \mathcal{L}^3(J(\mathbf{0},q)) \leq K \uptau(\mathbf{0},q)^4.
  \end{equation}
\end{proposition}
\begin{proof}
  We firstly tackle the case \(q \in J^+(\mathbf{0}) \setminus I^+(\mathbf{0})\). In this case, we have proven in \cref{prop:lightlike_case} and \cref{cor:maximum_area_and_timelike_line_case} that the only causal curve (up to parametrization) between $\mathbf{0}$ and \(q\) is the null geodesic joining them. Hence, \(J(\mathbf{0},q)\) is a one dimensional (piecewise) smooth compact curve, and thus \(\mathcal{L}^3(J(\mathbf{0},q)) = 0\).

  If \(q \in I^+(\mathbf{0})\), we reduce the problem to the case where \(b = 0\). To do so, it is sufficient to notice that the linear map defined by
  \[
    R \coloneqq
    \begin{pmatrix}
      \dfrac{a}{\sqrt{a^2-b^2}} & - \dfrac{b}{\sqrt{a^2-b^2}} & 0\\
      - \dfrac{b}{\sqrt{a^2-b^2}} & \dfrac{a}{\sqrt{a^2-b^2}} & 0 \\
      0 & 0 & 1
    \end{pmatrix}
  \]
  preserves the time separation function, sends \(q\) to \((\sqrt{a^2-b^2}, 0, c)\), and has determinant equal to 1, so that it also preserves the volume of the diamond.
  The map is well defined since \(q \in I^+(\mathbf{0})\) insures that \(a^2 - b^2 > 0\). Hence, from now onwards we will assume that \(q = (a,0,c)\).

  We want to find equations characterizing the diamond \(J(\mathbf{0},q)\). Since
  \[
    J^+(\mathbf{0}) = \{x \geq 0\} \cap \left\{ \abs{z} \leq \frac{x^2 - y^2}{4} \right\},
  \]
  we can find the equations for \(J^-(q)\) by noticing that
  \[
    p \in J^-(q) \iff q \in J^+(p) \iff -p \ast q \in J^+(0),
  \]
  which tells us that the equations for \(J^-(q)\) are
  \[
    J^-(q) = \{x \leq a\} \cap \left\{ \abs{z - c - \frac{a y}{2}} \leq \frac{(x - a)^2 - y^2}{4} \right\}.
  \]
  We wish to apply Fubini-Tonelli's Theorem to compute the volume of \(J(\mathbf{0},q)\). Since any causal curve from \(\mathbf{0}\) to \(q\) in \(\heis\) projects to a causal curve from $\mathbf{0}$ to \((a,0)\) in the Minkowski plane, the projection of \(J(\mathbf{0},q)\) onto the \(xy\)-plane lies in the Minkowski diamond
  \[
    Q \coloneqq \{0 \leq x \leq a\} \cap \{ \abs{y} \leq x \leq a - \abs{y} \}.
  \]
  Hence a point \((x,y,z)\) belongs to \(J(\mathbf{0},q)\) if and only if
  \[
    (x,y) \in Q \text{, } \ \abs{z - c - \frac{a y}{2}} \leq \frac{(x - a)^2 - y^2}{4} \text{, and } \ \abs{z} \leq \frac{x^2 - y^2}{4}.
  \]
  The last two conditions are equivalent to
  \[
    z \in \left( \max \left( -\tfrac{x^2 - y^2}{4},\ c + \tfrac{a y}{2} - \tfrac{(x - a)^2 - y^2}{4} \right),\ \min \left( \tfrac{x^2 - y^2}{4},\ c + \tfrac{a y}{2} + \tfrac{(x - a)^2 - y^2}{4} \right) \right),
  \]
  where the interval is empty if the lower bound exceeds the upper bound. We can now write the integral as
  \begin{equation}\label{eq:integral_with_max_min}
    \begin{split}
      \mathcal{L}^3(J(\mathbf{0},q)) ={}& \int_Q \mathcal{L}^1\{ z \in \R \ \colon \ (x,y,z) \in J(\mathbf{0},q) \} \, dz \, d\mathcal{L}^2(x,y) \\
      ={}& \int_Q \max \left[ 0,\ \min \left( \tfrac{x^2 - y^2}{4},\ c + \tfrac{a y}{2} + \tfrac{(x - a)^2 - y^2}{4} \right) \right. \\
      &\qquad \qquad \left. - \max \left( -\tfrac{x^2 - y^2}{4},\ c + \tfrac{a y}{2} - \tfrac{(x - a)^2 - y^2}{4} \right) \right] \, d\mathcal{L}^2(x,y),
    \end{split}
  \end{equation}
  where the outer \(\max\) accounts for the case when the interval is empty. To explicitly solve this integral, one must analyze the integrand case by case. We give the detailed computation in the appendix, see \cref{lemma:integral_computations}. The final result is
  \begin{equation}\label{eq:volume_diamonds_simplified_case}
    \mathcal{L}^3(J(\mathbf{0},q)) = -\frac{a^4}{8} \left( mM + M^2 \ln(M) + m^2 \ln(m) \right),
  \end{equation}
  where
  \[
    m \coloneqq \frac{1}{2} \left( 1 + \frac{4c}{a^2} \right), \quad M \coloneqq \frac{1}{2} \left( 1 - \frac{4c}{a^2} \right).
  \]
  This formula holds for points of the form \((a,0,c)\). To get \cref{eq:volume_diamonds_general_case}, it is enough to recall that any point \((a,b,c)\) can be reduced to \((\sqrt{a^2 - b^2}, 0, c)\), and then apply \cref{eq:volume_diamonds_simplified_case}.

  It remains to prove the growth estimate \cref{eq:volume_diamonds_growth}. For any point \(q \in I^+(\mathbf{0})\) (the case \(q \in J^+(\mathbf{0}) \setminus I^+(\mathbf{0})\) is trivial since the volume vanishes), set \(\lambda \coloneqq 1 / \uptau(\mathbf{0},q)\) and consider the dilation map \(\delta_\lambda\). This map has Jacobian determinant \(1 / \uptau(\mathbf{0},q)^4\), so we get
  \[
    \mathcal{L}^3(\delta_\lambda(J(\mathbf{0},q))) = \mathcal{L}^3(J(\mathbf{0},\delta_\lambda(q))) = \frac{1}{\uptau(\mathbf{0},q)^4} \mathcal{L}^3(J(\mathbf{0},q)).
  \]
  Consider the point \(\tilde{q} = \delta_\lambda(q)\), we have \(\uptau(0,\tilde{q}) = 1\) and
  \[
    \mathcal{L}^3(J(\mathbf{0},q)) = \uptau(\mathbf{0},q)^4 \mathcal{L}^3(J(\mathbf{0},\tilde{q})),
  \]
  so to prove the result it is enough to bound the volume of diamonds with unit time-diameter. We can use \cref{eq:volume_diamonds_simplified_case}, since the reduction procedure does not change the time-diameter of diamonds. Hence we may assume \(\tilde{q} = (a,0,c)\). Because \(\tilde{q}\) is at \(\uptau\)-distance 1 from the origin, it lies in the image of the exponential map described in \cref{def:exponential_map} for \(t = 1\), with parameters \(u,v,w\) satisfying \(u^2 - v^2 = 1\). Moreover, the \(y\)-coordinate of \(\tilde{q}\) vanishes. This fixes all but one of the parameters, and from the expression of the exponential map it is straightforward to see that \(\tilde{q}\) must be of the form
  \[
    \tilde{q} =
    \begin{pmatrix}
      \dfrac{\strut 2\sinh \left( \frac{w}{2}\right)}{\strut w}\\
      \strut 0\\
      \dfrac{\strut\sinh(w)-w}{2w^2}
    \end{pmatrix}, \qquad \text{ for some \(w \in \R\).}
  \]
  By plugging this expression into \cref{eq:volume_diamonds_simplified_case}, we get a continuous function of \(w\). So it is enough to prove that the limit of the volume as \(w \to \pm\infty\) is finite to obtain a global bound. Also, replacing \(w\) by \(-w\) does not change \cref{eq:volume_diamonds_simplified_case}, so we can restrict to \(w \to +\infty\). We observe that \(a^4 \sim e^{2w}/w^4\) as \(w \to +\infty\), so the explicit limit we want to compute is
  \[
    \begin{split}
      \lim_{w \to +\infty} \frac{e^{2w}}{w^4} &\left[ \frac{-e^{-w} + e^{-2w} + w e^{-w}}{(e^{-w} - 1)^2} - \left( \frac{-e^{-w} + e^{-2w} + w e^{-w}}{(e^{-w} - 1)^2} \right)^2 \right. \\
        &+ \left( \frac{-e^{-w} + e^{-2w} + w e^{-w}}{(e^{-w} - 1)^2} \right)^2 \ln \left( \frac{-e^{-w} + e^{-2w} + w e^{-w}}{(e^{-w} - 1)^2} \right) \\
      &\left. + \left( 1 + \frac{e^{-w} - e^{-2w} - w e^{-w}}{(e^{-w} - 1)^2} \right)^2 \ln \left( 1 + \frac{e^{-w} - e^{-2w} - w e^{-w}}{(e^{-w} - 1)^2} \right) \right],
    \end{split}
  \]
  which we show equals 0 in \cref{lemma:volume_limit_is_zero} of the appendix. The volume of \(J(\mathbf{0},\tilde{q})\) therefore goes to zero as \(w \to \pm\infty\). Since it is a positive continuous function, it must then admit a finite upper bound \(K > 0\), which proves the last claim.
\end{proof}
\begin{remark}
  We believe that the function \(\mathcal{L}^3(J(\mathbf{0},\tilde{q}))\) attains its maximum at \(\tilde{q} = (1,0,0)\), i.e. when \(w = 0\). This would imply that the optimal constant in \cref{eq:volume_diamonds_growth} is
  \[
    K = \mathcal{L}^3(J(\mathbf{0},(1,0,0))) = \frac{2\ln 2 - 1}{32}.
  \]
  Since a sharp bound is not needed here, we do not pursue a full proof of this claim.
\end{remark}
We are now set to prove \cref{thm:lorentzian_hausdorff} (compare with \cite[Theorem 2.1.34]{LeDonneBook}).
\begin{proof}[Proof of \cref{thm:lorentzian_hausdorff}]
  By \cref{prop:hausdorff_dimension_must_be_4}, it is enough to show that some ball \(B\) has positive and finite measure. We take any ball \(B\) and prove the two bounds separately.

  Fix any \(\delta\)-cover of \(B\) by diamonds \(\{J(p_i, q_i)\}_{i \in \N}\), then
  \[
    \sum_i \uptau(p_i, q_i)^4 = \frac{1}{K} \sum_i K \uptau(p_i, q_i)^4,
  \]
  where \(K > 0\) is the constant from the growth estimate in \cref{prop:volume_growth_estimate}.
  Therefore, we have
  \[
    \sum_i \uptau(p_i,q_i)^4 = \frac{1}{K} \sum_i K \uptau(p_i,q_i)^4 \geq \frac{1}{K}\sum_i \mathcal{L}^3(J(p_i,q_i)).
  \]
  Since \(\mathcal{L}^3\) is \(\sigma\)-subadditive and since the union of all of the diamonds covers \(B\), we get
  \[
    \sum_i \uptau(p_i,q_i)^4 \geq \frac{1}{K}\mathcal{L}^3\left(\bigcup_i J(p_i,q_i)\right) \geq \frac{1}{K} \mathcal{L}^3(B) > 0.
  \]
  The rightmost term does not depend on the cover or on \(\delta\), so taking the infimum over all \(\delta\)-covers by diamonds and letting \(\delta \to 0^+\) gives
  \[
    \hausdorff^4(B) \geq \frac{1}{K} \mathcal{L}^3(B) > 0,
  \]
  which shows that \(B\) has positive measure.

  To prove the other bound, let \(B = B(p, r)\), consider \(B(p, r/2)\) and \(\delta \in (0, r/2)\). Following \cite[proof of Theorem 2.1.34]{LeDonneBook}, we claim there exists a maximal finite set of points \(\{p_1, \dots, p_l\}\) in \(B(p, r/2)\) such that \(\dis(p_i, p_j) > \delta\). Indeed, if \(p_1, \dots, p_l\) are such points, then the balls \(B(p_i, \delta/2)\) are disjoint and contained in \(B\) by the triangle inequality. Since \(\mathcal{L}^3(B(q, \rho)) = M \rho^4\) for any \(q \in \heis\) and \(\rho > 0\), we have
  \[
    l\left(\frac{\delta}{2}\right)^4 = \sum_{i=1}^l \left(\frac{\delta}{2}\right)^4 = \frac{1}{M} \sum_{i=1}^l \mathcal{L}^3\left(B\left(p_i, \frac{\delta}{2}\right)\right) = \frac{1}{M} \mathcal{L}^3\left( \bigcup_{i=1}^l B\left(p_i, \frac{\delta}{2}\right) \right) \leq \frac{1}{M} \mathcal{L}^3(B),
  \]
  which implies
  \[
    l \leq \frac{2^4}{M \delta^4} \mathcal{L}^3(B) < \infty,
  \]
  so for fixed \(\delta\), a maximal \(l\) exists. If we now take a maximal set of such points \(p_1, \dots, p_m\), then \(\{B(p_i, \delta)\}_{i=1}^m\) covers \(B(p, r/2)\). Applying \cref{prop:controlling_balls_with_diamonds}, we get a collection of diamonds \(\{J(x_i, y_i)\}_{i=1}^m\) such that for each \(i \in \{1, \dots, m\}\),
  \[
    B(p_i, \delta) \subset J(x_i, y_i), \quad \uptau(x_i, y_i) = 2D\delta, \quad \text{and} \quad \operatorname{diam}_{\dis}(J(x_i, y_i)) \leq 4CD\delta.
  \]
  This family of diamonds consists of a \(4CD\delta\)-cover of \(B(p, r/2)\), so we have
  \[
    \begin{split}
      \prehausdorff{4CD\delta}^4\left(B\left(p, \frac{r}{2}\right)\right) &\leq \sum_{i=1}^m \uptau(x_i, y_i)^4 = 2^4 D^4 \sum_{i=1}^m \delta^4 = 4^4 D^4 \sum_{i=1}^m \left(\frac{\delta}{2}\right)^4 \\
      &= \frac{(4D)^4}{M} \sum_{i=1}^m \mathcal{L}^3\left(B\left(p_i, \frac{\delta}{2}\right)\right) \\
      &= \frac{(4D)^4}{M} \mathcal{L}^3\left( \bigcup_{i=1}^m B\left(p_i, \frac{\delta}{2}\right) \right) \leq \frac{(4D)^4}{M} \mathcal{L}^3(B).
    \end{split}
  \]
  The last estimate does not depend on \(\delta\) or on the specific cover, so taking the infimum over all \(4CD\delta\)-covers and letting \(\delta \to 0^+\), we obtain
  \[
    \hausdorff^4 \left( B \left( p, \frac{r}{2} \right) \right) \leq \frac{(4D)^4}{M} \mathcal{L}^3(B) < +\infty.
  \]
  By performing a final dilation of factor 2, we get
  \[
    \hausdorff^4(B) \leq \frac{(8D)^4}{M} \mathcal{L}^3(B) < +\infty.
  \]

  This shows that \(B\) has finite and positive measure, so the whole space satisfies \(\dim_{\uptau}(\heis) = 4\). Moreover, by \cref{prop:hausdorff_dimension_must_be_4}, we know that \(\hausdorff^4\) is a Haar measure, and thus must agree with \(\mathcal{L}^3\) and \(\mathcal{H}^4\) up to non-zero constant factors.
\end{proof}

\section{Synthetic timelike curvature bounds in the Heisenberg group}\label{sect:heis_tcd_tmcp}

We are now prepared to study synthetic curvature-dimension conditions in the sub-Lorentzian Heisenberg group. As explained in the introduction, we will show that the timelike curvature-dimension conditions and the timelike measure contraction properties (along with all their variants) fail in the Lorentzian geodesic space $(\heis, \dis, \m, \leq, \ll, \uptau)$. Here, we make the following extra assumptions on the Radon measure $\m$ attached to our structure. For the rest of this section, we assume that
\begin{enumerate}[label=\normalfont(\Alph*), ref=\normalfont(\Alph*), topsep=4pt,itemsep=4pt,partopsep=4pt, parsep=4pt]
  \item\label{item:assumptionA} there exists a point $q \in \heis$ and a neighbourhood $U \ni q$ such that $\m\mres U \in \mathscr{P}^{\mathrm{ac}}(\heis, \mathcal{L}^3)$, $f \coloneqq \diff \m/\mathrm{d} \mathcal{L}^3$ is continuous at $q$ with $f(q)>0$.
\end{enumerate}
It is possible that \labelcref{item:assumptionA} could be further weakened, but it is sufficient for proving the next result. Note that $p \in (0,1)$ is fixed for the remainder of this section; accordingly, we omit the subscript in the notation $\mathsf{TCD}_p$, $\mathsf{TMCP}_p$, and so on.

\begin{proposition}
  \label{prop:KNto0N}
  If $(\heis, \dis, \m, \leq, \ll, \uptau)$, with $\mathfrak{m}$ veryfying \labelcref{item:assumptionA}, satisfies the $\mathsf{TMCP}(K, N)$ (resp. $\mathsf{TCD}(K, N)$) condition for some $K \in \mathbb{R}$ and $N \geq 1$ (resp. $N \in [1, +\infty]$), then $(\heis, \dis, \mathcal{L}^3, \leq, \ll, \uptau)$ satisfies the $\mathsf{TMCP}(0, N)$ (resp. $\mathsf{TCD}(0, N)$).
\end{proposition}

\begin{proof}
  We only show the result for $\mathsf{TMCP}$ since the case for $\mathsf{TCD}$ is entirely analogous.
  We also observe that it is not restrictive to assume the point $q \in \heis$, around which $\mathfrak{m}$ is absolutely continuous with respect to the Lebesgue measure, to be the origin. Indeed, we can simply consider the translation map given by $L_{-q}$: this map is an isomorphism of Lorentzian metric measure spaces between the spaces $(\heis, \dis, \m, \leq, \ll, \uptau)$ and $(\heis, \dis, (L_{-q})_\sharp\mathfrak{m}, \ll,\leq, \uptau)$, so the $\mathsf{TMCP}(K, N)$ holds in the latter space as well.
  The pushforward measure $\mathfrak{m}' \coloneqq (L_{-q})_\sharp \mathfrak{m}$ satisfies \labelcref{item:assumptionA}, with $q' = \mathbf{0}$, \(U' = L_{-q}(U)\) and $f' = f \circ L_{-q}^{-1} \cdot \det(JL_{-q}^{-1}) = f \circ L_q$. Hence we assume that $q = \mathbf{0}$.

  By scaling invariance of the measure contraction property recalled in \cref{section:TCDTMCP}, we see that the space defined by
  \[
    \tilde{X}_n \coloneqq (\heis, \dis, n^4\mathfrak{m}, \ll,\leq, n\uptau,)
  \]
  satisfies the $\mathsf{TMCP}(K/n^2, N)$ condition. Since the Heisenberg's group dilation
  \[
    \delta_n : \heis \to \heis : (x, y, z) \mapsto (n x, n y, n^2 z)
  \]
  is an isomorphism of measured Lorentzian pre-length spaces between the structures $\tilde X_n$ and $X_n \coloneqq (\heis, \dis, (\delta_n)_\sharp(n^4\mathfrak{m}), \ll, \leq,\uptau)$, we deduce that $X_n$ satisfies $\mathsf{TMCP}(K/n^2, N)$ too.
  We aim to show that
  \[
    X_n \to X_{\infty} \coloneqq (\heis, \dis, f(\mathbf{0}) \mathcal{L}^3, \ll,\leq,\uptau), \qquad \text{ as } n \to +\infty,
  \]
  in the sense of \cite[Definition 3.24]{Braun2023a}. By the stability theorem \cite[Theorem 4.11]{Braun2023a} (see also \cite[Theorem 3.13]{cavallettimondino2024}), we would be able to conclude that $X_\infty$ satisfies $\mathsf{TMCP}(0, N)$, showing the statement since $f(0) > 0$ by hypothesis.

  The identity maps $\iota_n : X_n \to (\heis, \dis, \mathcal{L}^3, \ll, \leq, \uptau)$, and $\iota_\infty : X_\infty \to (\heis, \dis, \mathcal{L}^3, \ll, \leq, \uptau)$ are obviously Lorentzian isometric embeddings from $X_{n}$ and $X_\infty$ onto the causally closed, $\mathcal{K}$-globally hyperbolic (see \cite[Proposition 18]{borza2025}),
  measured Lorentzian geodesic space, according to \cite[Definition 3.23]{Braun2023a}. It remains to show that that
  \[
    (\iota_n)_\sharp (\delta_n)_\sharp(n^4\mathfrak{m}) \to f(\mathbf{0}) \mathcal{L}^3, \qquad \text{ as } n \to +\infty,
  \]
  in duality with $C_{\mathrm{c}}(\heis)$. Consider $\varphi \in C_{\mathrm{c}}(\heis)$ and $n$ sufficiently large so that $\supp(\varphi) \subset \delta_n(U)$, where $U$ is a neighbourhood of $\mathbf{0}$ such that $\mathfrak{m} \mres U$ is absolutely continuous with respect to $\mathcal{L}^3$; we can also assume that \(U\) is such that \(f\) is bounded on it, since \(f\) is continuous at $\mathbf{0}$. Notice that this condition, along with the fact that $\delta_n$ is a homeomorphism, implies that
  \[
    \supp(\varphi \circ \delta_n) = \delta_n^{-1} (\supp(\varphi)) \subset U.
  \]
  Therefore, setting $f(y) = 0$ if $y \notin U$, we have that
  \begin{align*}
    \int_{\heis} \varphi \diff (\iota_n)_\sharp (\delta_n)_\sharp(n^4\mathfrak{m}) & = \int_{\supp(\varphi)} \varphi(x) \diff (\delta_n)_\sharp(n^4\mathfrak{m})(x)\\
    & = \int_{\delta_n(U)} \varphi(x) \diff (\delta_n)_\sharp(n^4\mathfrak{m})(x) && \text{($\supp(\varphi) \subset \delta_n(U)$)}\\
    & = \int_{U}n^4\varphi(\delta_n(y)) \diff \mathfrak{m}(y) && \text{(pushforward)}\\
    & = \int_U n^4 \varphi(\delta_n(y)) f(y) \diff \mathcal{L}^3(y) && \text{(aboslute continuity)}\\
    &= \int_{\R^3} n^4 \varphi(\delta_n(y))f(y) \diff \mathcal{L}^3(y) && \text{($\supp(\varphi\circ \delta_n) \subset U)$} \\
    & = \int_{\R^3} n^4 \varphi(z) f(\delta_{1/n}(z)) \frac{1}{n^4} \diff \mathcal{L}^3(z) && \text{(smooth area formula)}.
  \end{align*}
  By Lebesgue's Dominated Convergence Theorem and since $f$ is continuous at $\mathbf{0}$ by Assumption \labelcref{item:assumptionA}, we get
  \[
    \lim_{n \to \infty} \int_{\heis} \varphi \diff (\iota_n)_\sharp (\delta_n)_\sharp(n^4\mathfrak{m}) = \int_{\R^3} \lim_{n \to \infty} \varphi(z) f(\delta_{1/n}(z)) \diff \mathcal{L}^3(z) = \int_{\R^3} \varphi(z) f(\mathbf{0}) \diff \mathcal{L}^3(z),
  \]
  concluding the proof.
\end{proof}

\begin{theorem}
  The Lorentzian space $(\heis, \dis, \m, \leq, \ll, \uptau)$ does not satisfy the $\mathsf{TCD}(K, N)$ condition for any $K \in \R$ and any $N \in [1, +\infty]$.
\end{theorem}

\begin{proof}
  The argument will follow Juillet's orginal idea, see \cite{juillet2009}, that we adapt to the Lorentzian setting. By \cref{prop:KNto0N}, it is enough to disprove the $\mathsf{TCD}(0, N)$ condition in $(\heis, \dis, \mathcal{L}^3, \leq, \ll, \uptau)$. We argue by contradiction, assuming that $\mathsf{TCD}(0, N)$ holds for some $N \in [1, +\infty]$. The consistency of the $\mathsf{TCD}$ implies that $\mathsf{TCD}(0, \infty)$ holds and, by \cref{prop:TCDtoTBM}, also $\mathsf{TBM}(0, \infty)$ is verified.

  Choosing $K = 0$ and $t = 1/2$ in \cref{Eq:BMinfty} means that for every relatively compact Borel sets $A_0,A_1\subset X$ with $\m(A_0)\,\m(A_1) >0$ such that
  \begin{align*}
    (\mu_0,\mu_1) \coloneqq (\mathcal{L}^3(A_0)^{-1}\,\mathcal{L}^3\mres A_0,\mathcal{L}^3(A_1)^{-1}\, \mathcal{L}^3\mres A_1).
  \end{align*}
  is timelike $p$-dualizable, it holds that
  \begin{equation}
    \label{eq:BMKinftyt1/2}
    \mathcal{L}^3(A_t) \geq \sqrt{\mathcal{L}^3(A_0) \mathcal{L}^3(A_1)}.
  \end{equation}
  We set
  \[
    q_{-1} \coloneqq (-1, 0, 0), \ q_{0} \coloneqq (0, 0, 0), \ q_1 \coloneqq (1, 0, 0),
  \]
  and we note that $q_{-1} \in I^{-}(q_0)$ while $q_1 \in I^+(q_0)$. We also define the so-called geodesic-inversion map as
  \[
    I_{q_0} : I^+(q_{-1}) \cup I^-(q_1) \to \heis : q \mapsto \exp_{\mathbf{0}}^{(-1)} \circ (\exp_{\mathbf{0}}^{(1)})^{-1} (q).
  \]
  Clearly, we have that
  \[
    I_{q_0}(\exp_{\mathbf{0}}^s(u, v, w)) = I_{q_0}(\exp_{\mathbf{0}}(s u, s v, s w)) = \exp_{\mathbf{0}}^{(-1)}(s u, s v, s w) = \exp_{\mathbf{0}}(-s u, -s v, -s w).
  \]
  In other words, the map $I_{q_0}$ takes a point joined by a unique geodesic through $q_0$ and sends it to a point the other side of the geodesic so that $q_0$ is their midpoint. For a small enough $\epsilon > 0$, consider a small compact ball $A^\epsilon_0 \coloneqq B_\epsilon(q_{-1})$ with centre $q_{-1}$ such that it is contained in \(I^-(q_0)\) and $A^\epsilon_1 \coloneqq I_{q_0}(A^\epsilon_{0})$. Note that the corresponding $(\mu_0, \mu_1)$ is $p$-timelike dualizable thanks to \cref{eq:compactsupportpdualizable}. Indeed, every point in $\mathrm{supp}(\mu_1)$ is in the chronological future of $q_0$, which is again in the chronological future of every point of $\mathrm{\mu_0}$.

  The inequality \cref{eq:BMKinftyt1/2} thus implies that
  \begin{equation}
    \label{eq:BMKinftyt1/2epsilon}
    \frac{\mathcal{L}^3(Z_{1/2}(A^\epsilon_0, I_{q_0}(A^\epsilon_{0})))}{\mathcal{L}^3(A^\epsilon_0)} \geq \sqrt{\frac{\mathcal{L}^3(I_{q_0}(A^\epsilon_{0}))}{\mathcal{L}^3(A^\epsilon_0)}}.
  \end{equation}

  We are going to take the limit in \cref{eq:BMKinftyt1/2epsilon} as $\epsilon$ tends to 0. Notice that by the chain rule and \Cref{lemma:jacobian_exponential}, we have that the Jacobian determinant of \(I_{q_0}\) is constantly equal to 1. By the area formula, this implies that \(I_{q_0}\) is a volume preserving map and thus the right hand side of \eqref{eq:BMKinftyt1/2epsilon} is equal to 1.
  To estimate the limit of the left hand-side of \cref{eq:BMKinftyt1/2epsilon} when $\epsilon$ goes to 0, we use the
  following result of Juillet, e.g. \cite[Theorem 1]{Juillet2010}:
  \[
    \lim_{\epsilon \to 0} \frac{\mathcal{L}^3(Z_{1/2}(A^\epsilon_0, I_{q_0}(A^\epsilon_{0})))}{\mathcal{L}^3(A^\epsilon_0)} \leq 2^3 \abs{D_{q_{-1}} M_{q_1}},
  \]
  where
  \[
    M_{q_1}(p) = \exp^{(1/2)}_{p} \circ (\exp_{p}^{(1)})^{-1}(q_1) = \exp_{q_1}^{(-1/2)} \circ (\exp_{q_1}^{(-1)})^{-1} (p)
  \]
  is a geodesic midpoint map. Since \(\exp_q^{(t)} =L_q \circ \exp_{\mathbf{0}}^{(t)}\) for any \(t \in \R\) and any \(q \in \heis\) and \(L_q\) are maps with Jacobian determinant equal to one, applying the chain rule immediately yields
  \[
    \abs{\det(D_{q_{-1}}M_{q_1})} = \abs{\frac{\det(D_{\tilde q}\exp^{(-1/2)}_{\mathbf{0}})}{\det(D_{\tilde q}\exp^{(-1)}_{\mathbf{0}})}}
  \]
  where
  \[
    \tilde q \coloneqq (\exp_{\mathbf{0}}^{(-1)})^{-1}(L_{q_1}^{-1}  (q_{-1})) = (\exp_{\mathbf{0}}^{(-1)})^{-1}(-2,0,0) = (2,0,0).
  \]
  We apply \cref{lemma:jacobian_exponential} to see that
  \[
    \abs{\det (D M_{q_1})_{\vert q_{-1}}} = \frac{(1/2)^5 *4/12}{4/12} = \frac{1}{32},
  \]
  and we conclude that $\mathsf{TBM}(0, +\infty)$ is not satisfied.
\end{proof}

\begin{theorem}
  The Lorentzian space $(\heis, \dis, \m, \leq, \ll, \uptau)$ does not satisfy the $\mathsf{TMCP}(K, N)$ for any $K \in \R$ and any $N \in [1, +\infty)$.
\end{theorem}

\begin{proof}
  It is again enough to disprove the $\mathsf{TMCP}(0, N)$ condition in $(\heis, \dis, \mathcal{L}^3, \leq, \ll, \uptau)$, thanks to \cref{prop:KNto0N}. As before, we argue by contradiction, assuming that the $\mathsf{TMCP}(0, N)$ holds for some $N \in [1, +\infty)$. By \cref{prop:TMCP->TBM1/2}, it follows that both $\mathsf{TBM}^{1/2+}_{p}(0,N)$ and $\mathsf{TBM}^{1/2-}_{p}(0,N)$ must hold. We will disprove $\mathsf{TBM}^{1/2+}_{p}(0,N)$, which suffices to establish the theorem; the argument for the failure of $\mathsf{TBM}^{1/2-}_{p}(0,N)$ is entirely analogous.

  Let $q \in I^{-}(\mathbf{0})$ and $\epsilon > 0$ sufficiently small so that $A_\epsilon \coloneqq B_{\epsilon}(q) \subseteq I^{-}(\mathbf{0})$. Clearly, $\mathcal{L}^3(A_\epsilon) > 0$ and $\mathbf{0} \in I^{+}(A_\epsilon)$ for all small $\epsilon > 0$. We claim that the pair given by $(\mathcal{L}^3(A_\varepsilon)^{-1}\,\mathcal{L}^3 \mres A_\varepsilon,\delta_{\mathbf{0}})$ is $p$-dualizable by \cref{eq:compactsupportpdualizable}: writing $\mu_0 \coloneqq \mathcal{L}^3(A_\varepsilon)^{-1}\,\mathcal{L}^3$, we observe that $\operatorname{supp}(\mu_0 \times \delta_{\mathbf{0}}) = A_\epsilon \times \{\mathbf{0}\}$, which is entirely contained in $\heis^2_\ll$ since $A_\epsilon \subseteq I^{-}(\mathbf{0})$.

  The $\mathsf{TBM}^{1/2+}_{p}(0,N)$ condition implies that
  \begin{equation}
    \label{eq:TBM0Nepsilon}
    \mathcal{L}^3(Z_t(A_\epsilon, \mathbf{0})) \geq (1 - t)^N \mathcal{L}^3(A_\epsilon), \qquad \text{ for all } t \in [0, 1).
  \end{equation}
  We know that $\exp_{\mathbf{0}}^{(-1)}$ is a diffeomorphism from \(\{u \geq \abs{v} \times \R\}\) onto $I^{-}(\mathbf{0})$, thus there is an open set $\mathcal{U}_\epsilon \subseteq \T_{\mathbf{0}}^*(\heis)$ such that $A_\epsilon = \exp_{\mathbf{0}}(\mathcal{U})$. Actually, by definition of the $t$-intermediate set appearing in \cref{eq:tintermediate} and by the uniqueness property of geodesics in the sub-Lorentzian Heisenberg group, we also have that $Z_t(A_\epsilon, \mathbf{0}) = \exp_{\mathbf{0}}^{(t - 1)}(\mathcal{U}_\epsilon)$. Therefore, we can obtain
  \begin{align*}
    \frac{\abs{D_{(u, v, w)} \exp_{\mathbf{0}}^{(t-1)}}}{\abs{D_{(u, v, w)} \exp_{\mathbf{0}}^{(-1)}}} & = \lim_{\epsilon \to 0} \frac{\int_{\mathcal{U}_\epsilon}(\exp_{\mathbf{0}}^{(t - 1)})^* \diff u \wedge \diff v \wedge \diff w}{\int_{\mathcal{U}_\epsilon}(\exp_{\mathbf{0}}^{(-1)})^* \diff u \wedge \diff v \wedge \diff w} = \lim_{\epsilon \to 0} \frac{\mathcal{L}^3(\exp_{\mathbf{0}}^{(t - 1)}(\mathcal{U}_\epsilon))}{\mathcal{L}^3(\exp_{\mathbf{0}}^{(-1)}(\mathcal{U}_\epsilon))} \\
    & = \lim_{\epsilon \to 0} \frac{\mathcal{L}^3(Z_t(A_\epsilon, \mathbf{0}))}{\mathcal{L}^3(A_\epsilon)} \geq t^N, \qquad \text{ for all } t \in [0, 1),
  \end{align*}
  where $(u, v, w)$ is the unique vector in \(\{u \geq \abs{v} \times \R\}\) such that $q = \exp_{\mathbf{0}}^{(-1)}(u, v, w)$ and where the last inequality follows from \cref{eq:TBM0Nepsilon}. By the arbitraryness of $q \in I^{-}(\mathbf{0})$, this inequality holds for all \((u, v, w) \in \{u \geq \abs{v} \times \R\}\). By using \cref{lemma:jacobian_exponential}, one can see that, for all $t \in (0, 1)$,
  \begin{equation}
    \label{eq:notMCPcomputation}
    \begin{aligned}
      \frac{\abs{D_{(u, v, w)} \exp_{\mathbf{0}}^{(t-1)}}}{\abs{D_{(u, v, w)} \exp_{\mathbf{0}}^{(-1)}}} = \frac{(t - 1) \sinh(\tfrac{w(t - 1)}{2}) (\tfrac{w(t - 1)}{2} \cosh(\tfrac{w(t - 1)}{2}) - \sinh(\tfrac{w(t - 1)}{2}))}{t \sinh(\tfrac{w t}{2})(\tfrac{w t}{2} \cosh(\tfrac{w t}{2}) - \sinh(\tfrac{w t}{2}))} \longrightarrow 0,
    \end{aligned}
  \end{equation}
  as $w \to - \infty$, just by noting that \(e^{(t-1)w}\) grows slower than \(e^{-w}\).
\end{proof}

\begin{remark}
  The limit \cref{eq:notMCPcomputation} is stronger than the failure of the $\mathsf{TMCP}$. It implies that the sub-Lorentzian structure would fail to satisfy any reasonable timelike version of the notion of \emph{qualitative non-degeneracy} introduced in \cite[Assumption 1]{Cavalletti2015}.
\end{remark}

\appendix
\section*{\appendixname}
\addcontentsline{toc}{section}{\appendixname}
\stepcounter{section}
\renewcommand{\thesection}{A}
\renewcommand{\thefigure}{\thesection.\arabic{figure}}
In this short appendix, we collect some of the more technical steps and computations carried out throughout the paper.

The following lemma provides a formula for the Jacobian determinant of the exponential map introduced in \Cref{def:exponential_map}. The key idea, which simplifies several of the computations, follows an argument similar to that presented in \cite{juillet2009}:
\begin{lemma}\label{lemma:jacobian_exponential}
  The sub-Lorentzian exponential map of the Heisenberg group, that appeared in \cref{def:exponential_map}, satisfies
  \[
    {D_{(u,v,w)} \exp^{(t)}_\mathbf{0}} = 4t(u^2-v^2) \, \frac{\sinh(wt/2)}{w} \, \frac{wt/2 \cosh(wt/2) - \sinh(wt/2)}{w^3},
  \]
  which admits a smooth extension at \(w=0\), given by
  \[
    {D_{(u,v,0)} \exp^{(t)}_\mathbf{0}} = \frac{t^5}{12}(u^2-v^2).
  \]
\end{lemma}
\begin{proof}
  We consider the map \(T \colon \R^3 \to \R^3\) defined by $T(u, v, w) := (t u, t v, t w)$
  and the map \(\varphi \colon \R^3 \to \R^3\) given by
  \[
    \begin{pmatrix}
      a\\
      b\\
      c
    \end{pmatrix}
    \mapsto
    \renewcommand\arraystretch{2}
    \begin{pmatrix}
      \dfrac{b(\cosh(c)-1) +a\sinh(c)}{c}\\
      \dfrac{b\sinh(c) +a(\cosh(c)-1)}{c}\\
      \dfrac{a^2-b^2}{2} \cdot \dfrac{\sinh(c)-c}{c^2}
    \end{pmatrix}.
  \]
  It is immediate to see that \(\exp^{(t)}_\mathbf{0} = \varphi \circ T\). Hence we have that
  \[
    {D_{(u,v,w)}\exp_\mathbf{0}^{(t)}} = D_{(ut,vt,wt)}\varphi \cdot D_{(u,v,w)}T = t^3 D_{(ut,vt,wt)}\varphi,
  \]
  so it remains only to compute \(D\varphi\). We also observe that for any pair of real numbers \(\alpha, \beta\) with \(\alpha > \abs{\beta}\) one can define the linear map \(R \colon \R^3 \to \R^3\) by
  \[
    R(a,b,c) \coloneqq
    \begin{pmatrix}
      \dfrac{\alpha}{\sqrt{\alpha^2 - \beta^2}} & \dfrac{\beta}{\sqrt{\alpha^2 - \beta^2}} & 0\\[1em]
      \dfrac{\beta}{\sqrt{\alpha^2 - \beta^2}} & \dfrac{\alpha}{\sqrt{\alpha^2 - \beta^2}} & 0\\[1em]
      0 & 0 & 1
    \end{pmatrix}
    \begin{pmatrix}
      a\\
      b\\
      c
    \end{pmatrix}.
  \]
  Its Jacobian determinant is equal to \(1\). Moreover, a straightforward computation shows that \(\varphi \circ R = R \circ \varphi\), which in turn implies
  \[
    D_{(a,b,c)}\varphi = D_{R(a,b,c)}\varphi.
  \]
  Since we want to compute the determinant of the exponential map, we can restrict attention to points \((a,b,c)\) with \(a > \abs{b}\). The previous equation then allows us to use \(\alpha = a\) and \(\beta = b\) to get that
  \[
    D_{(a,b,c)}\varphi = D_{(\sqrt{a^2-b^2}, 0, c)}\varphi,
  \]
  which tells us that it is not restrictive to assume that \(b=0\). We now compute the Jacobian matrix \(J\) of \(\varphi\) and evaluate it at \((a,0,c)\); we get
  \[
    J =
    \renewcommand\arraystretch{2}
    \begin{pmatrix}
      \dfrac{\sinh(c)}{c} & \dfrac{\cosh(c)-1}{c} & a \cdot \dfrac{c\cosh(c)-\sinh(c)}{c^2}\\
      \dfrac{\cosh(c)-1}{c} & \dfrac{\sinh(c)}{c} & a \cdot \dfrac{c\sinh(c)-(\cosh(c)-1)}{c^2}\\
      a \cdot\dfrac{\sinh(c)-c}{c^2} & 0 & \dfrac{a^2}{2} \cdot \dfrac{c^2(\cosh(c)-1) -2c(\sinh(c)-c)}{c^4}
    \end{pmatrix}.
  \]
  To compute its determinant, we can collect a factor \(a\) from the last row and the last column. We can also collect a factor \(1/c\) from the first two columns and the last row and a factor \(1/c^2\) from the last column. Overall, we get that
  \begin{align*}
    \det(J) &= \frac{a^2}{c^5}
    \begin{vmatrix}
      \sinh(c) & \cosh(c)-1 & c\cosh(c)-\sinh(c)\\
      \cosh(c)-1 & \sinh(c) & c\sinh(c) -(\cosh(c)-1)\\
      \sinh(c) -c & 0 & \dfrac{c^2(\cosh(c)-1) -2c(\sinh(c)-c)}{2c}
    \end{vmatrix}\\
    &= \dfrac{a^2}{c^4}
    \begin{vmatrix}
      \sinh(c) & \cosh(c)-1 & 1\\
      \cosh(c)-1 & \sinh(c) & 0 \\
      \sinh(c) -c & 0 & \dfrac{\cosh(c)-1}{2}
    \end{vmatrix},
  \end{align*}
  where the second equality is obtained with elementary algebraic simplifications.
  We now expand the determinant using Laplace's formula on the last row:
  \[
    \begin{split}
      \det(J) &= \frac{a^2}{c^4} \left[ (\sinh(c)-c)(-\sinh(c)) +\frac{\cosh(c)-1}{2}(-2+2\cosh(c)) \right]\\
      &= \frac{a^2}{c^4} (c\sinh(c) -2(\cosh(c)-1)).
    \end{split}
  \]
  Finally, we
  use the identities
  \[
    \sinh^2(c/2) = \frac{\cosh(c)-1}{2}, \quad \sinh(c) = 2\sinh(c/2)\cosh(c/2),
  \]
  to rewrite the previous expression as
  \[
    \det(J) = \frac{4a^2}{c^4}\sinh(c/2) (c/2\cosh(c/2) -\sinh(c/2))
  \]
  Retracing the simplification, to obtain the Jacobian determinant of the exponential map we multiply the preceding expression by \(t^3\) and substitute $a = \sqrt{(tu)^{\smash{2}} - (tv)^{\smash{2}}}$ and \(c = tw\), which yields the claim.
\end{proof}

The following lemma concludes the proof of \Cref{prop:volume_growth_estimate} by showing that the volume of unitary diamonds with one vertex being the origin goes to zero as the other vertex gets further from the origin:
\begin{lemma}\label{lemma:volume_limit_is_zero}
  It holds that
  \[
    \begin{split}
      \lim_{w \to +\infty} \frac{e^{2w}}{w^4} &\left[ \dfrac{-e^{-w} + e^{-2w} +we^{-w}}{(e^{-w}-1)^2} -  \left(\dfrac{-e^{-w} + e^{-2w} +we^{-w}}{(e^{-w}-1)^2}\right)^2 \right.\\
        & + \left(\dfrac{-e^{-w} + e^{-2w} +we^{-w}}{(e^{-w}-1)^2}\right)^2
        \ln \left( \dfrac{-e^{-w} + e^{-2w} +we^{-w}}{(e^{-w}-1)^2}\right)\\
        & \left.
        +\left(1 + \frac{e^{-w} -e^{-2w}-we^{-w}}{(e^{-w}-1)^2}\right)^2 \cdot  \ln \left( 1 + \frac{e^{-w} -e^{-2w}-we^{-w}}{(e^{-w}-1)^2}\right)
      \right]=0
    \end{split}
  \]
\end{lemma}
\begin{proof}
  Distributing the outer \(e^{2w}/w^4\) yields
  \[
    \begin{split}
      \lim_{w \to +\infty}  & \dfrac{-e^{w} + 1 +we^{w}}{w^4(e^{-w}-1)^2} -  \left(\dfrac{-1 + e^{-w} +w}{w^2(e^{-w}-1)^2}\right)^2 \\
      & + \left(\dfrac{-1 + e^{-w} +w}{w^2(e^{-w}-1)^2}\right)^2
      \ln \left( \dfrac{-e^{-w} + e^{-2w} +we^{-w}}{(e^{-w}-1)^2}\right)\\
      &
      +\left(\frac{e^w}{w^2} + \frac{1 -e^{-w}-w}{w^2(e^{-w}-1)^2}\right)^2 \cdot  \ln \left( 1 + \frac{e^{-w} -e^{-2w}-we^{-w}}{(e^{-w}-1)^2}\right).
    \end{split}
  \]
  We can already see that the limit of the second term is zero. The third term is also vanishing: indeed we have that the first factor is asymptotic to \(1/w^2\) as \(w \to +\infty\), so we may compute the limit as
  \[
    \begin{split}
      \lim_{w \to + \infty}& \dfrac{1}{w^2}
      \ln \left( \dfrac{-e^{-w} + e^{-2w} +we^{-w}}{(e^{-w}-1)^2}\right)\\
      =\lim_{w \to + \infty}& \frac{1}{w^2}\ln\left( we^{-w} \cdot \frac{-1/w + e^{-w}/w +1}{(e^{-w}-1)} \right)\\
      =\lim_{w \to + \infty}& \frac{1}{w^2} \left[\ln(w) -w + \ln \left( \dfrac{-1/w + e^{-w}/w +1}{(e^{-w}-1)^2}\right)\right].
    \end{split}
  \]
  Distributing the factor \(1/w^2\) it is easy to see that all terms are going to 0 (the argument of the logarithm is approaching 1). Hence we are just left with
  \[
    \begin{split}
      \lim_{w \to +\infty}  & \dfrac{-e^{w} + 1 +we^{w}}{w^4(e^{-w}-1)^2} \\
      &+ \left(\frac{e^w}{w^2} + \frac{1 -e^{-w}-w}{w^2(e^{-w}-1)^2}\right)^2 \cdot  \ln \left( 1 + \frac{e^{-w} -e^{-2w}-we^{-w}}{(e^{-w}-1)^2}\right)\\
      =\lim_{w \to +\infty}  & \dfrac{-e^{w} + 1 +we^{w}}{w^4(e^{-w}-1)^2} +\frac{(-1 +e^w -w)^2}{w^4(e^{-w} -1)^4} \cdot \ln \left( 1 + \frac{e^{-w} -e^{-2w}-we^{-w}}{(e^{-w}-1)^2}\right)
    \end{split}
  \]
  Notice that the logarithm is in the form \(\ln(1+f(w))\) with \(f(w) \to 0\) as \(w \to +\infty\). Moreover \(f(w) \sim -we^{-w}\), so we may use Taylor series to write
  \[
    \lim_{w \to +\infty}  \dfrac{-e^{w} + 1 +we^{w}}{w^4(e^{-w}-1)^2} +\frac{(-1 +e^w -w)^2}{w^4(e^{-w} -1)^4} \cdot \left(\frac{e^{-w} -e^{-2w}-we^{-w}}{(e^{-w}-1)^2} +O(w^2e^{-2w}) \right).
  \]
  If we expand the product and look at the last term we get, we find that
  \[
    \lim_{w \to +\infty} \frac{(-1+e^w -w)^2}{w^4(e^{-w}-1)^4}O(w^2e^{-2w}) = \lim_{w \to +\infty} \frac{(-e^{-w}+1 -we^{-w})^2}{(e^{-w} -1)^4} \frac{O(w^2)}{w^4} = 0,
  \]
  since the first fraction approaches 1 and the second fraction is going to 0. Therefore we just need to compute
  \[
    \lim_{w \to +\infty}  \dfrac{-e^{w} + 1 +we^{w}}{w^4(e^{-w}-1)^2} +\frac{(-1 +e^w -w)^2}{w^4(e^{-w} -1)^4} \cdot \frac{e^{-w} -e^{-2w}-we^{-w}}{(e^{-w}-1)^2}.
  \]
  As all denominators have a common factor of \((e^{-w}-1)^2\), which goes to 1 as \(w \to +\infty\), we can ignore it in our computations and just look at
  \[
    \begin{split}
      \lim_{w \to +\infty}  &\dfrac{-e^{w} + 1 +we^{w}}{w^4} +\frac{(-1 +e^w -w)^2}{w^4(e^{-w} -1)^4} \cdot (e^{-w} -e^{-2w}-we^{-w})\\
      =\lim_{w \to +\infty} &-e^{2w} \frac{e^{-w} -e^{-2w} -we^{-w}}{w^4} +\frac{(-1 +e^w -w)^2}{w^4(e^{-w} -1)^4} \cdot (e^{-w} -e^{-2w}-we^{-w}) \\
      =\lim_{w \to +\infty} & \frac{e^{-w} -e^{-2w} -we^{-w}}{w^4} \left( \frac{(-1+e^w-w)^2 -e^{2w}(e^{-w}-1)^4}{(e^{-w}-1)^4}  \right) \\
    \end{split}
  \]
  We can Taylor expand \((e^{-w}-1)^4\) as \(1 -4e^{-w} +o(e^{-w})\) and collect a factor \(e^{-w}\) from the first fraction, which we then distribute on the second one. We expand the square term as well, so that we get
  \[
    \begin{split}
      \lim_{w \to +\infty} & \frac{1 -e^{-w} -w}{w^4} \cdot \frac{e^{-w} +e^{w}(-4e^{-w} +o(e^{-w}) ) +w^2e^{-w} -2 +2we^{-w} -2w}{1 -4e^{-w} +o(e^{-w})}
      \\ \lim_{w \to +\infty} & \frac{1 -e^{-w} -w}{w^2} \cdot \frac{e^{-w} -4 +o(1)  +w^2e^{-w} -2 +2we^{-w} -2w}{w^2(1 -4e^{-w} +o(e^{-w}))},
    \end{split}
  \]
  and we finally see that both fractions are going to 0.
\end{proof}

The next lemma is used in the proof of \Cref{prop:volume_growth_estimate} to explicitly compute the integral in \eqref{eq:integral_with_max_min}:
\begin{lemma}\label{lemma:integral_computations}
  The integral featured in \eqref{eq:integral_with_max_min} is equal to
  \[
    \mathcal{L}^3(J(\mathbf{0},q)) =-\frac{a^4}{8} \left( mM + M^2 \ln(M) + m^2\ln(m) \right),
  \]
  where we set
  \[
    m \coloneqq \frac{1}{2} \left( 1 + \frac{4c}{a^2} \right), \ M \coloneqq \frac{1}{2} \left( 1 - \frac{4c}{a^2} \right).
  \]
\end{lemma}
\begin{proof}
  We apply the change of variables defined by
  \[
    \varphi \colon [0,1] \times [0,1] \to Q, \ \varphi(s,t) \coloneqq \left( \frac{a}{2}(s+t), \frac{a}{2}(t-s) \right).
  \]
  Such map has Jacobian determinant equal to \(a^2/2\), hence after applying the change of variable to the integral in \eqref{eq:integral_with_max_min} we get
  \[
    \begin{split}
      \mathcal{L}^3(J(0,q)) &= \frac{a^2}{2}\int_{[0,1]^2} \max \left[ 0, \min \left( \frac{a^2}{4}st, c + \frac{a^2}{4} - \frac{a^2}{2}s + \frac{a^2}{4}st \right) \right. \\
        & \left. - \max \left( -\frac{a^2}{4}st, c - \frac{a^2}{4} + \frac{a^2}{2}t - \frac{a^2}{4}st\right)
      \right] \, d\mathcal{L}^2(s,t).
    \end{split}
  \]
  Notice that the integrand is of the form \(\max(0, \min(\alpha,\beta) - \max(\gamma,\delta))\), which is non zero if and only if \(\min(\alpha, \beta) > \max(\gamma,\delta)\). This last condition is equivalent to saying that
  \[
    \alpha > \gamma \ \land \ \alpha > \delta \ \land \ \beta>\gamma \ \land \ \beta>\delta.
  \]
  By writing explicitly these conditions we find that the only non-trivial ones we get are
  \[
    t \leq \frac{1}{1-s} \cdot \frac{1}{2}\left( 1 - \frac{4c}{a^2} \right) \ \land \ t \geq 1 - \frac{1}{s}\cdot \frac{1}{2} \left( 1 + \frac{4c}{a^2} \right)
  \]
  We denote by \(m \coloneqq 1/2(1 +4c/a^2)\) and \(M \coloneqq 1/2(1 - 4c/a^2)\) and notice that \(m,M \in (0,1)\) as the point \((a,0,c)\) was in \(I^+(0)\); moreover \(m = 1-M\). The previous condition can then be rewritten as
  \[
    1-\frac{m}{s} \leq t \leq \frac{M}{1-s},
  \]
  hence we can shrink the domain of integration to the set
  \[
    D \coloneqq \left\{ (s,t) \in [0,1]^2 \ \colon \ 1-\frac{m}{s} \leq t \leq \frac{M}{1-s} \right\}
  \]
  and write
  \[
    \begin{split}
      \mathcal{L}^3(J(\mathbf{0},q)) &= \frac{a^2}{2}\int_D \min \left( \frac{a^2}{4}st, c + \frac{a^2}{4} - \frac{a^2}{2}s + \frac{a^2}{4}st \right)
      \\ &- \max \left( -\frac{a^2}{4}st, c - \frac{a^2}{4} + \frac{a^2}{2}t - \frac{a^2}{4}st\right) \, d\mathcal{L}^2(s,t).
    \end{split}
  \]
  We now have to discuss each case separately; namely, we have to find on which subsets of \(D\) we can write the \(\min\) term and the \(\max\) term as one of their arguments. We have that
  \[
    \frac{a^2}{4}st \leq c + \frac{a^2}{4} - \frac{a^2}{2}s + \frac{a^2}{4}st \implies s \leq m, \quad \text{ and } \quad
    -\frac{a^2}{4}st \geq c - \frac{a^2}{4} + \frac{a^2}{2}t - \frac{a^2}{4}st \implies t \leq M.
  \]
  One can then split the domain of integration into four pieces:
  \[
    \begin{split}
      D_1 &\coloneqq \{(s,t) \in D \ \colon \ s \leq m, t \leq M \}, \quad D_2 \coloneqq \{(s,t) \in D \ \colon \ s \geq m, t \leq M \},\\
      D_3 &\coloneqq \{(s,t) \in D \ \colon \ s \leq m, t \geq M \}, \quad D_4 \coloneqq \{(s,t) \in D \ \colon \ s \geq m, t \geq M \},\\
    \end{split}
  \]
  so that we get
  \[
    \begin{split}
      \mathcal{L}^3(J(\mathbf{0},q)) &= \frac{a^2}{2}\int_{D_1} \frac{a^2}{2}st \, d\mathcal{L}^2(s,t)  + \frac{a^2}{2}\int_{D_4} \, \frac{a^2}{2} - \frac{a^2}{2}s - \frac{a^2}{2}t +\frac{a^2}{2}st \,d\mathcal{L}^2(s,t) \\
      &+ \frac{a^2}{2}\int_{D_2} \frac{a^2}{2}m - \frac{a^2}{2}s + \frac{a^2}{2}st \, d\mathcal{L}^2(s,t)+ \frac{a^2}{2} \int_{D_3}\frac{a^2}{2}M - \frac{a^2}{2}t + \frac{a^2}{2}st \, d\mathcal{L}^2(s,t).
    \end{split}
  \]
  It is straightforward to see that \(D_1\) and \(D_4\) are rectangles, while \(D_2\) and \(D_3\) are, up to translations, the sub- and epigraphs of the functions defining \(D\) on some interval. This structure allows for a direct application of Fubini-Tonelli's Theorem, which immediately yields the result.
\end{proof}
\begin{spacing}{0.85}
  \printbibliography
\end{spacing}

\end{document}